\documentclass[10pt]{article}

\usepackage[english,activeacute]{babel}
\usepackage[utf8]{inputenc}  
\usepackage[T1]{fontenc}

\usepackage{epsfig}
\usepackage{stmaryrd}
\usepackage{amsmath,amsfonts,amssymb,mathrsfs,amsthm}
\usepackage{enumitem}
\usepackage{xcolor}
\usepackage{dsfont}
\usepackage{graphicx}
\usepackage[mathscr]{eucal}
\usepackage{makeidx}
\usepackage{verbatim}
\usepackage{graphics,graphicx}
\usepackage{textcomp}
\usepackage{float}
\usepackage{xspace}
\usepackage{mathrsfs}
\usepackage{amscd}
\usepackage{bibentry}

\usepackage{txfonts}
\usepackage{mathrsfs}

\usepackage[colorlinks=true, linkcolor=blue, citecolor=blue]{hyperref}

\newtheorem{lemma}{Lemma}[section]
\newtheorem{theorem}[lemma]{Theorem}
\newtheorem{proposition}[lemma]{Proposition}
\newtheorem{corollary}[lemma]{Corollary}

\theoremstyle{definition}

\newtheorem{definition}[lemma]{Definition}
\newtheorem{remark}[lemma]{Remark}

\makeatletter
\def\keywords{
    \vspace{1ex}
    \noindent
    \if@twocolumn
      \small{\bf  Keywords}\/---$\!$    \else
      \begin{center}\small\ {\bf Keywords}\end{center}\quotation\small
    \fi}
\def\endkeywords{\vspace{0.6em}\par\if@twocolumn\else\endquotation\fi
    \normalsize\rm}
\makeatother

\renewcommand{\O}{\ensuremath{\mathcal O}}
\renewcommand{\H}{\ensuremath{\mathcal H}}

\renewcommand{\P}{\ensuremath{\mathcal P}}
\renewcommand{\L}{\ensuremath{\mathcal L}}

\newcommand{\calS}{\ensuremath{\mathcal S}}

\newcommand{\scrB}{\ensuremath{\mathscr B}}
\newcommand{\scrC}{\ensuremath{\mathscr C}}

\newcommand{\scrH}{\ensuremath{\mathscr H}}
\newcommand{\scrP}{\ensuremath{\mathscr P}}

\newcommand{\bA}{\ensuremath{\bold A}}

\DeclareMathOperator{\Sp}{Sp}
\DeclareMathOperator{\loc}{loc}

\DeclareMathOperator{\Vol}{Vol}

\DeclareMathOperator{\Int}{Int}

\DeclareMathOperator{\comp}{comp}
\DeclareMathOperator{\Diag}{Diag}

\newcommand{\mb}[1]{\ensuremath{\mathbb{#1}}}
\newcommand{\N}{{\mb{N}}}

\newcommand{\R}{{\mb{R}}}
\newcommand{\C}{{\mb{C}}}
\newcommand{\bS}{{\mb{S}}}

\newcommand{\eps}{\varepsilon}

\newcommand{\M}{\ensuremath{M}} 
\newcommand{\T}{\ensuremath{\mathbb T}}
\newcommand{\D}{\ensuremath{\mathscr D}}
\newcommand{\B}{\ensuremath{\mathscr B}}
\newcommand{\A}{\ensuremath{\mathcal A}}

\let \Re \relax
\DeclareMathOperator{\Re}{Re}
\let \Im \relax
\DeclareMathOperator{\Im}{Im}

\newcommand{\ovl}[1]{\overline{#1}}

\DeclareMathOperator{\supp}{supp}

\DeclareMathOperator{\diag}{diag}
\DeclareMathOperator{\dist}{dist}

\DeclareMathOperator{\off}{off}
\DeclareMathOperator{\hw}{hw}
\DeclareMathOperator{\cut}{cut}

\DeclareMathOperator{\vect}{span}

\newcommand{\Delg}{\ensuremath{\Delta}} 

\DeclareMathOperator{\Op}{Op}

\DeclareMathOperator{\id}{Id}

\newcommand{\transp}{\ensuremath{\phantom{}^{t}}}

\renewcommand{\d}{\ensuremath{\partial}}

\newcommand{\nhd}{neighborhood\xspace}

\newcommand{\ie}{i.e.\@\xspace}
\newcommand{\eg}{e.g.\@\xspace}

\let \ae \relax
\newcommand{\ae}{a.e.\@\xspace}

\def\Ut{\tilde{U}}

\newcommand{\F}{\mathscr F}

\newcommand\bna{\begin{eqnarray*}}
\newcommand\ena{\end{eqnarray*}}

\newcommand\bnan{\begin{eqnarray}}
\newcommand\enan{\end{eqnarray}}

\newcommand\bnp{\begin{proof}}
\newcommand\enp{\end{proof}}

\newcommand\nor[2]{\left\|#1\right\|_{#2}}

\numberwithin{equation}{section}

\begin{document}

\title{Long time energy averages and a lower resolvent estimate for damped waves}

\author{Matthieu L\'eautaud\footnote{Laboratoire de Math\'ematiques d'Orsay, Universit\'e Paris-Saclay, CNRS, B\^atiment 307, 91405 Orsay Cedex France, email: matthieu.leautaud@universite-paris-saclay.fr.}} 


\maketitle

\begin{abstract}
We consider the damped wave equation on a compact manifold. We propose different ways of measuring decay of the energy (time averages of lower energy levels, decay for frequency localized data...) and exhibit links with resolvent estimates on the imaginary axis. 
As an application we prove a universal logarithmic lower resolvent bound on the imaginary axis for the damped wave operator when the Geometric Control Condition (GCC) is not satisfied. This is to be compared to the uniform boundedness of the resolvent on that set when GCC holds. 
The proofs rely on $(i)$ various (re-)formulations of the damped wave equation as a conservative hyperbolic part perturbed by a lower order damping term;
$(ii)$ a ``Plancherel-in-time'' argument as in classical proofs of the Gearhart-Huang-Pr\"uss theorem~\cite{Gearhart:78,Huang:85,Pruss:84,EN:book,HS:10} or in~\cite{BZ:04}; 
and $(iii)$ an idea of Bony-Burq-Ramond~\cite{BBR:10} of propagating a coherent state along an undamped trajectory up to Ehrenfest time. 
\end{abstract}

\begin{keywords}
Wave equations, linear hyperbolic systems, resolvent estimates, decay rates, semiclassical analysis, non-selfadjoint operators.

\medskip
\textbf{2010 Mathematics Subject Classification:}
35L05, 
 35L40, 
 35B35, 
  35B40, 
 35S30. 
 81Q20 
\end{keywords}
\tableofcontents

\section{Introduction and main results}

\subsection{Uniform stability of damped waves and resolvent estimates}

We consider a smooth connected compact Riemannian manifold $(\M,g)$ of dimension $n$ (without boundary), and denote by $\Delg$ the associated (nonpositive) Laplace-Beltrami operator.
Given a nonnegative continuous function $b \in C^0(\M;\R_+)$ and a nonnegative number $m \in \R_+$, we are interested in the damped Klein-Gordon/wave equation:
\begin{equation}
\label{eq: stabilization}
\left\{
\begin{array}{ll}
\d_{t}^2 u - \Delg u + m u + b(x) \d_{t} u = 0 & \text{in }\R_+ \times \M , \\
(u, \d_t u)|_{t=0} = (u_0,u_1)                  & \text{in } \M .\\
\end{array}
\right.
\end{equation}
Multiplying~\eqref{eq: stabilization} by $\d_{t} \ovl{u}$ and integrating on $(T_0,T_1)\times M$ yields the following dissipation identity: 
\begin{align}
\label{e:decay-energy-0}
E_m(u(T_1)) - E_m(u(T_0))= -\int_{T_0}^{T_1} \int_\M b|\d_t u|^2 dx dt \leq 0 , \quad \text{for all }T_1 \geq T_0\geq 0 ,
\end{align}
where the energy is defined by
$$
 E_m(u(t)) = \frac12 \left(\|\nabla u(t)\|_{L^2(M;TM)}^2 +m \|u(t)\|_{L^2(M)}^2 +\|\d_t u(t)\|_{L^2(M)}^2 \right) $$
 (and is actually a function of $(u(t),\d_t u(t))$ but we write it as a function of $u(t)$ only for concision).
 Here, $\nabla$ is the Riemannian gradient and $\|\cdot \|_{L^2(M)}$ and $\|\cdot\|_{L^2(M;TM)}$ are defined with respect to the Riemannian volume form.
For the sake of the exposition (in order to avoid an additional technical difficulty), we assume in this paragraph that $m>0$.
From~\eqref{e:decay-energy-0}, we see that the energy $E_m(u(t))$ is a nonincreasing function of time. As soon as $\omega_b:= \{b>0\} \neq \emptyset$, one can further prove that $E_m(u(t))\to 0$ as $t \to + \infty$, (as a consequence of unique continuation properties for eigenfunctions of $\Delg$, see~\cite{Leb:96}), and, in view of possible applications, it is of interest to understand at which rate this decay occurs. 
To this aim, it is customary to reformulate~\eqref{eq: stabilization} equivalently as a first order system 
\begin{equation}
\label{eq: first order eqation}
\left\{
\begin{array}{l}
\d_t U = \A_m U , \\
U|_{t=0} = \transp(u_0 , u_1) ,
\end{array}
\right.
\quad 
U = 
\begin{pmatrix}
u \\
\d_t u
\end{pmatrix} , \quad
\A_m = 
\begin{pmatrix}
0   &  \id \\
\Delg -m& - b
\end{pmatrix}
,\end{equation}
on the energy space $H^1_m(M)\times L^2(M)$, where $H^1_m(M)  = H^1(M)$ is endowed with the norm $\|v\|_{H^1_m(M)}^2 = \|\nabla v\|_{L^2(M)}^2 +m \|v\|_{L^2(M)}^2$ (recall that we assume $m>0$ in this introduction) and $\A_m$ is endowed with appropriate domain $D(\A_m)$. 
Stability properties of the energy of solutions to~\eqref{eq: stabilization} reformulate equivalently in terms of stability properties of the semigroup $\big(e^{t\A_m}\big)_{t\in \R_+}$ in the state space $H^1_m(M)\times L^2(M)$. 

The first key question raised by applications concerns uniform/exponential stability, namely, under which conditions (on $(M,g), b$) there exist $C, \gamma>0$ such that
 \begin{equation}
 \label{e:uniform-decay-wa}
 E_m(u(t)) \leq C e^{-\gamma t} E_m(u(0)), \quad \text{ for all solutions to~\eqref{eq: stabilization}} .
 \end{equation}
According to the celebrated Gearhart-Huang-Pr\"uss theorem~\cite{Gearhart:78,Huang:85,Pruss:84}, and on account to the fact that  $\big(e^{t\A_m}\big)_{t\in \R_+}$ is a contraction semigroup on $H^1_m(M)\times L^2(M)$, uniform/exponential stability~\eqref{e:uniform-decay-wa} is equivalent to the condition
\begin{align}
i\R \subset \rho(\A_m)\quad\text{and} \quad \sup_{s\in \R}\nor{(is -\A_m)^{-1}}{\L(H^1_m\times L^2)} < +\infty , 
\end{align}
where $\rho(\A_m) \subset \C$ denotes the resolvent set of $\A_m$.
In the context of Equation~\eqref{eq: stabilization} the question has been solved by the Rauch-Taylor Theorem~\cite{RT:74} (see also the Bardos-Lebeau-Rauch theorem on manifolds with boundary~\cite{BLR:88,BLR:92}): uniform/exponential stability~\eqref{e:uniform-decay-wa} holds if the Geometric Control Condition (GCC) is satisfied: every geodesic starting from $S^*\M$ enters the set $\omega_b =\{b>0\}$ in finite time. Reciprocally, for a continuous function $b$,~\eqref{e:uniform-decay-wa} implies GCC~\cite{Ralston:69,BG:97,MZ:02}. A corollary of this converse implication together with the Gearhart-Huang-Pr\"uss theorem is that if GCC is {\em not} satisfied, then the resolvent cannot be bounded at $\pm i \infty$ (it is a continuous function on $i\R$), that is $$\limsup_{s\to \pm \infty}\nor{(is -\A_m)^{-1}}{\L(H^1_m\times L^2)} = +\infty .$$
Remarkably enough, none of the above-mentioned references uses the semigroup setting and the Gearhart-Huang-Pr\"uss theorem, but rather directly consider energy estimates for the evolution equation~\eqref{eq: stabilization}, such as~\eqref{e:decay-energy-0}. Hence, the results of~\cite{Ralston:69,BG:97,MZ:02} do not furnish an estimate of the divergence of the resolvent at $\pm i \infty$. Our first result provides with an a priori lower estimate of the resolvent under the condition that $b$ has some regularity.

\begin{theorem}
\label{t:theorem-holder-0}
Assume that $b$ is H\"older continuous on $M$ and that GCC is {\em not} satisfied. Then, we have
$$
\limsup_{s\to \pm \infty}  \frac{1}{\log|s|}\nor{(is -\A_m)^{-1}}{\L(H^1_m\times L^2)}  >0 .
$$
\end{theorem}
A more precise version of this result (with the proper regularity condition on $b$, and a more explicit formulation of the divergence) is provided in Theorem~\ref{t:main-res} below. A logarithmic growth of the resolvent is optimal without further assumptions in view of the logarithmic upper bounds proved in different geometric contexts in~\cite{Christ:07,Christ:10, Christ:11,Riviere:14,CSVW:14,Schenck:10,Schenck:11,Nonnenmacher:11,Jin:17,DJN:20} (see Section~\ref{s:literature} below for a more precise discussion).

\subsection{Non-uniform stability and resolvent estimates}

When the set $\omega_b$ does not satisfy the Geometric Control Condition, (necessarily non-uniform) decay rates for the energy $E_m(u(t))$ and resolvent estimates for $\A_m$ on $i\R$ remain closely related. This was first remarked by Lebeau~\cite{Leb:96}, and more systematically studied by Batty and Duyckaerts~\cite{BD:08}. In this context, they introduced the following definition.

\begin{definition}[Non-uniform stability]
\label{def:non-unif-stab}
Given $a \in \R$ and a decreasing function $f: [a, +\infty) \to \R_+^*$ such that $f(t)\to 0$ as $t \to +\infty$,
we say that the solutions of~\eqref{eq: stabilization} decay at rate $f(t)$ if there exists $C>0$ such that for all $(u_0, u_1) \in D(\A_m) = H^2(\M)\times H^1(\M)$ (recall that we have $m>0$ in the introduction), for all $t\geq a$, we have
$$
E_m(u(t))^{\frac12} \leq C f(t) \nor{\A_m (u_0,u_1)}{H^1_m\times L^2} .
$$
\end{definition}
Note that this is equivalent to having $\nor{e^{t\A_m}\A_m^{-1}}{\L(H^1_m\times L^2)} \leq  C f(t)$.
As noticed in~\cite{Leb:96,BD:08}, decay at a rate $f(t)$ implies faster decay (like $f(t/k)^k$) for ``smoother'' data (in the sense $(u_0, u_1) \in D(\A_m^k)$). Resolvent estimates and decay rates for the energy are linked via the following general theorem.
\begin{theorem}[Batty and Duyckaerts~\cite{BD:08}]
\label{t:batty-duyckaerts}
Let $\mathsf{M} : \R\to \R_+^*$ be an even continuous function that is nondecreasing on $\R_+$ such that 
\begin{align}
\label{e:res-bounded-M}
i\R \subset \rho(\A_m)\quad \text{and} \quad \nor{(is-\A_m)^{-1}}{\L(H^1_m \times L^2)}  \leq \mathsf{M}(s) , \quad s \in \R .
\end{align}
Then, setting $\mathsf{M}_{\log}(s) =  \mathsf{M}(s) \big( \log(1+ \mathsf{M}(s)) + \log(1+ s)\big)$ for $s\geq0$, there exists $c>0$ such that solutions of~\eqref{eq: stabilization} decay at rate $f(t) = \frac{1}{\mathsf{M}_{\log}^{-1}(t/c)}$, where $\mathsf{M}_{\log}^{-1} : \R_+\to \R_+$ denotes the inverse of the strictly increasing function $\mathsf{M}_{\log}$.
\end{theorem}
Alternative proofs of this result are provided in~\cite{Duy:15,CS:16}. 
For instance, if $\mathsf{M}(s) = Ce^{c|s|}$ then Theorem~\ref{t:batty-duyckaerts} implies decay at rate $f(t) = \frac{1}{\log(t)}$, if $\mathsf{M}(s) = C(1+|s|)^{\alpha}$, then $f(t) = \frac{\log(t)}{t^\frac{1}{\alpha}}$, and if $\mathsf{M}(s) = C\log(2+|s|)$, then $f(t) = e^{-\gamma \sqrt{t}}$.  The literature concerning known resolvent estimates (and associated decay results) for damped waves is reviewed in Section~\ref{s:literature} below.
In the case of polynomial decay, the $\log(t)$ loss was removed by Borichev and Tomilov~\cite{BT:10} who proved decay at rate $f(t) = t^{-\frac{1}{\alpha}}$ if $\mathsf{M}(s) = C(1+|s|)^{\alpha}$ (using the Hilbert-space structure, as opposed to the results of~\cite{BD:08,CS:16}).  Generalizations of~\cite{BT:10} to functions $\mathsf{M}$ of ``positive increase'' were provided successively in~\cite{BCT:16,BCT:17}. See also the review articles~\cite{CPSST:19,CST:20}.
\begin{remark} Note that a classical Neumann series argument shows that~\eqref{e:res-bounded-M} implies that 
\begin{equation}
\label{e:spectral-gap}
\left\{ z \in \C ,\Re(z) > - \frac{1}{\mathsf{M}(\Im(z))}  \right\} \subset \rho(\A_m) ,
\end{equation}
which is a (non-uniform if $\mathsf{M}(s)\to +\infty$ as $s\to \pm\infty$) spectral gap.
\end{remark}

\subsection{Time-averages of the energy}
Our strategy for proving Theorem~\ref{t:theorem-holder-0} is inspired by that of Bony, Burq and Ramond~\cite{BBR:10} concerning the cutoff resolvent of Schr\"odinger operators in $\R^n$. First, we take advantage of the lack of GCC to prove that particular families of solutions (concentrated along a geodesic cuve that does not intersect $\omega_b$) to the wave equation~\eqref{eq: stabilization} conserve their initial energy for a long time. Second, we relate this long time energy conservation to the lower bound of the resolvent. 
Unfortunately, Theorem~\ref{t:batty-duyckaerts} does not seem to be precise enough to fulfill the second part of the proof. 

We propose other ways of measuring energy decay, which seem better adapted to the proof of Theorem~\ref{t:theorem-holder-0}. One instance of such results is the following theorem, in which we assume additional smoothness of $b$ for simplicity. In order to state it, we need to introduce the  function $\mathsf{M}_\eps =\mathsf{M}\left(\frac{\cdot}{1-\eps} \right)$ together with the operators $\Lambda_m = \sqrt{-\Delta + m}$ and $\mathsf{M}_\eps(\Lambda_m)^{-1} = \left(\frac{1}{\mathsf{M}_\eps}\right)(\Lambda_m)$, both defined {\em via} continuous functional calculus for selfadjoint operators.

\begin{theorem}
\label{t:thm-classiq-A-0}
Assume $b\in C^\infty (M;\R_+)$ is positive on a non-empty open set, and $\mathsf{M}  \in C^0(\R;\R_+^*)$ is an even function, nondecreasing on $\R_+$ such that for some $\lambda_0,C>0$ we have
\begin{align*}
\mathsf{M} (4\lambda) \leq C  \mathsf{M} (\lambda),  \quad \text{ and } \quad \mathsf{M} (\lambda) \leq C \lambda^{N}, \quad \text{ for all }\lambda\geq \lambda_0.
\end{align*}
Assume that~\eqref{e:res-bounded-M} holds. Then, for all $\eps \in (0,1)$, there is $C>0$ such that for all $\Psi \in H^1_{\loc}(\R)$ with $\supp(\Psi) \subset [0,+\infty)$ and $\Psi' \in L^2(\R)$, for all $(u_0,u_1)\in H^1_m \times L^2$, we have
\begin{align}
\label{e:Psi-T-classical-quad-0}
 \int_\R \left|\Psi \left( t\right)\right|^2 E_m \big(\mathsf{M}_\eps(\Lambda_m)^{-1}u(t) \big) dt  \leq C \nor{\Psi'}{L^2(\R)}^2 E_m(u(0)) , 
\end{align}
where $u$ is the associated solution to~\eqref{e:KG-eq}.
\end{theorem}
This might be understood as an alternative way of measuring decay of the energy. The main difference with Definition~\ref{def:non-unif-stab} and Theorem~\ref{t:batty-duyckaerts} is that the controlled quantity is not the energy at time $t$, but a time average of the energy. This statement can be interpreted as follows:
\begin{enumerate}
\item if $\Psi (t)= \min\{t,1\}$, then~\eqref{e:Psi-T-classical-quad-0} yields integrability of $E_m \big(\mathsf{M}_\eps(\Lambda_m)^{-1}u(t) \big)$ on $[1,\infty)$.
\item if $T$ is large and $\Psi \in C^\infty_c(0,T+1)$ with $\Psi=1$ on $[1,T]$ and $\|\Psi'\|_{L^2}$ independent of $T$, then~\eqref{e:Psi-T-classical-quad-0} implies that a time average of $E_m \big(\mathsf{M}_\eps(\Lambda_m)^{-1}u(t) \big)$ over a time interval of length $T$ decays like $T^{-1}$.
\end{enumerate}
A second difference is that we do not measure the usual $H^1_m\times L^2$ energy level but rather a ``lower energy level'', depending on the bound on the resolvent we have.
As opposed to Theorem~\ref{t:batty-duyckaerts}, different bounds on the resolvent give rise to the same decay rate of the average  (\eqref{e:Psi-T-classical-quad-0} always implies that $T$-averages decay like $\frac{1}{T}$), but measured in different spaces.
An advantage is that there is no logarithmic loss in the estimate.
A drawback is that it requires some regularity of the damping coefficient $b$.

We also provide with a semiclassical version of this result in Theorem~\ref{t:resolvent-implies-long-time} which focuses on data having a given ``scale of oscillation'' (we also assume that $b$ is smooth for simplicity).

\begin{theorem}
\label{t:semiclassic-intro}
Assume that $b \in C^\infty(M;\R_+)$, set
\begin{equation*}
I_\eps :=  [1-\eps, 1+\eps] , \quad \text{ for }\eps \in (0,1) ,
\end{equation*}
and take $\chi_\eps \in C^\infty_c(\R)$ with $\chi_\eps \geq0$, $\supp\chi_\eps \subset \Int(I_\eps)$.
Then, for any $N,N' \in \N$, any $\psi \in H^1_{\comp}(\R)$ with $\supp \psi \subset \R_+$, and any $\delta \in (0,1)$, there is $C>0$ such that 
for all $(u_0 , u_1) \in \chi_\eps(h \Lambda_m)(H^1_m\times L^2)$ and $u$ associated solution to~\eqref{eq: stabilization}, 
for all $h \in (0,1)$ and all $T_h \leq h^{-N}$
\begin{align}
\label{e:toto-intro}
\frac{1}{T_h}\int_\R \left|\psi\left( \frac{t}{T_h}\right) \right|^2 E_m(u(t)) dt
 \leq  \left(  \frac{G(h)^2}{T_h^2} \nor{\psi'}{L^2(\R)}^2 
+ C h^{N'}\right) E_m(u(0)) , \\
    E_m(u(T_h)) \leq  \left( (2+\delta) \frac{G(h)^2}{T_h^2}     +C h^{N'} \right) E_m(u(0)) ,
\end{align}
where 
\begin{align}
\label{e:asspt-G}
G(h) = \sup_{\tau \in \R , |\tau| \in I_\eps} \nor{(i\tau/h - \A_m)^{-1}}{\L\big(H^1_m \times L^2\big)}  .
\end{align}
\end{theorem}
Note that the operator $\chi_\eps(h \Lambda_m) = \chi_\eps(h \sqrt{-\Delta+m})$ is defined {\em via} functional calculus for selfadjoint operators.
In Theorem~\ref{t:semiclassic-intro}, assuming that the resolvent $(\lambda - \A_m)^{-1}$ is bounded for $ \lambda \in \pm i [(1 -\eps)h^{-1}, (1+\eps)h^{-1}]$ by $G(h)$, we deduce that
\begin{enumerate}
\item energy averages on time intervals of size $T_h$,
\item the pointwise value at time $T_h$
\end{enumerate} 
of solutions having typical frequency $h^{-1}$ (that is, associated with data spectrally localized in this energy window, in terms of the functional calculus for $\Lambda_m$) are estimated by $\frac{G(h)^2}{T_h^{2}}$ for all semiclassically tempered time scale $T_h$.

Note that we use (and prove) that for $b \in C^\infty$ the evolution $e^{t\A_m}$ does not mix frequencies. The same statement also holds for $b \in L^\infty(M;\R_+)$ (see Theorem~\ref{t:resolvent-implies-long-time}) but the remainder is only $O(hT_h)$ instead of $O(h^{N'})$, which reduces the range of times for which the result is nonempty.

Note that $\A_m$ having real coefficients, $\nor{(\overline{z} - \A_m)^{-1}}{\L\big(H^1_m \times L^2\big)}=\nor{(z - \A_m)^{-1}}{\L\big(H^1_m \times L^2\big)}$ and thus $ |\tau| \in I_\eps$  in~\eqref{e:asspt-G} may be equivalently replaced by either $\tau \in I_\eps$ or $\tau \in - I_\eps$.

Note finally that  the semiclassical Theorem~\ref{t:semiclassic-intro} (or one of its variants, Theorem~\ref{t:resolvent-implies-long-time}) is better suited to prove Theorem~\ref{t:theorem-holder-0} than the more classical Theorem~\ref{t:thm-classiq-A-0}.

\begin{remark}
We have formulated all of our results for damped waves and Klein-Gordon equations for simplicity. Most of the results presented here may actually be generalized to more general equations/systems consisting in a conservative/hyperbolic principal part perturbed by a lower order damping term. We have sticked to the paradigmatic example of the damped wave and Klein-Gordon equations for readability and concision.
\end{remark}

\subsection{Known resolvent estimates for damped waves}
\label{s:literature}

It is convenient to introduce the subset of phase-space consisting in points-directions that are never brought into the damping region $\omega_b$ by the geodesic flow. Namely, the {\em undamped set} is defined by
\begin{equation}
\label{e:defsingularS}
S = \{(x,\xi) \in S^*\M, \text{ for all } t \in \R , \ \pi \circ \varphi_t(x,\xi)  \cap \omega_b =\emptyset \} ,
\end{equation}
where $\varphi_t$ is the geodesic flow on $S^*\M$ and $\pi : T^*\M \to \M$ the canonical projection. With this definition, GCC is equivalent to $S=\emptyset$.

To our knowledge, Lebeau~\cite{Leb:96} was the first to consider the question of resolvent estimates and its relation to non-uniform decay rates, \ie go beyond GCC.  He proved that~\eqref{e:res-bounded-M} always holds for $\mathsf{M}(s)=Ce^{\gamma|s|}$ (whence decay at rate $1/\log t$), independently of $(\M,g)$ and $b$ as soon as $\omega_b \neq \emptyset$ (see also~\cite{Burq:98,BD:08}). 
Moreover, he constructed a series of explicit examples of geometries for which this resolvent growth is optimal, including for instance the case where $\M = \bS^2 \subset \R^3$ and $\supp(b)$ does not intersect an equator (see also~\cite[Theorem~1.3]{Hitrik:03}). This result is generalized in~\cite{LR:97} for a wave equation damped on a (small) part of the boundary. 
Lebeau also constructed surfaces of revolution and a damping function $b$ for which GCC is not satisfied (hence uniform decay does not hold) but a uniform spectral gap holds: there is $\eps>0$ such that 
$
\left\{ z \in \C ,\Re(z) > - \eps  \right\} \subset \rho(\A_m) ,
$
(compare with~\eqref{e:spectral-gap}).

\paragraph{``Hyperbolic situations''.}
Christianson~\cite{Christ:07,Christ:10} proved~\eqref{e:res-bounded-M} with $\mathsf{M}(s) = C\log(2+|s|)$ (hence energy decays at rate $f(t) =e^{-\gamma \sqrt{t}}$ for some $\gamma>0$), in the case
where the undamped set $\pi(S)$ is a hyperbolic closed
geodesic (see also~\cite{CSVW:14}). This result was generalized by Christianson~\cite{Christ:11} to the case $\pi(S)$ is a semi-hyperbolic orbit, and by Nonnenmacher and Rivi\`ere~\cite{Riviere:14} and Christianson, Schenck, Vasy and Wunsch~\cite{CSVW:14} to the case where $S$ is a small (in terms of topological pressure) hyperbolic set. These results only assume a hyperbolicity condition near $S$, and Burq and Christianson~\cite{BC:15} proved the optimality of the decay $e^{-C\sqrt{t}}$ under this sole dynamical condition.
With an additional assumption on the {\em global} dynamics, Schenck~\cite{Schenck:10,Schenck:11} proved both that $\mathsf{M}(s) = C\log(2+|s|)$ and a uniform spectral gap (hence energy decay at rate
$e^{-Ct}$, still in the sense of Definition~\ref{def:non-unif-stab}) on manifolds with negative sectional curvature, if the
undamped set is ``small enough'' in terms of topological pressure (for
instance, a small \nhd of a closed geodesic), and if the damping is
``large enough''. This result was later generalized by Nonnenmacher~\cite{Nonnenmacher:11}. Finally on compact {\em surfaces} with negative curvature, Dyatlov-Jin-Nonnenmacher~\cite{Jin:17,DJN:20} recently proved the $e^{-Ct}$ decay for any non-trivial damping function $b$.
In these papers, the geodesic flow (either near the undamped set $S$ or globally) enjoys strong instability properties: it is uniformly hyperbolic, in particular, trajectories are exponentially unstable.

\paragraph{``Integrable situations''.}
On the other hand, Liu and Rao considered in \cite{LR:05} the case where $\M$ is a square and the set $\{b >0\}$ contains a vertical strip. In this situation, the undamped trajectories consist in a family of parallel vertical geodesics; these are unstable, in the sense that nearby geodesics diverge at a linear rate. They proved~\eqref{e:res-bounded-M} with $\mathsf{M}(s) = C\langle s\rangle^2$ (\ie solutions of~\eqref{eq: stabilization} decay at rate $t^{-\frac12}$). This was extended by Burq and Hitrik~\cite{BH:07}  to the case of partially rectangular two-dimensional domains, if the set $\{b >0\}$ contains a \nhd of the non-rectangular part and in~\cite{AL:14} on the torus. See also~\cite{Phung:07b,Phung:08} where Phung proved $\mathsf{M}(s) = C\langle s\rangle^{N}$ for some $N>0$ in a three-dimensional domain having two parallel faces, and~\cite{ALM:16cras,ALM:16} where the bound $\mathsf{M}(s) = C(1+s^2)$ is proved in the disk if $\omega_b$ intersects the boundary. Assuming additional smoothness assumptions on $b$,  this was improved to $\mathsf{M}(s) = C\langle s\rangle^{1+\eps}$ in~\cite{BH:07,AL:14}. In all these situations, the only obstruction to GCC is due to ``cylinders of periodic orbits''. It was also proved in~\cite{AL:14} that~\eqref{e:res-bounded-M} implies $\limsup_{s\to\pm \infty}\frac{\mathsf{M}(s)}{|s|}>0$ as soon as the ``cylinder of periodic orbits'' is nonempty. 
These resolvent bounds on the torus were further refined in the case of H\"older regular damping in~\cite{Stahn:17,Kleinhenz:18,Sun:21}. Lower resolvent bounds in the context of nonselfadjoint operators were also proved in~\cite{Arnaiz:22}.

\paragraph{``Locally undamped situations''.}
Related to the above-discussed cases, the following typical situation was studied in~\cite{LeLe:14,LeLe:17}: $M=\T^2$, $\gamma$ is a periodic geodesic, and $b(x) = \dist_g(x,\gamma)^{2\alpha}$, so that the undamped set is exactly $\pi(S)=\gamma$ (as opposed to the ``cylinder of periodic orbits'' case). In such a situation, it was proved that~\eqref{e:res-bounded-M} holds with $\mathsf{M}(s) = C\langle s \rangle^{\frac{\alpha}{\alpha+1}}$ (thus improving on the $C\langle s\rangle^{1+\eps}$ bound) and that this is optimal, \ie~\eqref{e:res-bounded-M} implies  $\limsup_{s\to\pm \infty}\frac{\mathsf{M}(s)}{|s|^{\frac{\alpha}{\alpha+1}}}>0$.
Later on, Burq-Zuily~\cite{BZ:15} weakened some of the assumptions of this result. They also investigated in~\cite{BZ:16} a similar situation in case the manifold has no product structure: this covers \eg the case of the sphere $\mathbb{S}^2\subset\R^3$ and $b(x)=\dist_g(x,\gamma)^{2\alpha}$, where $\gamma$ is an equator and proved that~\eqref{e:res-bounded-M} holds with $\mathsf{M}(s) = C\langle s \rangle^{\alpha}$ (thus improving on the $Ce^{c|s|}$ a priori bound). In the situation where $\gamma$ is a hyperbolic closed geodesic curve on a compact manifold and $b(x)=\dist_g(x,\gamma)^{2\alpha}$, no precise result was known so far. Theorem~\ref{t:theorem-holder-0} shows that~\eqref{e:res-bounded-M} implies $\limsup_{s\to\pm \infty}\frac{\mathsf{M}(s)}{\log|s|}>0$ whatever $\alpha$. Hence, one cannot hope to improve on the logarithmic bound of~\cite{Christ:07,Christ:10, Christ:11,Riviere:14,CSVW:14,Schenck:10,Schenck:11,Nonnenmacher:11,Jin:17,DJN:20} by enhancing the damping up to $\gamma$, \ie taking lower values of $\alpha$.

\subsection{Plan of the paper, ideas of proofs}

Section~\ref{s:first-order-system} is a preliminary section in which we reformulate  the wave/Klein-Gordon equations~\eqref{eq: first order eqation} in more customary forms. Namely, in many steps of the proofs, we exhibit and exploit the structure of the Klein-Gordon/wave equation as a conservative-propagative equation perturbed by a lower-order dissipation term.
 We first reformulate~\eqref{eq: first order eqation} under the form of a hyperbolic system, \ie
\begin{align}
\label{e:evol-hyp}
D_tU = \P U , \quad D_t = \frac{\d_t}{i} ,
\end{align}
on $L^2(M;\C^d)$ (here with $d=2$) or a subspace of $L^2(M;\C^d)$, and where $\P$ has the form 
\begin{align}
\label{e:oper-hyp}
\P =D\Lambda_m+iQ,\quad D \text{ diagonal with real entries, }\Lambda_m=\sqrt{-\Delta+m}, \quad Q \text{ bounded with } \Re(QU,U) \geq 0 .
\end{align}
This reduction is performed in an algebraic way in Section~\ref{s:KG-reformulation} for the Klein-Gordon equation and in Section~\ref{s:KG-waves} for the wave equation. The semigroups $(e^{t\A_m})_{t\in\R_+}$ and $(e^{it\P})_{t\in\R_+}$ are compared, as well as the resolvents $(z-\A_m)^{-1}$ and  $(z-\P)^{-1}$.
Under the form~\eqref{e:evol-hyp}-\eqref{e:oper-hyp}, the evolution is a first order conservative (and decoupled) system $D_t - D\Lambda_m$, perturbed by a zero order dissipative term $iQ$, that couples the equations.
This algebraic reduction is useful in connection with the abstract/functional resolvent approach performed in Sections~\ref{s:semiclassical-resolvent} and~\ref{s:abstract-section}.

Then, in Section~\ref{s:semiclass-hyperbolic}, we explain how~\eqref{e:oper-hyp} can be semiclassically diagonalized if all eigenvalues of $D$ are different (which is the case for the wave/Klein-Gordon equations~\eqref{eq: first order eqation} for which the eigenvalues of $D$ are $\pm1$), that is to say, how the system decouples at high frequencies. More precisely, we prove that for data oscillating at frequency $1/h$, $h\to 0^+$, the solution of~\eqref{e:evol-hyp}-\eqref{e:oper-hyp} is at time $t$ essentially $O(ht)$ close to that of the associated diagonal system (\ie simply putting $0$ in the off-diagonal terms of $Q$). This will in particular be useful in Section~\ref{s:proof-main} in order to describe the long time (depending on $h$) propagation of coherent states.

\medskip
Section~\ref{s:semiclassical-resolvent} is devoted to the proofs of two theorems like Theorem~\ref{t:semiclassic-intro}. In Section~\ref{s:bounded-damping-general}, an abstract/functional analysis theorem like Theorem~\ref{t:semiclassic-intro} is proved for semigroups of the form $(e^{\frac{it}{h}\P_h})_{t\in \R}$ with $\P_h = P _h+ i h Q_h$ on a Hilbert space $\H$, assuming only $P_h$ selfadjoint and $Q_h$ bounded with nonnegative real part. In this setting, we prove an analogue of Theorem~\ref{t:semiclassic-intro} with a remainder of the form $O(hT_h)$ instead of $O(h^{N'}T_h)$. The proof relies on Fourier transform in time of the equation and the use of the Plancherel theorem. It is inspired by the proofs of~\cite[Theorem~4]{BZ:04} (concerning selfadjoint operators) and of the Gearhart-Huang-Pr\"uss theorem~\cite{Gearhart:78,Huang:85,Pruss:84,EN:book,HS:10}. This version of the theorem only assumes boundedness of the ``damping operator'' and is thus well-suited to prove a result like Theorem~\ref{t:theorem-holder-0} (which only requires a little regularity of the damping coefficient).
Then, Section~\ref{s:remainder-regularity} is devoted to an improvement of the $O(hT_h)$ remainder to an $O(h^{N'}T_h)$, still in the abstract setting, but assuming the existence of an improved parametrix for $(\P_h-\tau)^N$ for $\tau$ away from $\supp(\chi_\eps)$.
The latter is constructed in Section~\ref{s:parametrice} for the wave/Klein-Gordon operators using pseudodifferential calculus and assuming $b$ smooth. This eventually leads to a proof of Theorem~\ref{t:semiclassic-intro} in Section~\ref{s:proof-thm-smooth-h}.

\medskip
Section~\ref{s:egorov} is devoted to the proof of an Egorov theorem for a non selfadjoint operator in Ehrenfest time $T_h = \mu \log(1/h)$ where $\mu$ is a classical dynamical quantity. Here the group involved is $(e^{itP})_{t\in \R}$, where $P= \pm \Lambda_m + iQ$ is a scalar operator, and $Q$ is a semiclassical operator of order zero, the real part of which being nonnegative at the principal level (one can take $Q=b$ if $b$ is smooth).
A technicality comes from the fact that, for the purposes of Theorem~\ref{t:theorem-holder-0}, we do not want to assume that $b$ is smooth. To overcome this issue, we construct (in Appendix~\ref{e:reg-b}) an $h-$dependent regularization of $b$ called $b_{h^\nu}$.
To include such an operator in the Egorov Theorem, we assume that the operator $Q$ in Section~\ref{s:egorov} belongs to a mildly exotic class of operators.
The proof of the Egorov theorem then follows that of~\cite{AnNon,DG:14,DJN:20}. We collect in Section~\ref{sub:classical-estimates} a number of classical/symbolic estimates, checking that the propagated symbol remains bounded in a reasonable mildly exotic class for times of order $T_h = \mu \log(1/h)$. Then, the Egorov theorem is proved in Section~\ref{sub:Egorov}.

\medskip
Section~\ref{s:proof-main} is then devoted to the proof of Theorem~\ref{t:theorem-holder-0} in a more precise form (Theorem~\ref{t:main-res} and Corollary~\ref{c:ralalalala}), and is inspired by~\cite{BBR:10} (see also~\cite{Wang:91}), which concerns the cutoff resolvent of Schr\"odinger operators in $\R^n$. It relies first on the results of Section~\ref{s:semiclassical-resolvent} which relate the resolvent  and the evolution for the wave/Klein-Gordon equation, together with those of Sections~\ref{s:KG-reformulation} and~\ref{s:KG-waves} to formulate the latter in the ``hyperbolic system'' form~\eqref{e:evol-hyp}. Then we replace $b$ by its $h-$dependent regularization $b_{h^\nu}$. The next step consists in replacing the operator in~\eqref{e:evol-hyp} by its diagonal part using the results of Section~\ref{s:semiclass-hyperbolic}. 
We finally use the Egorov theorem of Section~\ref{s:egorov} together with a coherent state localized at a point $\rho_0 \in S$, \ie in the undamped set (see~\eqref{e:defsingularS}). That $S \neq \emptyset$ is equivalent to the assumption of Theorem~\ref{t:theorem-holder-0} that GCC is not satisfied. In an inequality like~\eqref{e:toto-intro}, this implies that $E_m(u(T_h)) \sim 1 = E_m(u(0))$ for $T_h = \mu \log(1/h)$, whence $G(h)\geq 2^{-1/2} T_h$ and the sought lower resolvent bound.

\medskip
Section~\ref{s:abstract-section} is devoted to the proofs of non-semiclassical results like Theorem~\ref{t:thm-classiq-A-0}.
As a preliminary, Section~\ref{s:modified-resolvent} reformulates resolvent estimates like~\eqref{e:res-bounded-M} in an abstract setting, in different forms, under a ``compatibility condition'' between the selfadjoint, the skewadjoint parts of the operator and the function $\mathsf{M}$. Section~\ref{s:from-res-to-integrated} then proves Theorem~\ref{t:thm-classiq-A-0}, following classical proofs of the Gearhart-Huang-Pr\"uss-Greiner Theorem \eg~\cite[Chapter~V.1]{EN:book} or~\cite{HS:10} (see also the related~\cite{BZ:04},~\cite[Theorem~2.3]{Miller:12} or~\cite[Theorem~4.7]{BCT:16}). The abstract ``compatibility condition'' between the selfadjoint, the skewadjoint parts of the operator and $\mathsf{M}$ is rephrased in Section~\ref{s:Mand-asspt-2} as a temperence condition on $\mathsf{M}$ and the fact that the skewadjoint part of the operator does not mix the frequencies of the selfadjoint part of the operator at high frequencies. The latter condition is finally checked for the wave/Klein-Gordon equations in Section~\ref{s:app-damped-semiclass} using semiclassical analysis, and this concludes the proof of results like Theorem~\ref{t:thm-classiq-A-0}.

\medskip 
The paper ends with two appendices. In Appendix~\ref{appendix}, we collect various facts and notations of semiclassical pseudodifferential calculus that are used in the main part of the paper. We also construct in Section~\ref{e:reg-b} the regularization procedure needed for the above-discussed proofs in case $b$ is only H\"older continuous (or has a logarithmic modulus of continuity).
Finally, Appendix~\ref{appendix-elementary} collects several elementary technical lemmata used throughout the proofs.

\bigskip
\noindent
{\em Acknowledgements.} 
We would like to thank Camille Laurent for early discussions on Egorov theorems, Gabriel Rivi\`ere for later discussions on Egorov theorems, and Benjamin Delarue for a discussion on~\cite{Kuster:17}.
Most of this project was carried on when the author was in CRM in Montr\'eal and he wants to acknowledge this institution for its kind hospitality. 
The author is partially supported by the Agence Nationale de la Recherche under grants SALVE (ANR-19-CE40-0004) and ADYCT (ANR-20-CE40-0017).

\section{First order hyperbolic systems}
\label{s:first-order-system}

In this section, we first reformulate wave and Klein-Gordon equations~\eqref{eq: stabilization} as first order hyperbolic systems. This reformulation consists essentially of algebraic manipulations; it is slightly simpler for Klein-Gordon equations ($m>0$) since in this case $-\Delta+m$ is injective. We thus explain it first in this setting in Section~\ref{s:KG-reformulation}, and then proceed to the case $m=0$ in Section~\ref{s:KG-waves}. 
Finally, in Section~\ref{s:semiclass-hyperbolic}, we proceed with a semiclassical (approximate) diagonalization of hyperbolic systems of the form obtained in Sections~\ref{s:KG-reformulation} and~\ref{s:KG-waves}. The latter will be useful for describing propagation of singularities.

\bigskip
We denote by $(u,v)_{L^2(M)} = \int_M u(x) \ovl{v(x)}d\Vol_g(x)$ the $L^2$ inner product on $M$, where $d\Vol_g$ denotes the Riemannian volume element for the metric $g$.
We also recall that $\nabla$ denotes the Riemannian gradient and, for two (possibly complex-valued) vector-fields $X,Y$ on $M$ (in particular $X=\nabla u$) use the notation 
$$
(X,Y)_{L^2(M)}=(X,Y)_{L^2(M;TM)} = \int_M g_x\big(X(x) ,\ovl{Y(x)}\big)d\Vol_g(x), 
$$
and associated norm $\|X\|_{L^2(M)}=\|X\|_{L^2(M;TM)} =\sqrt{(X,X)_{L^2(M)}}$.
We shall make use of a Hilbert basis $(e_j)_{j\in \N}$ of eigenfunctions of $-\Delta$, namely
$$
-\Delta e_j  = \lambda_j^2 e_j , \quad (e_j , e_k)_{L^2(M)} = \delta_{jk} , \quad  0 = \lambda_0 <\lambda_1 \leq \cdots \leq \lambda_j \leq \lambda_{j+1} \to + \infty ,
$$
where $e_0$ is the constant function normalized in $L^2(M)$ (recall $\dim \ker \Delta =1$ is a consequence of the fact that $M$ is connected), that is $e_0(x) = \Vol_g(M)^{-1/2}$ for $x \in M$.
For all $m\geq0$, we finally define $\Lambda_m : = \sqrt{- \Delta +m}$ via functional calculus as 
\begin{align}
\label{e:lambda-m}
\Lambda_m : L^2(M)\to L^2(M) ,  \quad \Lambda_m \left( \sum_{j=0}^\infty u_j e_j \right) = \sum_{j=0}^\infty \sqrt{\lambda_j^2+m} \ u_j e_j  ,  \quad  \text{ with }D(\Lambda_m)=H^1(M) .
\end{align}

\subsection{The damped Klein-Gordon equation}
\label{s:KG-reformulation}
We consider the damped Klein-Gordon equation (the damping function $b$ is only assumed $L^\infty(M)$ in this section)
\begin{align}
\label{e:KG-eq}
(\d_t^2 - \Delta +m + b\d_t)u=0 , \quad (u,\d_t u)|_{t=0} = (u_0,u_1) ,
\end{align}
with $m>0$ a positive constant. This simplifies the semigroup framework; the case $m=0$ is treated in Section~\ref{s:KG-waves} below. 
As already mentioned in the introduction, the energy under interest for Equation~\eqref{e:KG-eq} is the quantity
\begin{align*}
E_m(v_0,v_1)& =\frac12 \left( \nor{\nabla v_0}{L^2(M)}^2 + m \nor{v_0}{L^2(M)}^2 + \nor{v_1}{L^2(M)}^2 \right)
= \frac12 \left( \nor{\Lambda_m v_0}{L^2(M)}^2 + \nor{v_1}{L^2(M)}^2 \right)\\
& = \frac12 \left( \nor{v_0}{H^1_m(M)}^2 + \nor{v_1}{L^2(M)}^2 \right) . 
\end{align*}
Here, we define, for $s\in \R$,
\begin{align}
\label{e:m-dep-norm}
\nor{u}{H^s_m(M)} := \nor{\Lambda_m^s u}{L^2(M)} , \quad \nor{(u,v)}{H^\sigma_m(M)\times H^s_m(M)}^2 :=  \nor{\Lambda_m^\sigma u}{L^2(M)}^2 +  \nor{\Lambda_m^s v}{L^2(M)}^2 ,
\end{align}
which are norms on $H^s(M)$ (resp. $H^\sigma(M)\times H^s(M)$) that are equivalent to the usual one since $m>0$. We define $(\cdot ,\cdot)_{H^1_m\times L^2}$ the inner product associated to the $H^1_m\times L^2$ norm (which is the norm involved in the energy $E_m$).

We consider the following operators:
\begin{align}
\label{e:def-Am-KG}
\A_m & = 
\begin{pmatrix}
0   &  \id \\
\Delta - m & - b
\end{pmatrix} ,
 \text{ acting on } H^1(M)\times L^2(M) , \text{ with domain }
D(\A_m) =  H^2(M)\times H^1(M) , \\
\label{e:def-tAm-KG}
\tilde{\A}_m & = 
\begin{pmatrix}
0   &  \Lambda_m\\
-\Lambda_m & - b
\end{pmatrix} ,
 \text{ acting on } L^2(M;\C^2) , \text{ with domain }
D(\tilde{\A}_m) =  H^1(M;\C^2) , \\
\label{e:def-Pm-KG}
\P_m & = \Lambda_m
\begin{pmatrix}
1   &  0\\
0 & - 1
\end{pmatrix} + 
i \frac{b}{2}
\begin{pmatrix}
1   &  1\\
1 &  1
\end{pmatrix} ,
 \text{ acting on } L^2(M;\C^2) , \text{ with domain }
D(\P_m) =  H^1(M;\C^2) ,
\end{align}

The operator $\A_m$ already appears in~\eqref{eq: first order eqation} as the infinitesimal generator of the semigroup of the damped Klein-Gordon equation. The operator $\P_m$ takes the form of a hyperbolic system with diagonal principal part, which is mainly used in the rest of the paper.
The above three operators are linked via the following elementary lemma.

\begin{lemma}
\label{l:equiv-operators-m}
For all $s \in \R$, the operator 
\begin{equation}
\label{e:def-Lm}
L_m
 := \begin{pmatrix}
\Lambda_m   & 0\\
0  &  \id
\end{pmatrix}
\end{equation}
is an isomorphic isometry $H^{s+1}_m(M)\times H^s_m(M) \to H^s_m(M;\C^2)$, endowed with the norms~\eqref{e:m-dep-norm}, and we have 
$$
\tilde{\A}_m = 
L_m
 \A_m L_m^{-1}, 
 \qquad L_m^{-1} = 
 \begin{pmatrix}
\Lambda_m^{-1}   & 0\\
0  &  \id
\end{pmatrix} .
$$
The matrix 
\begin{equation}
\label{e:def-sigma}
\Sigma  := \frac{1}{\sqrt{2}}
\begin{pmatrix}
1& \frac1i \\
-1 & \frac1i
\end{pmatrix} 
\text{  is unitary, with }\quad
\Sigma^{-1} = \frac{1}{\sqrt{2}} \begin{pmatrix}
1& -1\\
i& i
\end{pmatrix} , 
\end{equation}
and we have 
$$
i\P_m=\Sigma\tilde{\A}_m\Sigma^{-1}.
$$
\end{lemma}

As a direct consequence, we deduce the following corollary (recall again that $M$ is connected).
\begin{corollary}
\label{corollary-Pm-Am-reso}
The following results hold:
\begin{enumerate}
\item 
The operators $\A_m,\tilde{\A}_m,\P_m$ have compact resolvents
$$
(z\id - \tilde{\A}_m)^{-1} = L_m 
 (z\id - \A_m)^{-1} 
L_m^{-1}, \quad  (z\id  - i\P_m)^{-1} = \Sigma (z\id - \tilde{\A}_m)^{-1} \Sigma^{-1} .
$$
 Moreover, $\Sp(\A_m)=\Sp(\tilde{\A}_m) = i \Sp(\P_m)$ consists in eigenvalues with finite multiplicity, accumulating only at infinity.
\item \label{i:resolvantes}
For all $z\in \C$, we have 
$$
\nor{(z\id - \A_m)^{-1}}{\L\big(H^1_m \times L^2\big)} = 
\nor{(z\id - \tilde{\A}_m)^{-1}}{\L\big(L^2(M;\C^2)\big)} = \nor{(z\id  - i\P_m)^{-1}}{\L\big(L^2(M;\C^2)\big)} ,  
$$
(with equality in $(0,+\infty]$).
\item For all $U=(u_0,u_1)\in H^2\times H^1$ and $\tilde{U}=L_mU = (\Lambda_m u_0,u_1)$ we have $$(\A_m U,U)_{H^1_m\times L^2} = (\tilde{\A}_m \tilde{U},\tilde{U})_{L^2(M;\C^2)} =2i \Im(u_1,u_0)_{H^1_m} -(b u_1,u_1)_{L^2(M)}.$$ 
In particular (assuming $b\geq 0$ \ae on $M$), $\Sp(\A_m)=\Sp(\tilde{\A}_m)\subset \{z \in \C ,- \nor{b}{L^\infty(M)}  \leq \Re(z)\leq 0\}$ and 
$\Sp(\P_m) \subset  \{z \in \C ,0 \leq \Im(z) \leq \nor{b}{L^\infty(M)} \}$.
\item \label{i:refined-loc} Assuming that $b\geq 0$ on $M$ and that $b$ does not vanish identically (that is to say, $\{b>0\}$ has positive measure), we have 
\begin{align*}
\Sp(\A_m)=\Sp(\tilde{\A}_m) &\subset \left\{z \in \C ,- \frac12 \nor{b}{L^\infty(M)}  \leq \Re(z) <  0 \right\} \cup \left\{r \in \R ,-  \nor{b}{L^\infty(M)} \leq r <  0 \right\}  , \\
\Sp(\P_m) &\subset \left\{z \in \C , 0< \Im(z) \leq  \frac12 \nor{b}{L^\infty(M)} \right\} \cup \left\{ i s  ,s\in \R, 0<s \leq \nor{b}{L^\infty(M)} \right\}  ,
\end{align*}
and in particular, $\Sp(\A_m) \cap i\R = \emptyset$, $\Sp(\P_m)\cap \R=\emptyset$.
\item The operators $\A_m,\tilde{\A}_m$ and $i \P_m$ generate groups of evolution that are contraction semigroups (on their respective Hilbert space of definition) denoted by $(e^{t\A_m})_{t\in \R},(e^{t\tilde{\A}_m})_{t\in \R}$ and $(e^{it\P_m})_{t\in \R}$, and we have 
$$
e^{t \tilde{\A}_m} = L_m
e^{t \A_m} 
L_m^{-1}, 
\quad \text{ and } \quad 
e^{it\P_m}=\Sigma e^{t\tilde{\A}_m}\Sigma^{-1}.
$$
\item \label{i:Energ-am} For all $U_0= (u_0,u_1) \in H^1(M)\times L^2(M)$, there is a unique solution to~\eqref{e:KG-eq} in $C^0(\R; H^1(M))\cap C^1(\R;L^2(M))$ and we have $(u,\d_t u)(t)= e^{t\A_m}U_0$, together with 
$$
E_m(u,\d_t u)(t) = \frac12 \nor{e^{t\A_m}U_0}{H^1_m\times L^2}^2 = \frac12 \nor{e^{t\tilde{\A}_m}L_mU_0}{L^2(M;\C^2)}^2 = \frac12 \nor{e^{i t\P_m}\Sigma L_m U_0}{L^2(M;\C^2)}^2 .
$$
\item \label{i:fPm} When $b=0$ in~\eqref{e:def-Am-KG},~\eqref{e:def-tAm-KG} and~\eqref{e:def-Pm-KG}, the resulting operators $A_m = \begin{pmatrix}
0   &  \id \\
-\Lambda_m^2 & 0
\end{pmatrix} $ and $\tilde{A}_m  = 
\begin{pmatrix}
0   &  \Lambda_m\\
-\Lambda_m & 0
\end{pmatrix}$  are skew-adjoint, and the operator
$P_m :=  \Lambda_m
\begin{pmatrix}
1   &  0\\
0 & - 1
\end{pmatrix}$
is selfadjoint (when acting on their respective spaces, with their respective domains).
Moreover, assuming that $f \in C^0(\R)$ is such that $f(-s) = f(s)$ for all $s\in \R$, we have 
$
f(P_m) = f(\Lambda_m) I_2 ,
$
and $f(P_m)\Sigma L_m U = \Sigma L_m f(\Lambda_m)U$ for all $U = \sum_{j=0}^\infty \begin{pmatrix} a_j\\ b_j \end{pmatrix}e_j$ such that $\sum_{j=0}^\infty f\big(\sqrt{\lambda_j^2 + m}\big)\big(\lambda_j^2 |a_j|^2 + |b_j|^2\big) <\infty$. In particular, for all $t\in \R$,
$$
 \frac12 \nor{ f(\Lambda_m)  e^{t\A_m}U_0}{H^1_m\times L^2}^2 = \frac12 \nor{f(P_m) e^{i t\P_m} (\Sigma L_m U_0)}{L^2(M;\C^2)}^2 .
$$
\end{enumerate}
\end{corollary}

\bnp
Most items come from Lemma~\ref{l:equiv-operators-m}, using that $L_m$ is an isomorphic isometry $H^{s+1}_m(M)\times H^s_m(M) \to H^s_m(M;\C^2)$ with the norms~\eqref{e:m-dep-norm}, and $\Sigma$ is unitary (hence an isomorphic isometry of $H^s_m(M;\C^2)$).

The refined localization of the spectrum in Item~\ref{i:refined-loc} relies on the unique continuation principle:
$$
\Big( (\lambda ,\psi) \in \R \times H^2_m, \quad (-\Delta +m) \psi = \lambda \psi ,  \quad b \psi = 0 \text{ \ae on }M \Big) \implies \psi = 0 \text{ on } M , 
$$
proved in~\cite{HS:89}, and we refer to~\cite{Leb:96} or~\cite[Lemma~4.2]{AL:14} for a complete proof. That $0\notin \Sp(\A_m)$ comes from the fact that $0\notin \Sp(-\Delta + m\id)$ since $m>0$.

Finally, the proof of Item~\ref{i:fPm} relies on the fact that $\big((e_j,0),(0,e_j)\big)_{j\in \N}$ is a Hilbert basis of $L^2(M;\C^2)$ diagonalizing $P_m$ with $P_m(e_j,0)=  \sqrt{\lambda_j^2+m} \ (e_j,0)$ and $P_m(0,e_j)= - \sqrt{\lambda_j^2+m} \ (0,e_j)$. This yields $\Sp(P_m) = \Sp(\Lambda_m)\cup (-\Sp(\Lambda_m))$ together with, for $U=(u_0,u_1)$,
$$f(P_m)U= \sum_{j\in\N}f\left(\sqrt{\lambda_j^2+m}\right)(u_0,e_j)_{L^2(M)}
\begin{pmatrix}
e_j\\
0
\end{pmatrix}+f\left(-\sqrt{\lambda_j^2+m}\right) (u_1,e_j)_{L^2(M)} \begin{pmatrix}
0\\
e_j
\end{pmatrix} = \begin{pmatrix}
f(\Lambda_m) u_0 \\
f(-\Lambda_m)u_1
\end{pmatrix},$$
whence the result if $f$ is even.
\enp

\subsection{The damped wave equation}
\label{s:KG-waves}
We now consider the damped wave equation 
\begin{align}
\label{e:DWE}
(\d_t^2 - \Delta + b\d_t)u=0 , \quad (u,\d_t u)|_{t=0} = (u_0,u_1) ,
\end{align}
on the compact manifold $M$ (which corresponds to~\eqref{e:KG-eq} with $m=0$). Recalling that $\Lambda_0 = \sqrt{-\Delta}$ is defined in~\eqref{e:lambda-m} via functional calculus, the energy under interest is the quantity
\begin{equation}
\label{e:def-energy}
E(v) = \frac12 \left( \nor{\nabla v_0}{L^2(M)}^2 + \nor{v_1}{L^2(M)}^2 \right)
=\frac12 \left( \nor{\Lambda_0 v_0}{L^2(M)}^2 + \nor{v_1}{L^2(M)}^2 \right) .
\end{equation}
When compared to the study of~\eqref{e:KG-eq} with $m>0$,  there is an additional difficulty linked to the fact that $0\in \Sp (\Delta)$, or equivalently, that the energy is not a norm on $H^1\times L^2$. Let us explain how to remedy this (see also~\cite{AL:14,LL:21} or~\cite{BG:20} for slightly different approaches). We recall that $M$ is assumed connected, so that $\dim \ker \Delta=1$. 
We define (recall the definitions of $\lambda_j, e_j$ at the beginning of Section~\ref{s:first-order-system})
$$
\Pi_0 u  = (u,e_0)_{L^2(M)} e_0 = \frac{1}{\Vol_g(M)}\int_M u(x) d\Vol_g(x) , \qquad \Pi_+ = \id - \Pi_0 , 
$$
or equivalently $\Pi_0(\sum_{j=0}^\infty u_j e_j ) = u_0 e_0$ and $\Pi_+(\sum_{j=0}^\infty u_j e_j ) = \sum_{j=1}^\infty u_j e_j $.
We set $L^2_+(M) = \Pi_+ L^2(M)$ and $H^s_+(M) = \Pi_+ H^s(M)$.
Next, we notice that 
$$
\Lambda_0 : L^2(M)\to L^2(M) ,  \quad \Lambda_0 \left( \sum_{j=0}^\infty u_j e_j \right) = \sum_{j=0}^\infty \lambda_j u_j e_j = \sum_{j=1}^\infty \lambda_j u_j e_j ,  \quad  \text{ with }D(\Lambda_0)=H^1(M) ,
$$
and we have $\ker\Lambda_0 = \vect(e_0)$ and $\Lambda_0 H^1(M) = L^2_+(M)$. Hence, we define its restriction $\Lambda_+:= \Lambda_0|_{L^2_+(M)}$ and have
$$
\Lambda_+^{-1} :L^2_+(M) \to L^2_+(M) , \quad \Lambda_+^{-1} \left( \sum_{j=1}^\infty u_j e_j \right) = \sum_{j=1}^\infty \frac{u_j}{\lambda_j} e_j .
$$
We thus have $\Lambda_+^{-1} \Lambda_+ = \Lambda_+ \Lambda_+^{-1} = \id_{L^2_+(M)}$ together with
$$
\Lambda_+^{-1} \Lambda_0 u = \Pi_+ u \quad \text{for all } u \in H^1(M) , \qquad 
\Lambda_0 \Lambda_+^{-1} v = v \quad \text{for all } v \in L^2_+(M) . 
$$
We now define 
\begin{align}
\label{e:+-dep-norm}
\nor{u}{H^s_+(M)} := \nor{\Lambda_+^s u}{L^2(M)} , \quad s\in \R ,
\end{align}
which is a norm on $H^s_+(M)$ (whereas $\nor{\Lambda_0^su}{L^2(M)}$ is only a seminorm on $H^s$). We also consider 
$$
 \nor{(u,v)}{H^\sigma_+(M)\times H^s(M)}^2 :=  \nor{\Lambda_+^\sigma u}{L^2(M)}^2 +  \nor{ (1+ \Lambda_0)^s v}{L^2(M)}^2 ,
$$
which is a norm on $H^{\sigma}_+(M)\times H^s(M)$. In particular, the norm in $H^1_+(M)\times L^2(M)$ is $ \nor{(u,v)}{H^1_+(M)\times L^2(M)}^2 =  \nor{\Lambda_+ u}{L^2(M)}^2 +  \nor{v}{L^2(M)}^2$ and we use the associated inner product $(\cdot ,\cdot)_{H^1_+\times L^2}$.

\medskip
We consider the following operators:
\begin{align}
\bA & = 
\begin{pmatrix}
0   &  \id  \\
\Delta & - b
\end{pmatrix} ,
 \text{ acting on } H^1(M)\times L^2(M) , \text{ with domain }
D(\bA) =  H^2(M)\times H^1(M) , \\
\A & = 
\begin{pmatrix}
0   &  \Pi_+ \\
\Delta & - b
\end{pmatrix} ,
 \text{ acting on } H^1(M)\times L^2(M) , \text{ with domain }
D(\A) =  H^2(M)\times H^1(M) , \\
\tilde{\A} & = 
\begin{pmatrix}
0   &  \Lambda_0\\
-\Lambda_0 & - b
\end{pmatrix} ,
 \text{ acting on } L^2(M;\C^2) , \text{ with domain }
D(\tilde{\A}) =  H^1(M;\C^2) , \\
\label{e:def-P}
\P & = \Lambda_0
\begin{pmatrix}
1   &  0\\
0 & - 1
\end{pmatrix} + 
i \frac{b}{2}
\begin{pmatrix}
1   &  1\\
1 &  1
\end{pmatrix} ,
 \text{ acting on } L^2(M;\C^2) , \text{ with domain }
D(\P) =  H^1(M;\C^2) .
\end{align}
Remark that the interest of $\tilde{\A}$ is that, as far as the energy $E$ in~\eqref{e:def-energy} is concerned, the damped wave equation is an equation on the couple $(\Lambda_0 u, \d_t u)$ only. The interest of $\P$ is that under this form, the damped wave equation~\eqref{e:DWE} formulates as a strictly hyperbolic problem.
We first make the following remark.
\begin{lemma}
\label{l:ker-op}
Assume that $b=0$. Then 
$$
\ker \bA = \vect\{(e_0,0)\} , \quad  \ker \A= \ker \tilde\A = \ker \P = \vect\{(e_0,0) , (0,e_0)\}  . 
$$
Assume $b\in L^\infty(M)$, $b\geq0$ and $b$ does not vanish identically (or equivalently $\{b>0\}$ has positive measure). Then
$$
\ker \bA = \ker \A= \vect\{(e_0,0)\} , \quad \ker \tilde\A = \vect\{(e_0,0)\} , \quad \ker \P = \vect\{(e_0,-e_0)\} .
$$
\end{lemma}
\bnp
The first part follows from the spectral decomposition of $\Delta,\Lambda_0$. Then, in case $b\geq0$ does not vanish identically, we have $\ker \bA = \{(u_0,u_1), u_1=0, \Delta u_0=0 \}= \vect\{(e_0,0)\}$. Taking now $(u_0,u_0) \in \ker \A$ is equivalent to $\Pi_+ u_1=0$ and $\Delta u_0-bu_1 = 0$, which in turn is equivalent to the existence of $\alpha \in \C$ such that $u_1=\alpha e_0$ and $\Delta u_0-\alpha b e_0 = 0$. Taking the inner product of this last identity with $e_0$ and using that $\Delta e_0 = 0$ implies $\alpha (b e_0, e_0)_{L^2(M)} = 0$. Since $e_0$ is constant on $M$ and $\int_M b(x) dx \neq 0$, we deduce that $\alpha =0$, and then $u_0 \in \ker \Delta$. Conversely $(e_0,0) \in \ker \A$. The same proof works for $\ker \tilde\A$. 
Finally, writing $u_0 = u_+-u_-$ and $u_1=i(u_++u_-)$, one finds that $(u_+,u_-) \in \ker \P$ if and only if $(u_0,u_1) \in \ker \tilde\A= \vect\{(e_0,0)\}$, that is to say $(u_+,u_-)\in \vect\{(e_0,-e_0)\}$.
\enp
We are thus led to consider restriction of these operators to a well-chosen supplementary subspace to (respectively) $\ker \A, \ker \tilde\A , \ker \P$.
Note that $\A \big(H^2_+(M)\times H^1(M) \big) \subset H^1_+(M)\times L^2(M)$ and $\tilde{\A} \big(H^1_+(M)\times H^1(M) \big) \subset L^2_+(M)\times L^2(M)$. Similarly, setting
\begin{align*}
\H^s_+  =  \{(u_+,u_-)\in H^s(M;\C^2), \Pi_0 u_+ = \Pi_0 u_-\} =\Sigma\left(H^s_+\times H^s\right) , \qquad \L^2_+  = \H^0_+ ,
\end{align*}
we remark  that we have 
$$
\begin{pmatrix}
\Pi_0   &  0\\
0  &  \Pi_0
\end{pmatrix} 
 \P
\begin{pmatrix}
u_+\\
u_-
\end{pmatrix}
=
\frac{i}{2} \begin{pmatrix}
\Pi_0 \big(b(u_+ + u_-) \big)\\
\Pi_0 \big(b(u_+ + u_-) \big)
\end{pmatrix} ,
$$
so that $\P$ may be restricted as a map $\H^1_+ \to \L^2_+$.
We endow $\L^2_+$ with the usual $L^2(M;\C^2)$ norm.
We may thus define the following restrictions:
\begin{align}
\label{e:def-A+}
\A_+  & = 
\begin{pmatrix}
0   &  \Pi_+ \\
-\Lambda_+^2 & - b
\end{pmatrix} ,
 \text{ acting on } H^1_+(M)\times L^2(M) , \text{ with domain }
D(\A_+) =  H^2_+(M)\times H^1(M) , \\
\label{e:def-tA+}
\tilde{\A}_+ & = 
\begin{pmatrix}
0   &  \Lambda_0\\
-\Lambda_+ & - b
\end{pmatrix} ,
 \text{ acting on } L^2_+(M)\times L^2(M) , \text{ with domain }
D(\tilde{\A}_+) =  H^1_+(M)\times H^1(M), \\
\label{e:def-P+}
\P_+ & = \Lambda_0
\begin{pmatrix}
1   &  0\\
0 & - 1
\end{pmatrix} + 
i \frac{b}{2}
\begin{pmatrix}
1   &  1\\
1 &  1
\end{pmatrix} ,
 \text{ acting on } \L^2_+ , \text{ with domain }
D(\P_+) =  \H^1_+ ,
\end{align}

The above operators are linked via the following elementary lemma (Analogue of Lemma~\ref{l:equiv-operators-m} adapted to the damped wave equation).

\begin{lemma}[Links between $\A_+,\tilde{\A}_+,\P_+$]
\label{l:equiv-operators}
The operator 
\begin{align}
\label{e:def-L+}
L_+
 := \begin{pmatrix}
\Lambda_+   & 0\\
0  &  \id
\end{pmatrix}
\end{align}
is an isomorphic isometry $H^{s+1}_+(M)\times H^s(M) \to H^s_+(M) \times H^s(M)$, endowed with the norms~\eqref{e:+-dep-norm} and we have 
$$
\tilde{\A}_+ = 
L_+
 \A_+ L_+^{-1}, 
 \qquad L_+^{-1} = 
 \begin{pmatrix}
\Lambda_+^{-1}   & 0\\
0  &  \id
\end{pmatrix} .
$$
The map $U\mapsto \Sigma U$ defined in~\eqref{e:def-sigma} is an isomorphism $H^s_+(M)\times H^s(M) \to \H^s_+$, an isometry $L^2_+(M)\times L^2(M) \to \L^2_+$ (with the norms~\eqref{e:+-dep-norm}),  
and we have 
$$
i\P_+ =\Sigma\tilde{\A}_+ \Sigma^{-1}.
$$
\end{lemma}
\bnp
That $\Sigma : L^2_+(M)\times L^2(M) \to \L^2_+$ is an isometry follows from the parallelogram identity together with the fact that these spaces are endowed with usual $L^2$ norms.
\enp
The interest of this lemma is to reformulate the damped wave equation as a first order hyperbolic system with diagonal principal part.
The following corollary (analogue of Corollary~\ref{corollary-Pm-Am-reso} for the damped wave equation) is a direct consequence of Lemma~\ref{l:equiv-operators}  (recall again that $M$ is connected).
\begin{corollary}[Properties of $\A_+,\tilde{\A}_+,\P_+$]
\label{c:prop-op-+}
The following results hold:
\begin{enumerate}
\item 
The operators $\A_+,\tilde{\A}_+,\P_+$ have compact resolvents
$$
(z\id - \tilde{\A}_+)^{-1} = L_+ 
 (z\id - \A_+)^{-1} 
L_+^{-1}, \quad  (z\id  - i\P_+)^{-1} = \Sigma (z\id - \tilde{\A}_+)^{-1} \Sigma^{-1} .
$$
 Moreover, $\Sp(\A_+)=\Sp(\tilde{\A}_+) = i \Sp(\P_+)$ consists in eigenvalues with finite multiplicity, accumulating only at infinity.
\item \label{i:egal-resol+}
For all $z\in \C$, we have 
$$
\nor{(z\id - \A_+)^{-1}}{\L\big(H^1_+ \times L^2\big)} = 
\nor{(z\id - \tilde{\A}_+)^{-1}}{\L\big(L^2(M;\C^2)\big)} = \nor{(z\id  - i\P_+)^{-1}}{\L\big(L^2(M;\C^2)\big)} , 
$$
(with equality in $(0,+\infty]$). Note that $\nor{\cdot}{\L^2_+} =\nor{\cdot}{L^2(M;\C^2)}$ so that $\nor{\cdot}{\L(\L^2_+)} = \nor{\cdot}{\L\big(L^2(M;\C^2)\big)}$.
\item For all $U=(u_0,u_1)\in H^2_+\times H^1$ and $\tilde{U}=L_+U = (\Lambda_+ u_0,u_1)$ we have $$(\A_+ U,U)_{H^1_+\times L^2} = (\tilde{\A}_+ \tilde{U},\tilde{U})_{L^2(M;\C^2)} =2i\Im(\Lambda_0 u_1,\Lambda_+u_0)_{L^2_+} -(b u_1,u_1)_{L^2(M)}.$$ 
In particular (assuming $b\geq 0$ \ae on $M$), $\Sp(\A_+)=\Sp(\tilde{\A}_+)\subset \{z \in \C ,- \nor{b}{L^\infty(M)}  \leq \Re(z)\leq 0\}$ and 
$\Sp(\P_+) \subset  \{z \in \C ,0 \leq \Im(z) \leq \nor{b}{L^\infty(M)} \}$.
\item \label{i:loc-spec+}  Assuming that $b\geq 0$ on $M$ and that $b$ does not vanish identically (that is to say, $\{b>0\}$ has positive measure), we have 
\begin{align*}
\Sp(\A_+)=\Sp(\tilde{\A}_+) &\subset \left\{z \in \C ,- \frac12 \nor{b}{L^\infty(M)}  \leq \Re(z) <  0 \right\} \cup \left\{r \in \R ,-  \nor{b}{L^\infty(M)} \leq r <  0 \right\}  , \\
\Sp(\P_+) &\subset \left\{z \in \C , 0< \Im(z) \leq  \frac12 \nor{b}{L^\infty(M)} \right\} \cup \left\{ i s  ,s\in \R, 0<s \leq \nor{b}{L^\infty(M)} \right\}  ,
\end{align*}
and in particular, $\Sp(\A_+) \cap i\R = \emptyset$, $\Sp(\P_+)\cap \R=\emptyset$.
\item The operators $\A_+,\tilde{\A}_+$ and $i \P_+$ generate groups of evolution that are contraction semigroups (on their respective Hilbert space of definition)
 denoted $(e^{t\A_+})_{t\in \R},(e^{t\tilde{\A}_+})_{t\in \R}$ and $(e^{it\P_+})_{t\in \R}$, and we have 
$$
e^{t \tilde{\A}_+} = L_+
e^{t \A_+} 
L_+^{-1}, 
\quad \text{ and } \quad 
e^{it\P_+}=\Sigma e^{t\tilde{\A}_+}\Sigma^{-1}.
$$
\item \label{i:++enrgy} For all $U_0= (u_0,u_1) \in H^1_+(M)\times L^2(M)$, there is a unique solution to~\eqref{e:KG-eq} in $C^0(\R; H^1_+(M))\cap C^1(\R;L^2(M))$ and we have $(u,\d_t u)(t)= e^{t\A_+}U_0$, together with 
$$
E(u(t)) = \frac12 \nor{e^{t\A_+}U_0}{H^1_+\times L^2}^2 = \frac12 \nor{e^{t\tilde{\A}_+}L_+U_0}{L^2(M;\C^2)}^2 = \frac12 \nor{e^{i t\P_+}\Sigma L_+ U_0}{L^2(M;\C^2)}^2
$$
\end{enumerate}
\end{corollary}

\bnp
The proof is the same as that of Corollary~\ref{corollary-Pm-Am-reso} and most items come from Lemma~\ref{l:equiv-operators}.
For the refined localization of the spectrum of Item~\ref{i:loc-spec+}, one uses again unique continuation~\cite{HS:89} and 
the fact that $0\notin \Sp(\A_+)$ comes from the fact that $0\notin \Sp(\Lambda_+)$ (by definition).
\enp

\begin{lemma}[Links between $\A,\tilde{\A},\P$ and $\A_+,\tilde{\A}_+,\P_+$]
\label{l:link-+ou-pas}
For any $b \in L^\infty(M;\R_+)$, we have 
\begin{enumerate}
\item $\A_+= \A|_{H^1_+\times L^2}$, $\tilde{\A}_+= \tilde{\A}|_{L^2_+\times L^2}$, and $\P_+ = \P|_{\L^2_+}$, and 
\begin{align}
\A & =\A_+ \begin{pmatrix}
\Pi_+ & 0\\
0& \id
\end{pmatrix}+ \begin{pmatrix}
\Pi_0 & 0\\
0& 0
\end{pmatrix} , \quad 
\tilde{\A} =\tilde{\A}_+ \begin{pmatrix}
\Pi_+ & 0\\
0& \id
\end{pmatrix}+ \begin{pmatrix}
\Pi_0 & 0\\
0& 0
\end{pmatrix} ,  \nonumber
\\
\label{e:link-P-P+}
\P & =\P_+\Pi_{\L^2_+}
+
\frac{\Pi_0}{2} \begin{pmatrix}
1& -1\\
-1& 1
\end{pmatrix}
, \quad \text{ where }
\Pi_{\L^2_+} = 
 \begin{pmatrix}
\Pi_+ & 0\\
0& \Pi_+
\end{pmatrix}
+  \frac{\Pi_0}{2} \begin{pmatrix}
1& 1\\
1& 1
\end{pmatrix}.
\end{align}
\item \label{i:commutat-0} $\left[ \A , \begin{pmatrix}
\Pi_+ & 0\\
0& \id
\end{pmatrix} \right] = 0$, $\left[ \tilde{\A} , \begin{pmatrix}
\Pi_+ & 0\\
0& \id
\end{pmatrix} \right] = 0$, and $[\P,\Pi_{\L^2_+}]=0$.

\item \label{i:resolvents-+ou-pas}
$\Sp(\A) = \Sp(\tilde{\A}) = i \Sp(\P) = \Sp(\A_+)\cup \{0\}= i\Sp(\P_+)\cup \{0\}$ and 
\begin{align*}
(\lambda \id -\A)^{-1}
& =
(\lambda \id -\A_+)^{-1}
 \begin{pmatrix}
\Pi_+ & 0\\
0& \id
\end{pmatrix}  + \lambda^{-1}  \begin{pmatrix}
\Pi_0 & 0\\
0& 0
\end{pmatrix} ,   \\
(\lambda \id -\tilde{\A})^{-1}
& =
(\lambda \id - \tilde{\A}_+)^{-1}
 \begin{pmatrix}
\Pi_+ & 0\\
0& \id
\end{pmatrix}  + \lambda^{-1}  \begin{pmatrix}
\Pi_0 & 0\\
0& 0
\end{pmatrix} , \\
(\lambda \id - \P)^{-1} & = (\lambda \id - \P_+)^{-1} \Pi_{\L^2_+}
+ \lambda^{-1}
\frac{\Pi_0}{2} \begin{pmatrix}
1& -1\\
-1& 1
\end{pmatrix} .
\end{align*}

\item  \label{i:++A++A+}The operators $\A,\tilde{\A}$ and $i \P$ generate groups of evolution that are bounded semigroups denoted $(e^{t\A})_{t\in \R_+},(e^{t\tilde{\A}})_{t\in \R_+}$ and $(e^{it\P})_{t\in \R_+}$, and we have 
\begin{align*}
e^{t \A} & = e^{t \A_+}
\begin{pmatrix}
\Pi_+ & 0\\
0& \id
\end{pmatrix}+ \begin{pmatrix}
\Pi_0 & 0\\
0& 0
\end{pmatrix}
, \quad 
e^{t \tilde{\A}}
= 
e^{t \tilde{\A}_+}
\begin{pmatrix}
\Pi_+ & 0\\
0& \id
\end{pmatrix}+ \begin{pmatrix}
\Pi_0 & 0\\
0& 0
\end{pmatrix} , \\
e^{it\P} & = e^{it\P_+}\Pi_{\L^2_+}
+
\frac{\Pi_0}{2} \begin{pmatrix}
1& -1\\
-1& 1
\end{pmatrix} .
\end{align*}
\end{enumerate}

\end{lemma}

\bnp
That the first two commutators vanish in Item~\ref{i:commutat-0} is straightforward and the third one follows from the fact that $\Pi_{\L^2_+} + \frac{\Pi_0}{2} \begin{pmatrix}
1& -1\\
-1& 1
\end{pmatrix} = \id$ together with 
\begin{align*}
\P  \frac{\Pi_0}{2} 
\begin{pmatrix}
1& -1\\
-1& 1
\end{pmatrix}  
=
i\frac{b}{2} 
\begin{pmatrix}
1& 1\\
1& 1
\end{pmatrix}  
 \frac{\Pi_0}{2} 
\begin{pmatrix}
1& -1\\
-1& 1
\end{pmatrix}  
= 0 
= 
 \frac{\Pi_0}{2} 
\begin{pmatrix}
1& -1\\
-1& 1
\end{pmatrix}  
i\frac{b}{2} 
\begin{pmatrix}
1& 1\\
1& 1
\end{pmatrix} 
=  \frac{\Pi_0}{2} \begin{pmatrix}
1& -1\\
-1& 1
\end{pmatrix}  \P.
\end{align*}
Item~\ref{i:resolvents-+ou-pas} come from the fact that 
$$
(\lambda \id -\A)^{-1}
 \begin{pmatrix}
\Pi_+ & 0\\
0& \id
\end{pmatrix} 
=
(\lambda \id -\A_+)^{-1}
 \begin{pmatrix}
\Pi_+ & 0\\
0& \id
\end{pmatrix}  , \quad (\lambda \id -\A)^{-1}
 \begin{pmatrix}
\Pi_0 & 0\\
0& 0
\end{pmatrix} 
= \lambda^{-1} 
 \begin{pmatrix}
\Pi_0 & 0\\
0& 0
\end{pmatrix} .
$$
\enp

We also state the following in case $b=0$, which is an analogue of Item~\ref{i:fPm} in Corollary~\ref{corollary-Pm-Am-reso}, together with additional precision concerning the eigenfrequency $0$.
\begin{lemma}
Assume $b=0$ in~\eqref{e:def-P},~\eqref{e:def-A+},~\eqref{e:def-tA+} and~\eqref{e:def-P+}, then:
\begin{enumerate}
\item the resulting operators $A_+ = \begin{pmatrix}
0   &  \id  \\
-\Lambda_+^2 & 0
\end{pmatrix} $ and $\tilde{A}_+  = 
\begin{pmatrix}
0   &  \Lambda_0 \\
-\Lambda_+ & 0
\end{pmatrix}$ are skew-adjoint (when acting on their respective spaces, with their respective domains), and the operators
$P :=  \Lambda_0
\begin{pmatrix}
1   &  0\\
0 & - 1
\end{pmatrix}$ acting on $L^2(M;\C^2)$ and 
$P_+ :=  \Lambda_0
\begin{pmatrix}
1   &  0\\
0 & - 1
\end{pmatrix}$ 
acting  on $\L^2_+$ are selfadjoint.
\item \label{i:ker-ker} $\ker(\tilde{A}_+) = \ker(A_+)= \vect\{(0,e_0)\}$, $\ker(P_+)=  \vect\{(e_0,e_0)\}$ ;
\item \label{i:P-P+-funct} if $f \in C^0(\R)$ is such that $f(-s) = f(s)$ for all $s\in \R$, we have 
\begin{align*}
f(P) & = f(\Lambda_0) I_2 ,  \quad \text{ on } L^2(M;\C^2),  
\quad f(P_+) = f(\Lambda_+) I_2 + f(0)\Pi_0 I_2, \quad \text{ on } \L^2_+ ,\\
 f(P) & =  f(P_+) \Pi_{\L^2_+} + f(0) \frac{\Pi_0}{2} \begin{pmatrix}
1& -1\\
-1& 1
\end{pmatrix} .
\end{align*}
\end{enumerate}
\end{lemma}
Item~\ref{i:ker-ker} of the lemma has to be compared with Lemma~\ref{l:ker-op}.
\bnp
The proof of the first two items is classical. 
That $f(P) = f(\Lambda_0) I_2$ in Item~\ref{i:P-P+-funct} comes from the same computation as in the proof of Item~\ref{i:fPm} in Corollary~\ref{corollary-Pm-Am-reso}. The second and third statements in Item~\ref{i:P-P+-funct}  come from~\eqref{e:link-P-P+} together with $\id= \Pi_++\Pi_0$ and thus $f(\Lambda_0) = f(\Lambda_0) \Pi_++f(\Lambda_0)\Pi_0 = f(\Lambda_+)\Pi_+  +  f(0)\Pi_0$.
\enp

We finally come back to the operator $\bA$ which also naturally enters into the game.
\begin{lemma}[Links with $\bA$]
The following statements hold:
\begin{enumerate}
\item \label{iioo--iuu} For all $U=(u_0,u_1)\in H^2\times H^1$, we have
$$(\bA U,U)_{H^1\times L^2} = 2i\Im(\nabla u_1,\nabla u_0)_{L^2(M)} + ( u_1, u_0)_{L^2(M)} -(b u_1,u_1)_{L^2(M)}.$$  
\item \label{i:intertwin} On $H^2\times H^1$, we have
$$
 \begin{pmatrix}
\Pi_+ & 0\\
0& \id
\end{pmatrix}
\bA = \A = 
\A
 \begin{pmatrix}
\Pi_+ & 0\\
0& \id
\end{pmatrix}
= \A_+
 \begin{pmatrix}
\Pi_+ & 0\\
0& \id
\end{pmatrix} .
$$
\item  \label{i:inter-semi} The operator $\bA$ generates a group of evolution on $H^1\times L^2$, denoted $(e^{t\bA})_{t\in \R}$, satisfying
$$
 \begin{pmatrix}
\Pi_+ & 0\\
0& \id
\end{pmatrix}
e^{t\bA} = 
e^{t\A}
 \begin{pmatrix}
\Pi_+ & 0\\
0& \id
\end{pmatrix}
= e^{t\A_+}
 \begin{pmatrix}
\Pi_+ & 0\\
0& \id
\end{pmatrix} .
$$
\item For all $U_0= (u_0,u_1) \in H^1(M)\times L^2(M)$, there is a unique solution to~\eqref{e:DWE} in $C^0(\R; H^1(M))\cap C^1(\R;L^2(M))$ and we have $(u,\d_t u)(t)= e^{t\bA}U_0$, together with 
\begin{align*}
E (u(t))
&= E (e^{t\bA}U_0) =  E \left( e^{t\A}U_0^+ \right)
 =  \frac12 \nor{e^{t\A_+} U_0^+}{H^1_+\times L^2}^2 
= \frac12 \nor{e^{t\tilde{\A}_+}L_+U_0^+}{L^2(M;\C^2)}^2 \\
& = \frac12 \nor{e^{i t\P_+}\Sigma L_+ U_0^+}{L^2(M;\C^2)}^2 
= \frac12 \nor{e^{i t\P_+}\Sigma L U_0}{L^2(M;\C^2)}^2 
= \frac12 \nor{e^{i t\P}\Sigma L U_0}{L^2(M;\C^2)}^2 ,
\end{align*}
where $U_0^+ =  \begin{pmatrix}
\Pi_+ & 0\\
0& \id
\end{pmatrix} U_0 =  \begin{pmatrix}
\Pi_+u_0\\
u_1
\end{pmatrix}$ and $L=  \begin{pmatrix}
\Lambda_0 & 0\\
0 & \id
\end{pmatrix}$.
\end{enumerate}
\end{lemma}

\begin{proof}
Items~\ref{iioo--iuu} and~\ref{i:intertwin}  are direct computations.
To prove Item~\ref{i:inter-semi}, first notice that $\bA$ generates a group of evolution from Item~\ref{iioo--iuu} and a variant of the Hille-Yosida theorem.
Then, setting $U(t) =  \begin{pmatrix}
\Pi_+ & 0\\
0& \id
\end{pmatrix} e^{t\bA} U_0$ for $U_0 \in D(\bA)$, we have $\d_t U(t) =  \begin{pmatrix}
\Pi_+ & 0\\
0& \id
\end{pmatrix} \bA  U(t)  = \A U(t)$
according to Item~\ref{i:intertwin}, that is to say $U(t) = e^{t\A}U_0$. 
Using Item~\ref{i:++A++A+} of Lemma~\ref{l:link-+ou-pas}, we deduce that 
 $U(t) =  \begin{pmatrix}
\Pi_+ & 0\\
0& \id
\end{pmatrix} U(t) =  \begin{pmatrix}
\Pi_+ & 0\\
0& \id
\end{pmatrix} e^{t\A}U_0 = e^{t\A} \begin{pmatrix}
\Pi_+ & 0\\
0& \id
\end{pmatrix} U_0$, which proves Item~\ref{i:inter-semi}.
\end{proof}

As a consequence of this lemma, studying decay rates of the energy for the damped wave equation~\eqref{e:DWE} is equivalent to study decay properties for the semigroups $e^{t\A_+}, e^{t\tilde{\A}_+}$ or $e^{it\P}, e^{it\P_+}$. From the point of view of applying results like the Gearhart-Huang-Pr\"uss theorem~\cite{Gearhart:78,Huang:85,Pruss:84}, the Batty-Duyckaerts Theorem~\ref{t:batty-duyckaerts}, or our Theorems~\ref{t:thm-classiq-A-0} and~\ref{t:semiclassic-intro}, the useful resolvents are thus those of $\A_+ ,\tilde{\A}_+,\P_+$.

Note however that on account to Item~\ref{i:resolvents-+ou-pas} in Lemma~\ref{l:link-+ou-pas}, the latter are equivalent to those of  $\A ,\tilde{\A},\P$ away from zero (and in particular near $\pm i\infty$, which is the main concern of the abovementioned theorems), when measured with the ``operator norms'' associated to the seminorms of $H^1_+\times L^2$, $L^2_+\times L^2$ and $\L^2_+$ respectively.

\bigskip
To conclude this section, let us give an additional property of the operator $\bA$. The latter is not needed for the rest of the paper, but clarifies the convergence of solutions of the damped wave equation as $t\to + \infty$, and makes a link with the slightly different approach in~\cite{AL:14}.
\begin{lemma}[Properties of $\bA$]
Assume $b\geq 0$ and $b$ does not vanish identically.
Then, we have $\ker(\bA)=  \vect(e_0,0)$ and the spectral projector of $\bA$ associated to the eigenvalue $0$ writes:
$$
 \Pi_{\ker\bA}
\left(
\begin{array}{c}
u_0 \\ u_1
\end{array}
\right) 
=\left(\frac{1}{\int_M b} \int_M (b(x)u_0(x) + u_1(x))d\Vol_g(x)\right)
\left(
\begin{array}{c}
1 \\ 0
\end{array}
\right) ,
$$
Moreover, there exists a constant $C>0$ such that for all $U = (u_0,u_1) \in H^1\times L^2$, we have
$$
2 E(U) \leq \nor{U-\Pi_{\ker\bA}U}{H^1\times L^2}^2 \leq 2 C E(U) ,
$$
that is to say 
\begin{align}
\label{e:equiv-norm-dw}
\|\nabla u_0\|_{L^2}^2 + \| u_1\|_{L^2}^2 \leq \nor{u_0 -\frac{1}{\int_M b} \int_M (b \  u_0 + u_1 )d\Vol_g}{H^1}^2 + \nor{u_1}{L^2}^2 \leq C \left( \|\nabla u_0\|_{L^2}^2 + \| u_1\|_{L^2}^2\right) .
\end{align}
Finally, we have 
\begin{align}
\label{e:decay-decomp}
e^{t\bA} = \Pi_{\ker\bA} + e^{t\bA} (\id_{H^1\times L^2} -  \Pi_{\ker\bA} )  
\end{align}
\end{lemma} 
Note that~\eqref{e:equiv-norm-dw} combined with~\eqref{e:decay-decomp} prove that if convergence holds (that is to say, as soon as $b\geq 0$ does not vanish identically, see \eg~\cite{Leb:96}), then 
$$
 u(t)  \to   \frac{1}{\int_M b} \int_M (b \  u_0 + u_1 )d\Vol_g  \quad\text{ in } H^1(M) \text{ as }t \to + \infty .
$$
\bnp
Only the explicit expression of $\Pi_{\ker\bA}$ and~\eqref{e:equiv-norm-dw} are not proved in~\cite[Section~4]{AL:14}. A direct computation shows that $\Pi_{\ker\bA}^2=\Pi_{\ker\bA}$, $[\Pi_{\ker\bA} ,\bA]=0$ and $\Pi_{\ker\bA}(H^1\times L^2) =  \vect(e_0,0)$.
As far as~\eqref{e:equiv-norm-dw} is concerned, we remark that the central term in~\eqref{e:equiv-norm-dw} is equal to 
$$
\nor{U-\Pi_{\ker\bA}U}{H^1\times L^2}^2  =  \nor{\nabla u_0}{L^2}^2  +  \nor{ u_0 -\frac{1}{\int_M b} \int_M (b \  u_0 + u_1 )d\Vol_g}{L^2}^2 + \nor{u_1}{L^2}^2 ,
$$
so that the left inequality holds. Concerning the right inequality, we only need to estimate the term
\begin{align*}
 \nor{ u_0 -\frac{1}{\int_M b} \int_M (b \  u_0 + u_1 )d\Vol_g}{L^2} & \leq  \nor{ u_0 -\frac{1}{\int_M b} \int_M b \  u_0 d\Vol_g}{L^2}+ \frac{1}{\int_M b} \nor{ \int_M  u_1 d\Vol_g}{L^2} \\
 & \leq  \nor{ u_0 -\frac{1}{\int_M b} \int_M b \  u_0 d\Vol_g}{L^2}+ \frac{\Vol_g(M)}{\int_M b} \|u_1\|_{L^2} .
\end{align*}
To conclude, it suffices to remark that the inequality
$$
\nor{ \psi -\frac{1}{\int_M b} \int_M b \psi d\Vol_g}{L^2(M)} \leq C_{M,g,b} \|\nabla \psi\|_{L^2(M)} ,\quad \text{ for all } \psi \in H^1(M) ,
$$
is a version of the Poincar\'e-Wirtinger inequality (which can be proved by the classical contradiction compactness argument). The last three lines yield the right inequality in~\eqref{e:equiv-norm-dw}.
\enp

\subsection{Semiclassical diagonalization of hyperbolic systems}
\label{s:semiclass-hyperbolic}
In Sections~\ref{s:KG-reformulation} and~\ref{s:KG-waves}, we have reformulated the evolution equation~\eqref{eq: stabilization} under the form 
\begin{align}
\label{e:damp-wave-hyp}
D_t U = \P_m U , \quad \text{with} \quad 
\P_m = \Lambda_m \begin{pmatrix}
1& 0 \\
0 & -1
\end{pmatrix}  + i \frac{b}{2} \begin{pmatrix}
1& 1 \\
1 & 1
\end{pmatrix}.
\end{align}
on $L^2(M;\C^2)$. In this section, we consider the following semiclassical evolution equation
\begin{align}
\label{e:first-system-h}
\left(D_t - \mathsf{P} \right) U = 0  ,  \quad \text{ with } \mathsf{P} =  \Lambda I_\pm  + i Q_h  , 
\quad 
I_\pm = 
 \begin{pmatrix}
1& 0 \\
0 & -1
\end{pmatrix} ,
\quad 
Q_h =  
 \begin{pmatrix}
Q_h^+ & Q_h^{\pm}\\
Q_h^{\mp} &Q_h^-
\end{pmatrix}
\in \Psi^0_\nu (M ;\C^{2\times 2}) .
\end{align}
Here, the definition of the classes $S^m_\nu, \Psi^m_\nu$ is recalled in Appendix~\ref{appendix}, $\Lambda \in \Psi^1(M)$ is any selfadjoint first order pseudodifferential operator with principal symbol $\sigma_h (\Lambda)(x,\xi) = |\xi|_x :=\sqrt{g_x^*(\xi,\xi)}$ (where $g^*$ is the metric on $T^*M$). In particular, any $\Lambda_m$ with $m\geq0$ satisfies these assumptions according to functional calculus~\cite[Section~8]{DS:book} or~\cite[Theorem~14.9]{Zworski:book}.
In case $b \in C^\infty(M)$, notice that the damped Klein-Gordon equation~\eqref{e:KG-eq}, as formulated in~\eqref{e:damp-wave-hyp}, can be rewritten under the  form~\eqref{e:first-system-h} with $Q_h^+=Q_h^-=Q_h^{\pm}=Q_h^{\mp} = \frac{b}{2}$, and hence $\nu=0$.
If $b$ is not smooth, then~\eqref{e:damp-wave-hyp} does not take the form~\eqref{e:first-system-h}. However, in this situation, we shall also be able to use the results of the present section after having appropriately regularized $b$ (see Appendix~\ref{e:reg-b} and the proof of Theorem~\ref{t:main-res} in Section~\ref{s:proof-main}).
We denote by
\begin{align}
q_h =  
 \begin{pmatrix}
q_h^+ & q_h^{\pm}\\
q_h^{\mp} &q_h^-
\end{pmatrix}
= \sigma_h (Q_h) 
\in S^0_\nu (M ;\C^{2\times 2})
\end{align}
the principal symbol of $Q_h$. We shall always assume that the system is dissipative (at the level of the principal symbol), that is to say:
\begin{align}
\label{e:asspt-positive}
\Re(q_h) = \frac12( q_h + q_h^*) \geq 0 \text{ on } T^*M ,
\end{align}
with $q_h^*$ the adjoint matrix of $q_h$ in $\C^{2\times2}$ (and the inequality holds in the sense of hermitian matrices).
We define two energy cutoffs by functional calculus as $\chi_0(h\Lambda)$ and $\tilde{\chi}_0(h\Lambda)$, where 
\begin{align}
\label{e:def-cut-off-energy-0}
\chi_0 ,\tilde{\chi}_0 \in C^\infty_c((0,2);[0,1]) , \quad \chi_0 =1 \text{ in a neighborhood of } 1,
 \quad \tilde{\chi}_0 =1 \text{ in a neighborhood of }\supp(\chi_0) .
\end{align}
We also identify $\chi_0(h\Lambda)$ with $\chi_0(h\Lambda) I_2$, where $I_2= \diag(1,1)$ (and similarly with $\tilde{\chi}_0(h\Lambda)$).
According to~\cite[Theorem~14.9]{Zworski:book} (see also Corollary~\ref{c:functional-calc} below), we have 
$$
\chi_0(h\Lambda) , \tilde{\chi}_0(h\Lambda) \in \Psi^{-\infty}(M; \C^{2\times 2})
$$ with respective principal symbols $\sigma_h(\chi_0(h\Lambda))(x,\xi) = \chi_0(|\xi|_x)$ (See Appendix~\ref{s:calcul-pseudo} for notation of semiclassical calculus).

The following proposition aims at quantifying the error made by replacing $\mathsf{P}$ by a spectrally truncated version of $\mathsf{P}$ when data are themselves spectrally localized. This allows in particular to remove the singularity at $\xi=0$ of $|\xi|_x$ and $\Lambda$.
\begin{proposition}
\label{p:addition-energy-cutoff}
Assume~\eqref{e:asspt-positive} and set 
\begin{align}
\label{e:def-P1}
\mathsf{P}_{\cut} = \Lambda\tilde{\chi}_0(h\Lambda) I_\pm    + i \tilde{\chi}_0(h\Lambda)Q_h \tilde{\chi}_0(h\Lambda)  .
\end{align}
Then, there are $C,C_0,h_0>0$ such that we have 
\begin{align}
\label{e:energy-cutoff}
\nor{\chi_0(h\Lambda) e^{it\mathsf{P}} - e^{it\mathsf{P}_{\cut}} \chi_0(h\Lambda)}{\L(L^2)}  \leq C t h^{1-2\nu} e^{C_0t h^{1-2\nu}} ,\\
\label{e:energy-cutoff-bis}
\nor{e^{it\mathsf{P}} \chi_0(h\Lambda) - e^{it\mathsf{P}_{\cut}} \chi_0(h\Lambda)}{\L(L^2)}  \leq C t h^{1-2\nu} e^{C_0t h^{1-2\nu}} ,
\end{align}
for all $t\geq0$ and $h\in (0,h_0)$.
If we assume further $\Re(Q_h) = \frac12 (Q_h+ Q_h^*) \geq 0$ in the sense of operators (which implies \eqref{e:asspt-positive}), then \eqref{e:energy-cutoff}-\eqref{e:energy-cutoff-bis} hold true with $C_0=0$.
\end{proposition}
The next proposition estimates the error made by replacing $\mathsf{P}_{\cut}$ by its diagonal part.
\begin{proposition}
\label{p:diagonalization}
Assume~\eqref{e:asspt-positive}, recall that $\mathsf{P}_{\cut}$ is defined in~\eqref{e:def-P1}, and let 
\begin{align}
\label{e:def-Pd}
\mathsf{P}_{\diag} = \Lambda \tilde{\chi}_0(h\Lambda) I_\pm    + i
 \begin{pmatrix}
\tilde{\chi}_0(h\Lambda) Q_h^+\tilde{\chi}_0(h\Lambda) &0  \\
0 &\tilde{\chi}_0(h\Lambda)  Q_h^-\tilde{\chi}_0(h\Lambda) 
\end{pmatrix} .
\end{align}
Then there are $C,C_0,h_0>0$ such that we have 
\begin{align}
\label{e:diagonal}
\nor{e^{it\mathsf{P}_{\cut}} - e^{it\mathsf{P}_{\diag}}}{\L(L^2)}  \leq C\left(th^{1-2\nu}+h\right) e^{C_0t h^{1-2\nu}} ,
\end{align}
for all $t\geq0$ and $h\in (0,h_0)$.
If we assume further $\Re(Q_h) = \frac12 (Q_h+ Q_h^*) \geq 0$ in the sense of operators, then \eqref{e:diagonal} holds true with $C_0=0$.
\end{proposition}

Combining Propositions~\ref{p:addition-energy-cutoff} and~\ref{p:diagonalization} implies that for data spectrally localized, one can replace $e^{it\mathsf{P}}$ by $e^{it\mathsf{P}_{\diag}}$ with a small error if $t$ is not too large. We rewrite this in the following corollary under the form used below for future reference.
\begin{corollary}
\label{c:energy-cutoff-final}
With all above definitions, assuming $\Re(Q_h) \geq 0$ in the sense of operators, we have
\begin{align}
\label{e:energy-cutoff-final}
\nor{e^{it\mathsf{P}} \chi_0(h\Lambda) I_2 - e^{it\mathsf{P}_{\diag}}  \chi_0(h\Lambda)I_2 }{\L(L^2)}  \leq C\left(th^{1-2\nu}+h\right) ,   
\end{align}
for all $t\geq0$ and $h\in (0,h_0)$.
\end{corollary}
We now give a proof of Propositions~\ref{p:addition-energy-cutoff} and~\ref{p:diagonalization}.
In the proofs of this section, we simply write $\Pi = \chi_0(h\Lambda)$ and $\tilde{\Pi} = \tilde{\chi}_0(h\Lambda)$ for concision.

\bnp[Proof of Proposition~\ref{p:addition-energy-cutoff}]
First notice that Assumption~\eqref{e:asspt-positive} together with Lemma~\ref{l:bound-etip} below ensure that $i\mathsf{P}$ generates a strongly continuous semigroup $(e^{it \mathsf{P}})_{t\in \R_+}$ on $L^2(M)$ such that 
\begin{align}
\label{e:semigroup-bd}
\nor{e^{it\mathsf{P}}}{\L(L^2)} \leq e^{C_0t h^{1-2\nu}}, \quad \text{for all }t\geq 0, h \in (0,h_0).
\end{align}
Second, we remark that~\eqref{e:energy-cutoff} and~\eqref{e:energy-cutoff-bis} are equivalent. Indeed, recalling that $\Pi  \in \Psi^{-\infty}(M)$ is identified with $\Pi I_2  \in \Psi^{-\infty}(M;\C^{2\times 2})$, we have $[\Pi I_2,\Lambda I_\pm] =[\Pi ,\Lambda ]I_\pm = 0$ and, using $Q_h \in\Psi^{0}_{\nu}(M;\C^{2\times 2})$, 
\begin{align}
\label{e:comm-pi-qh}
i [\mathsf{P},\Pi] = [\Pi , Q_h]  = [\Pi I_2 , Q_h] =
 \begin{pmatrix}
[\Pi , Q_h^+ ]& [\Pi , Q_h^{\pm}]\\
[\Pi , Q_h^{\mp}] &[\Pi , Q_h^-]
\end{pmatrix}
 \in h^{1-2\nu}\Psi^{-\infty}_{\nu}(M ; \C^{2\times 2}) . 
 \end{align}
As a consequence, Lemma~\ref{l:1-BB} implies 
\begin{align*}
\nor{e^{it\mathsf{P}}\Pi  - \Pi e^{it\mathsf{P}}}{\L(L^2)} & \leq t e^{C_0t h^{1-2\nu}}  \nor{[\mathsf{P},\Pi]}{\L(L^2)} \leq C t h^{1-2\nu} e^{C_0t h^{1-2\nu}}  .
\end{align*}
Hence, we only need to prove~\eqref{e:energy-cutoff}. 
Now write $U(t) = e^{it \mathsf{P}}U_0$ for the solution of~\eqref{e:first-system-h} for $t>0$ such that $U|_{t=0}=U_0$. Applying $\Pi$ to the evolution equation~\eqref{e:first-system-h} we obtain, for $t>0$
$$
\left(- D_t \Pi+\Lambda \Pi I_\pm  + i \Pi Q_h  \right) U = 0 , 
$$
where we used $\Lambda \Pi = \Pi\Lambda$. With $\Pi = \tilde{\Pi} \Pi$, this implies 
$$
\left(-D_t \Pi +\Lambda \tilde{\Pi} I_\pm \Pi   + i \tilde{\Pi} Q_h \Pi + i \tilde{\Pi}  [\Pi , Q_h]  \right) U = 0 , 
$$
and hence 
\begin{align}
\label{e:pi-first-system-h}
\left(- D_t + \mathsf{P}_{\cut} \right) \Pi U=
\left(- D_t +\Lambda \tilde{\Pi} I_\pm    + i \tilde{\Pi} Q_h \tilde{\Pi}  \right)  \Pi U = - i \tilde{\Pi}  [\Pi , Q_h] U  .
\end{align}
where, according to~\eqref{e:comm-pi-qh}, we have  
$\tilde{\Pi}  [\Pi , Q_h] \in h^{1-2\nu}\Psi^{-\infty}_{\nu}(M ;\C^{2\times 2})$. 
Recalling the definition of $\mathsf{P}_{\cut}$ in~\eqref{e:def-P1}, and noticing that $\Lambda \tilde{\Pi} I_\pm$ is selfadjoint and $\tilde{\Pi} Q_h \tilde{\Pi}$ has principal symbol $\sigma_h(\tilde{\Pi})^2 q_h$ which satisfies $\Re(\sigma_h(\tilde{\Pi})^2 q_h) =\sigma_h(\tilde{\Pi})^2 \Re(q_h)$ and hence is a nonnegative selfadjoint family of matrices on $T^*M$, we may apply Proposition~\ref{l:bound-etip} to obtain 
$\nor{e^{it\mathsf{P}_{\cut}}}{\L(L^2)} \leq e^{C_0t h^{1-2\nu}}$ for all $t\geq 0, h \in (0,h_0)$. 
We can now solve~\eqref{e:pi-first-system-h} with the Duhamel formula as 
\begin{align*}
\Pi U(t) = e^{it\mathsf{P}_{\cut}} \Pi U_0 - \int_0^t   e^{i(t-s) \mathsf{P}_{\cut}}\tilde{\Pi}   [\Pi , Q_h] U (s) ds , \quad   U (s) = e^{i s\mathsf{P}}U_0 .
\end{align*} 
Together with~\eqref{e:semigroup-bd}, we deduce that 
\begin{align}
\label{e:premiere-reduction}
\nor{\Pi e^{it\mathsf{P}} U_0 - e^{it\mathsf{P}_{\cut}} \Pi U_0}{L^2(M;\C^2)} & \leq \int_0^t   \nor{e^{i(t-s) \mathsf{P}_{\cut}}}{\L(L^2)} \nor{\tilde{\Pi}   [\Pi , Q_h] }{\L(L^2)}
\nor{e^{is\mathsf{P}}}{\L(L^2)}\nor{U_0}{L^2(M;\C^2)} ds \nonumber \\
& \leq C t h^{1-2\nu}  e^{C_0t h^{1-2\nu}}\nor{U_0}{L^2(M;\C^2)} ,
\end{align} 
uniformly for all $t\geq0$ and $h\in (0,h_0)$. This concludes the proof of the proposition.
\enp
We next prove Proposition~\ref{p:diagonalization}, relying on the following lemma.
\begin{lemma}
Let $V(t)$ solve $(D_t - \mathsf{P}_{\cut})V = 0$ for $t>0$ and set $W(t) = (\id - K)V(t)$ where $K$ is a time-invariant bounded operator. Then, we have 
\begin{align}
\label{e:eqn-comm-W}
\left(D_t - \mathsf{P}_{\cut} + [K,\mathsf{P}_{\cut}] \right) W(t) = [\mathsf{P}_{\cut},K] K V(t) .
\end{align}
\end{lemma}
The term $[K,\mathsf{P}_{\cut}]$ is used in the proof of Proposition~\ref{p:diagonalization} as a correction term to kill the off-diagonal part of $\mathsf{P}_{\cut}$. This idea is due to Taylor~\cite[Section~2]{Taylor:75}, see also~\cite[Proof of Proposition~2.9]{LL:16}.
\bnp
This follows from the computation
\begin{align*}
\left(D_t - \mathsf{P}_{\cut} + [K,\mathsf{P}_{\cut}] \right) W(t) & = \left(D_t - \mathsf{P}_{\cut} + [K,\mathsf{P}_{\cut}] \right)  (\id - K)V(t) \\
& = (\id - K)(D_t - \mathsf{P}_{\cut}) V(t) + [D_t - \mathsf{P}_{\cut} , \id - K ]V(t)  + [K,\mathsf{P}_{\cut}]  (\id - K)V(t) \\
& = 0+ [ \mathsf{P}_{\cut} , K ]V(t)  + [K,\mathsf{P}_{\cut}]  V(t) - [K,\mathsf{P}_{\cut}] K V(t)\\
& = - [K,\mathsf{P}_{\cut}] K V(t) ,
\end{align*}
where we have used $(D_t - \mathsf{P}_{\cut})V = 0$ in the third line.
\enp

\bnp[Proof of Proposition~\ref{p:diagonalization}]
First notice that Assumption~\eqref{e:asspt-positive} implies that $\Re(q_h^+) \geq0 , \Re(q_h^-) \geq0$, and hence the principal symbol of $\begin{pmatrix}
\tilde{\Pi} Q_h^+ \tilde{\Pi} &0  \\
0 &\tilde{\Pi} Q_h^-\tilde{\Pi}
\end{pmatrix}$ has nonnegative real part. 
Together with Proposition~\ref{l:bound-etip}, this ensures that $i\mathsf{P}_{\diag}$ generates a semigroup $(e^{it \mathsf{P}_{\diag}})_{t\in \R_+}$ such that 
\begin{align}
\label{e:semigroup-bd-Pd}
\nor{e^{it\mathsf{P}_{\diag}}}{\L(L^2)} \leq e^{C_0t h^{1-2\nu}}, \quad \text{for all }t\geq 0, h \in (0,h_0).
\end{align}
Next, we decompose $\mathsf{P}_{\cut}$ in~\eqref{e:def-P1}-\eqref{e:first-system-h} into its diagonal and off-diagonal part as
$$
\mathsf{P}_{\cut} = \mathsf{P}_{\diag} + \mathsf{P}_{\off} , 
$$
with $\mathsf{P}_{\diag}$ defined in~\eqref{e:def-Pd} and 
\begin{align}
\label{e:Pad}
  \mathsf{P}_{\off} = i \begin{pmatrix}
 0  & \tilde{\Pi} Q_h^{\pm}  \tilde{\Pi} \\
 \tilde{\Pi} Q_h^{\mp}  \tilde{\Pi}& 0
\end{pmatrix}
\in \Psi^{-\infty}_{\nu} (M ;\C^{2\times 2}) . 
\end{align}
Hence, denoting by $\tilde{\Pi}\Lambda^{-1} : = h \tilde{\chi}_0(h\Lambda) (h\Lambda)^{-1} = h \frac{\tilde{\chi}_0(s)}{s}|_{s= h\Lambda}$ defined via functional calculus, we have from~\cite[Theorem~14.9]{Zworski:book} (see also Corollary~\ref{c:functional-calc} below) that $\tilde{\Pi}\Lambda^{-1}\in h \Psi^{-\infty}(M)$. We choose the operator $K$ to be plugged in~\eqref{e:eqn-comm-W} as
\begin{align}
\label{e:choice-K}
K = \frac{i}{2}
\begin{pmatrix}
0 &- \tilde{\Pi}\Lambda^{-1} Q_h^\pm  \\
 \tilde{\Pi}\Lambda^{-1} Q_h^\mp  & 0
\end{pmatrix}  \in h \Psi^{-\infty}_{\nu}(M; \C^2) .
\end{align}
With this choice of correction $K$, and using that $ \Lambda \tilde{\Pi}\tilde{\Pi}\Lambda^{-1} =\tilde{\Pi}^2$, we obtain
\begin{align}
\label{e:KP1comm}
[K, \mathsf{P}_{\cut}] & =   [K, \tilde{\Pi}\Lambda I_\pm] + i  [K, \tilde{\Pi} Q_h \tilde{\Pi}] , \quad \text{with}\\
\nonumber
 [K, \tilde{\Pi}\Lambda I_\pm] & = 
 \frac{i}{2}
\begin{pmatrix}
0 &  \tilde{\Pi}\Lambda^{-1} Q_h^\pm\Lambda \tilde{\Pi} + \tilde{\Pi}^2  Q_h^\pm  \\
 \tilde{\Pi}\Lambda^{-1} Q_h^\mp\Lambda \tilde{\Pi} + \tilde{\Pi}^2 Q_h^\mp  & 0
\end{pmatrix} , \\
\label{e:KQhcomm}
 [K, \tilde{\Pi} Q_h \tilde{\Pi}] & \in h \Psi^{-\infty}_{\nu}(M) .
\end{align}
Next, rewriting
$$[K, \tilde{\Pi}\Lambda I_\pm]  = 
 \frac{i}{2}
\begin{pmatrix}
0 & \tilde{\Pi}(h\Lambda)^{-1} Q_h^\pm(h\Lambda) \tilde{\Pi} + \tilde{\Pi}^2  Q_h^\pm  \\
 \tilde{\Pi}(h\Lambda)^{-1} Q_h^\mp(h\Lambda) \tilde{\Pi} + \tilde{\Pi}^2 Q_h^\mp  & 0
\end{pmatrix} ,
$$
we deduce from symbolic calculus that 
\begin{align}
\label{e:KLamcomm}
[K, \tilde{\Pi}\Lambda I_\pm] \in \Psi^{-\infty}_{\nu}(M; \C^2),
\end{align}
 with principal symbol
$$
\sigma_h([K, \tilde{\Pi}\Lambda I_\pm]) 
= i 
\begin{pmatrix}
0 &  \sigma_h(\tilde{\Pi})^2  q_h^\pm   \\
\sigma_h(\tilde{\Pi})^2 q_h^\mp   & 0
\end{pmatrix} = \sigma_h (\mathsf{P}_{\off}) ,
$$
where the last equality comes from~\eqref{e:Pad}. This implies (see Appendix~\ref{s:calcul-pseudo})
\begin{align}
\label{e:comm-van}
- \mathsf{P}_{\off} +[K, \tilde{\Pi}\Lambda I_\pm] \in h^{1-2\nu} \Psi^{-\infty}_{\nu}(M;\C^2).
\end{align}
As a consequence of \eqref{e:KP1comm}-\eqref{e:KQhcomm}-\eqref{e:comm-van}, we now have 
\begin{align*}
- \mathsf{P}_{\cut} + [K,\mathsf{P}_{\cut}]  & = - \mathsf{P}_{\diag} - \mathsf{P}_{\off} + [K, \tilde{\Pi}\Lambda I_\pm]  + i  [K, \tilde{\Pi} Q_h \tilde{\Pi}]  \\
& =  - \mathsf{P}_{\diag} +R_1 , \quad \text{with} \quad R_1 \in h^{1-2\nu} \Psi^{-\infty}_{\nu}(M;\C^2) .
\end{align*}
Moreover, \eqref{e:KP1comm}-\eqref{e:KQhcomm} and \eqref{e:KLamcomm} imply $[\mathsf{P}_{\cut},K] \in \Psi^{-\infty}_{\nu}(M; \C^2)$ and thus, with~\eqref{e:choice-K},
 $$R_2 := [\mathsf{P}_{\cut},K] K \in h \Psi^{-\infty}_{\nu}(M; \C^2).$$
Now, we apply~\eqref{e:eqn-comm-W}, and recalling that $W(t) = (\id - K)V(t)$, we have obtained from the above two lines that $W$ solves
$$
(D_t - \mathsf{P}_{\diag} + R_1 ) W(t) = R_2 V(t), 
$$
that is 
$$
(D_t - \mathsf{P}_{\diag} ) W(t) = R  V(t) , \quad \text{ with } R = R_2-R_1(\id - K)  \in  h^{1-2\nu}\Psi^{-\infty}_{\nu}(M; \C^2) .
$$
If we further assume that $W(0) = V(0) = V_0 \in L^2(M;\C^2)$, hence $V(t) = e^{it\mathsf{P}_{\cut}} V_0$, we now have 
$$
W(t) = e^{it\mathsf{P}_{\diag}} V_0 + i \int_0^t e^{i(t-s)\mathsf{P}_{\diag}} R e^{is\mathsf{P}_{\cut}} V_0 ds , \quad t\geq 0 .
$$
Using that $\nor{e^{it\mathsf{P}_{\cut}}}{\L(L^2)} \leq e^{C_0t h^{1-2\nu}}$ for all $t\geq 0, h \in (0,h_0)$, together with~\eqref{e:semigroup-bd-Pd}, this implies in particular
$$
\nor{W(t) - e^{it\mathsf{P}_{\diag}} V_0}{L^2(M;\C^2)} \leq t h^{1-2\nu}e^{C_0t h^{1-2\nu}} \nor{V_0}{L^2(M;\C^2)}  , \quad t\geq 0 , h \in (0,h_0).
$$
Recalling that $W(t) = (\id - K)e^{it\mathsf{P}_{\cut}} V_0$ with~\eqref{e:choice-K}, we now deduce that
$$
\nor{e^{it\mathsf{P}_{\cut}} V_0- e^{it\mathsf{P}_{\diag}} V_0}{L^2(M;\C^2)} \leq \left( t h^{1-2\nu} + Ch \right)e^{C_0t h^{1-2\nu}} \nor{V_0}{L^2(M;\C^2)}  , \quad t\geq 0, h \in (0,h_0) ,
$$
which concludes the proof of the proposition.
\enp

\section{From semiclassical resolvent estimates to damped waves}
\label{s:semiclassical-resolvent}
In this section, we consider the question of obtaining an averaged semigroup decay starting from a resolvent estimate. This question is considered in an abstract semiclassical setting, close to~\cite[Section~3]{BZ:04}. Namely, we let $\H$ be a separable Hilbert space, and we study the following $h$-dependent family operators on $\H$:
\begin{equation}
\label{e:def-P-sc}
\P_h = P_h + i h Q_h ,
\end{equation}
where $P_h$ is a selfadjoint operator.
A typical application we have in mind is the half damped wave operator $P_h = h\sqrt{-\Delta}$ and $Q_h=b$. The damped wave and Klein-Gordon equations also fit in this setting, once reformulated as in Section~\ref{s:first-order-system}.
We define the energy window
\begin{equation*}
I_\eps := [1-\eps(h), 1+\eps(h)] , \quad \text{with }0<\eps(h) \leq \eps_0 < 1 .
\end{equation*}
We let $\chi \in C^\infty_c(\R)$ with $\chi \geq0$, $\supp\chi \subset (-1,1)$ and $\chi(0)=1$, and set $\chi_\eps(s) = \chi \left((s-1)\eps(h)^{-1}\right)$, so that $\supp(\chi_\eps) \subset \Int(I_\eps)$. We then define (using functional calculus for selfadjoint operators)
 \begin{equation}
 \label{def:Pieps}
 \Pi_\eps = \chi_\eps \left( P_h \right) . 
 \end{equation}
We also consider the case in which the energy window is symmetric about $0$, namely, 
$$
I_\eps^\pm := [-1-\eps(h), -1+\eps(h)] \cup [1-\eps(h), 1+\eps(h)] , \quad
 \chi_\eps^\pm(s) = \chi_\eps(s)+\chi_\eps(-s), \quad   \Pi_\eps^\pm = \chi_\eps^\pm \left( P_h \right) .
$$
Note that the link with the damped Klein-Gordon/wave operators (as reformulated in Section~\ref{s:first-order-system}) is explained in~\eqref{e:def-Ph-encore}-\eqref{e:defPipm} below.

We first formulate a general result in an abstract framework in Section~\ref{s:bounded-damping-general}. We then improve the remainder term assuming a parametrix is given in Section~\ref{s:remainder-regularity}, and finally construct the sought parametrix in the case of the Klein-Gordon/wave equations (using pseudodifferential calculus) in Section~\ref{s:parametrice}. 

\subsection{Functional analysis: bounded damping term}
\label{s:bounded-damping-general}
Note that, under the above assumptions, for any $h>0$, $\P_h$ is a bounded perturbation of a selfadjoint operator such that $\Im(\P_h z, z)_\H = h\Re(Q_h z,z)_\H \geq 0$. Hence,  the operator $\frac{i}{h}\P_h$ generates a contraction semigroup (see \eg~\cite{Pazy:83}), which we denote $\left(e^{\frac{it}{h}\P_h}\right)_{t\in \R_+}$.
The main result of this section links a resolvent estimate for $\P_h$ in $I_\eps$ together with time-averages of the semigroup $e^{\frac{it}{h}\P_h}$ for data spectrally localized in the window $I_\eps$ or $I_\eps^\pm$.
\begin{theorem}
\label{t:resolvent-implies-long-time}
Assume that $\P_h$ is defined by~\eqref{e:def-P-sc} for $h \in (0,h_0)$, with $P_h$ a family of selfadjoint operators with $h$-independent domains $D(P_h) \subset \H$, and $Q_h$ a family of bounded operators on $\H$ such that $\Re(Q_h v,v)_\H \geq 0$ for all $v\in \H$.
Then, there is a constant $\mathsf{C}_0$ (depending only on $\chi$) such that for all $\psi \in H^1_{\comp}(\R)$ with $\supp \psi \subset \R_+$, $G(h)\in (0,+ \infty]$, $T_h \geq 1$, $\eps(h)\in (0,\eps_0]$, if  
\begin{equation}
\label{e:hyp-resP}
I_\eps \subset \rho(\P_h) \quad \text{and} \quad  
\nor{(\P_h -\tau)^{-1}}{\L(\H)} \leq \frac{G(h)}{h} , \quad \text{for all } \tau \in I_\eps \text{ and }h \in (0,h_0) ,
\end{equation}
 then, for all $h \in (0,h_0)$ and $u\in \H$, 
 \begin{align}
 \label{e:estim-}
\frac{1}{T_h} \int_\R \left|\psi\left( \frac{t}{T_h}\right) \right|^2 \nor{e^{\frac{it}{h}\P_h} \Pi_\eps u}{\H}^2  dt
  \leq \left( \frac{G(h)^2}{T_h^2} \nor{\psi'}{L^2(\R)}^2  + C_{Q,\psi}  \frac{h T_h}{\eps(h)^2}\right)  \nor{ \Pi_\eps u }{\H}^2 , \\
 \label{e:estim-pointwise}
\nor{e^{\frac{i \theta T_h }{h}\P_h} \Pi_\eps u}{\H}^2
\leq \left( \frac{G(h)^2}{T_h^2}\frac{\nor{\psi'}{L^2(\R)}^2}{\nor{\psi}{L^2(0,\theta)}^2} +\frac{C_{Q,\psi}}{\nor{\psi}{L^2(0,\theta)}^2}\frac{h T_h}{\eps(h)^2} \right) \nor{\Pi_\eps u}{\H}^2,\quad\text{for all }\theta>0,\\
    \label{e:estim-pointwise-opt}
\nor{e^{\frac{i T_h }{h}\P_h} \Pi_\eps u}{\H}^2 
  \leq \left(  (2+\delta)\frac{G(h)^2}{T_h^2}   + \frac{C_Q}{\delta^2}  \frac{h T_h}{\eps(h)^2} \right) \nor{ \Pi_\eps u }{\H}^2 , \quad \text{for all }\delta \in (0,1),
\end{align}
 with
  $$
 C_{Q,\psi} =  \mathsf{C}_0(1 +\|Q_h\|_{\L(\H)}^2) \max \{\nor{\psi}{L^1}^2 , \nor{\psi'}{L^1}^2 \} , \quad  C_{Q} =  \mathsf{C}_0(1 +\|Q_h\|_{\L(\H)}^2) .
 $$
 Moreover, the same result holds if we replace $I_\eps, \chi_\eps, \Pi_\eps$ by $I_\eps^\pm,\chi_\eps^\pm, \Pi_\eps^\pm$ respectively.
 \end{theorem}

\begin{remark}
The remainder term depends on $h$ as $\frac{h T_h}{\eps(h)^2}$. In case the energy window is fixed, namely $\eps(h)=\eps_0$, it remains bounded for $T_h \lesssim \frac{1}{h}$.
\end{remark}

In this paragraph (and in particular in the proof below), we use the semiclassical (unitarily normalized) Fourier transform, defined for $f \in \calS(\R)$ by
\begin{equation}
\label{e:Fourier-sc}
\hat{f}(\tau) = \F_{t\to \tau}(f)(\tau) =\frac{1}{\sqrt{2\pi h}} \int_\R e^{-\frac{i\tau t}{h}} f(t) dt , \quad \text{ whence }\quad  f(t) = \frac{1}{\sqrt{2\pi h}} \int_\R e^{\frac{it \tau}{h}} \hat{f}(\tau) d\tau  .
\end{equation}
Note that $\F_{t\to \tau}\left(\frac{h}{i}\d_t f\right)(\tau) = \tau \hat{f}(\tau)$ and $\nor{f}{L^2(\R)} = \nor{\hat{f}}{L^2(\R)}$.

\medskip
The proof of Theorem~\ref{t:resolvent-implies-long-time} relies on a ``Plancherel-in-time'' argument as in~\cite[Theorem~4]{BZ:04} or in classical proofs of the Gearhart-Huang-Pr\"uss theorem~\cite{Gearhart:78,Huang:85,Pruss:84,EN:book,HS:10}.

\bnp[Proof of Theorem~\ref{t:resolvent-implies-long-time}]
From $u \in \H$, we define $v (t) =e^{\frac{it}{h}\P_h} \Pi_\eps u$. Since $\Pi_\eps u \in D(\P_h) = D(P_h)$, we have $v \in C^1(\R_+; \H)\cap C^0(\R_+; D(P_h))$, see \eg~\cite{Pazy:83}. We also set 
$$
w(t) =\psi\left(\frac{t}{T_h}\right)v (t) , \quad \text{ and } \quad g (t) = \psi' \left(\frac{t}{T_h}\right)v (t) ,
$$ and extend $w(t), g (t)$ by $0$ for $t\leq 0$. Since $\psi =0$ on $(-\infty, 0]$, the extended functions then satisfy $w\in H^1_{\comp}(\R; \H)$, $g\in L^2_{\comp}(\R; \H)$. Moreover, we have 
$$
\left(\frac{h}{i}\d_t - \P_h\right)v (t)=0, \quad \text{ for } t\in \R_+^*, 
$$
and hence
\begin{equation}
\label{e:eq-trunk-time}
\left(\frac{h}{i}\d_t - \P_h\right)w (t)= \frac{h}{iT_h} g(t) \quad \text{ in } L^2_{\comp}(\R; \H) . 
\end{equation}
We take the semiclassical Fourier transform in time~\eqref{e:Fourier-sc} of~\eqref{e:eq-trunk-time} to obtain 
$$
 \left(\tau - \P_h\right)\hat{w} (\tau)= \frac{h}{iT_h}\hat{g} (\tau) ,\quad \tau \in \R .
$$

Our goal is to prove an estimate on $\int_\R \|\hat{w}(\tau)\|_{\H}^2 d\tau$, and then use the Plancherel Theorem. To this aim, we decompose the integral over $\tau \in \R$ into $\tau \in I_\eps$ (for which we use Assumption~\eqref{e:hyp-resP}) and $\tau \notin I_\eps$ (for which we use the invertibility of $(\P_h-\tau)\Pi_\eps$).

We first consider the case $\tau \in I_\eps$ and obtain a bound on $\int_{I_\eps} \|\hat{w}(\tau)\|_{\H}^2 d\tau$.
For this, remark that if  $\tau \in I_\eps$, then Assumption~\eqref{e:hyp-resP} yields $\tau \notin \Sp(\P_h)$ and we have 
\begin{equation*}
\hat{w} (\tau)= \frac{h}{iT_h}   \left(\tau - \P_h \right)^{-1} \hat{g} (\tau) .
\end{equation*}
With~\eqref{e:hyp-resP}, this implies, \begin{align}
\label{e:key-res-tau}
\nor{\hat{w} (\tau)}{\H} & \leq  \frac{h}{T_h} \nor{\left(\tau - \P_h\right)^{-1}}{\L(\H)} \nor{ \hat{g} (\tau)}{\H} 
\leq\frac{G(h)}{T_h}\nor{ \hat{g} (\tau)}{\H} ,\quad \text{ for all }\tau \in I_\eps.
\end{align}
Integrating~\eqref{e:key-res-tau} yields
\begin{align}
\label{e:estim-wz}
\int_{I_\eps} \|\hat{w}(\tau)\|_{\H}^2 d\tau \leq  \frac{G(h)^2}{T_h^2}  \int_{I_\eps} \nor{\hat{g} (\tau)}{\H}^2 d\tau  \leq  \frac{G(h)^2}{T_h^2}  \int_\R\nor{\hat{g} (\tau)}{\H}^2 d\tau 
= \frac{G(h)^2}{T_h^2}  \int_\R \nor{g(t)}{\H}^2 dt  ,
\end{align}
where we used the vector-valued Plancherel theorem~\cite[Theorem~1.8.2]{ABHN:book} ($\F$ is well-normalized).
Recalling that $\left(e^{\frac{it}{h}\P_h}\right)_{t\in \R_+}$ is a $C^0$ semigroup of contractions, we have $\nor{v(t)}{\H}=  \nor{e^{\frac{it}{h}\P_h} \Pi_\eps u}{\H}\leq \nor{\Pi_\eps u}{\H}$, and hence
\begin{align}
\label{e:plangherel-gtau}
 \int_{\R} \|g(t)\|_{\H}^2 dt \leq 
 \int_\R \left|\psi'\left(\frac{t}{T_h}\right)\right|^2 dt \nor{\Pi_\eps u}{\H}^2 
=  T_h  \nor{\psi'}{L^2(\R)}^2 \nor{\Pi_\eps u}{\H}^2 .
\end{align}
With~\eqref{e:estim-wz} and~\eqref{e:plangherel-gtau}, we have obtained 
\begin{align}
\label{e:estim-wz-Ieps}
\int_{I_\eps} \|\hat{w}(\tau)\|_{\H}^2 d\tau 
 \leq  \frac{G(h)^2}{T_h}\nor{\psi'}{L^2(\R)}^2\nor{\Pi_\eps u}{\H}^2 .
\end{align}

We now consider the case $\tau \in \R \setminus I_\eps$ and obtain a bound on $\int_{\R \setminus I_\eps} \|\hat{w}(\tau)\|_{\H}^2 d\tau$.
Let $\alpha\in (0,1)$ such that $\supp(\chi) \subset [-(1-\alpha) , 1-\alpha]$. Then, by construction of $\chi_\eps$ and definition of $I_\eps$, we have 
$$
s \in \supp(\chi_\eps) \text{ and } \tau \notin  I_\eps \implies |s-\tau | \geq \alpha \eps(h).
$$
Moreover, we have 
$$
s \in \supp(\chi_\eps) \text{ and } \tau \in \R \implies |s-\tau | \geq \big(|\tau| -(1+\eps(h))\big)_+.
$$
As a consequence, we have in particular 
\begin{align*}
s \in \supp(\chi_\eps) \text{ and }\tau \notin  I_\eps \implies |s-\tau |&  \geq \max \left\{ \alpha \eps(h) , \big(|\tau| -(1+\eps(h))\big)_+ \right\} \\
& \geq \frac{\alpha \eps(h)}{2 + (1+\alpha)\eps(h)}(1+|\tau|) \geq \frac{\alpha \eps(h)}{4} \langle\tau\rangle ,
\end{align*}
and hence
$$
\chi_\eps(s)\mathds{1}_{\R\setminus I_\eps}(\tau) |s-\tau| \geq \chi_\eps(s)\mathds{1}_{\R\setminus I_\eps}(\tau) \frac{\alpha \eps(h)}{4} \langle\tau\rangle , \quad (s, \tau ) \in \R^2 .
$$
Together with the spectral theorem for selfadjoint operators, this gives, for all $y \in D(P_h)$,
$$
\nor{\mathds{1}_{\R\setminus I_\eps}(\tau) (P_h-\tau) \Pi_\eps y}{\H}  \geq \mathds{1}_{\R\setminus I_\eps}(\tau) \frac{\alpha \eps(h)}{4} \langle\tau\rangle \nor{ \Pi_\eps y}{\H}, \quad (s, \tau ) \in \R^2 .
$$
This implies that $(P_h-\tau)$ is invertible on $\Pi_\eps \H$, together with
\begin{equation}
\label{e:estim-elliptic}
\nor{(P_h-\tau)^{-1} \Pi_\eps y }{\H}  \leq \frac{4}{\alpha \eps(h) \langle\tau\rangle}  \nor{ \Pi_\eps y }{\H}, \quad \tau \in  \R\setminus I_\eps.
\end{equation}
Now, we recall the definition of the semiclassical Fourier transform~\eqref{e:Fourier-sc} to write 
$$
\sqrt{2\pi h} \ \hat{w}(\tau) =  \int_\R e^{-\frac{i\tau t}{h}} w(t) dt = \int_\R e^{-\frac{i\tau t}{h}} \psi(t/T_h) e^{\frac{it}{h}\P_h} \Pi_\eps u \, dt = \int_\R  \psi(t/T_h) e^{\frac{it}{h}(\P_h-\tau)} \Pi_\eps u \, dt .
$$
We now let $\bar{\Pi}_\eps = \bar{\chi}_\eps(P_h)$ (see~\eqref{def:Pieps}) with $\bar{\chi}$ defined as $\chi$ and such that  $\bar{\chi} \chi =\chi$. In particular, we have $\bar{\Pi}_\eps \Pi_\eps =\Pi_\eps$. 
We next write
\begin{equation}
\label{e:split-param}
(\P_h - \tau) S_1 = \bar{\Pi}_\eps + R_1 , \quad \text{ with } \quad   S_1 = (P_h-\tau)^{-1}  \bar{\Pi}_\eps, \quad R_1  = ih Q_h (P_h-\tau)^{-1} \bar{\Pi}_\eps.
\end{equation}
We decompose accordingly
\begin{align}
\label{e:T1+T2}
&\sqrt{2\pi h} \ \hat{w}(\tau)  = \int_\R  \psi(t/T_h) e^{\frac{it}{h}(\P_h-\tau)} \bar{\Pi}_\eps  \Pi_\eps u \, dt  = T_1(\tau) - T_2 (\tau), \quad \text{with }\\
&T_1(\tau) =  \int_\R  \psi(t/T_h) e^{\frac{it}{h}(\P_h-\tau)} (\P_h-\tau) S_1 \Pi_\eps u \, dt, \qquad T_2(\tau) = \int_\R  \psi(t/T_h) e^{\frac{it}{h}(\P_h-\tau)} R_1  \Pi_\eps u \, dt , \nonumber
\end{align}
and we study the two terms separately.
Concerning the first term, we remark that $(\P_h-\tau) e^{\frac{it}{h}(\P_h-\tau)}= \left(\frac{h}{i}\d_t\right) e^{\frac{it}{h}(\P_h-\tau)}$. Hence, for $\tau \in \R$, using an integration by parts, we obtain\begin{align*}
T_1(\tau) & =  \int_\R  \psi(t/T_h) \left(\frac{h}{i}\d_t\right) e^{\frac{it}{h}(\P_h-\tau)} S_1 \Pi_\eps u \, dt  = \int_\R \left[ \left(-\frac{h}{i}\d_t\right) \psi(t/T_h) \right] e^{\frac{it}{h}(\P_h-\tau)} S_1 \Pi_\eps u \, dt \\
 & = \left(\frac{ih}{T_h}\right)  \int_\R \psi'(t/T_h)e^{\frac{it}{h}(\P_h-\tau)} S_1 \Pi_\eps u \, dt .
\end{align*}
We thus have 
\begin{align*}
\nor{T_1(\tau)}{\H} & \leq \frac{h}{T_h} \int_\R |\psi'(t/T_h)| dt   \nor{S_1 \Pi_\eps u}{\H}  = h  \nor{\psi'}{L^1(\R)}  \nor{S_1 \Pi_\eps u}{\H} .
\end{align*}
Recalling $S_1 = (P_h-\tau)^{-1}\bar{\Pi}_\eps$, together with~\eqref{e:estim-elliptic}, yields, for $\tau \in  \R\setminus I_\eps$,
$$
\nor{S_1 \Pi_\eps u}{\H} = \nor{(P_h-\tau)^{-1} \Pi_\eps u }{\H}  \leq \frac{4}{\alpha \eps(h) \langle\tau\rangle}  \nor{ \Pi_\eps u }{\H} .
$$
We hence obtain
\begin{align}
\label{e:estim-T1}
\nor{T_1(\tau)}{\H} \leq \frac{4h}{\alpha \eps(h) \langle\tau\rangle} \nor{\psi'}{L^1(\R)}    \nor{ \Pi_\eps u }{\H} .
\end{align}
We estimate $T_2$ roughly as 
$$
\nor{T_2(\tau)}{\H} \leq  \int_\R  |\psi(t/T_h) | dt \nor{R_1  \Pi_\eps u}{\H} = T_h \|\psi \|_{L^1(\R)}  \nor{R_1  \Pi_\eps u}{\H} ,
$$
Recalling the definition of $R_1 = ih Q_h (P_h-\tau)^{-1}\bar{\Pi}_\eps$ together with~\eqref{e:estim-elliptic}, we have, for $\tau \in  \R\setminus I_\eps$,
 $$
 \nor{R_1  \Pi_\eps u}{\H} \leq h\|Q_h\|_{\L(\H)}  \nor{(P_h-\tau)^{-1}\Pi_\eps u }{\H} \leq h\|Q_h\|_{\L(\H)} \frac{4}{\alpha \eps(h) \langle\tau\rangle}  \nor{ \Pi_\eps u }{\H} ,
 $$
and hence
\begin{align}
\label{e:estim-T2}
\nor{T_2(\tau)}{\H} \leq  4/\alpha \|\psi \|_{L^1(\R)}  \|Q_h\|_{\L(\H)} \frac{hT_h}{\eps(h) \langle\tau\rangle}  \nor{ \Pi_\eps u }{\H} .
\end{align}
Combining~\eqref{e:T1+T2},~\eqref{e:estim-T1} and~\eqref{e:estim-T2}, we have obtained, for $\tau \in  \R\setminus I_\eps$,
$$
\sqrt{2\pi h} \nor{\hat{w}(\tau) }{\H} \leq 4/\alpha \nor{\psi'}{L^1(\R)}  \frac{h}{\eps(h) \langle\tau\rangle}  \nor{ \Pi_\eps u }{\H}  +4/\alpha \|\psi \|_{L^1(\R)} \|Q_h\|_{\L(\H)} \frac{hT_h}{\eps(h) \langle\tau\rangle}  \nor{ \Pi_\eps u }{\H} ,
$$
that is, using $T_h \geq 1$,
$$
\nor{\hat{w}(\tau) }{\H} \leq \bar{C}_{Q,\psi}  \frac{\sqrt{h} T_h}{\eps(h) \langle\tau\rangle}  \nor{ \Pi_\eps u }{\H} , \quad \tau \in  \R\setminus I_\eps .
$$
with $\bar{C}_{Q,\psi} = \frac{4}{\alpha \sqrt{2\pi}} (1+ \|Q_h\|_{\L(\H)}) \max\{\nor{\psi}{L^1(\R)} ,\nor{\psi'}{L^1(\R)} \}$.
We thus have 
\begin{align}
\label{e:final-I-eps-comp}
\int_{\R\setminus I_\eps }\nor{\hat{w}(\tau) }{\H}^2 d\tau 
& \leq  \int_{\R\setminus I_\eps } \bar{C}_{Q,\psi}^2  \frac{h T_h^2}{\eps(h)^2 \langle\tau\rangle^2}  \nor{ \Pi_\eps u }{\H}^2 d\tau  \leq \pi \bar{C}_{Q,\psi}^2 \frac{h T_h^2}{\eps(h)^2}  \nor{ \Pi_\eps u }{\H}^2 .
\end{align}
We finally combine~\eqref{e:estim-wz-Ieps} and \eqref{e:final-I-eps-comp} to obtain 
\begin{align*}
\int_{\R}\nor{\hat{w}(\tau)}{\H}^2 d\tau   
& \leq  \left(  \frac{G(h)^2}{T_h} \nor{\psi'}{L^2(\R)}^2 
+ \pi \bar{C}_{Q,\psi}^2 \frac{h T_h^2}{\eps(h)^2} \right) \nor{ \Pi_\eps u }{\H}^2
\end{align*}
The vector-valued Plancherel theorem finally yields 
\begin{align*}
\int_\R \left|\psi\left( \frac{t}{T_h}\right) \right|^2 \nor{e^{\frac{it}{h}\P_h} \Pi_\eps u}{\H}^2  dt
&  = \int_{\R}\nor{w(t)}{\H}^2 dt  = \int_{\R}\nor{\hat{w}(\tau)}{\H}^2 d\tau   \\
& \leq  \left(  \frac{G(h)^2}{T_h} \nor{\psi'}{L^2(\R)}^2 
+ \pi \bar{C}_{Q,\psi}^2 \frac{h T_h^2}{\eps(h)^2} \right) \nor{ \Pi_\eps u }{\H}^2 ,
\end{align*}
 and hence  the sought result~\eqref{e:estim-} when dividing by $T_h$.
Then,~\eqref{e:estim-pointwise} follows from~\eqref{e:estim-} combined with~\eqref{e:debile-2} in Lemma~\ref{l:debile} and the fact that $\left(e^{\frac{it}{h}\P_h}\right)_{t\in \R_+}$ is a contraction semigroup.
Finally,~\eqref{e:estim-pointwise-opt} comes from taking $\theta=1$ in~\eqref{e:estim-pointwise} and optimizing in the function $\psi$. That is to say,  for $\delta \in (0,1)$, we choose $\psi=\psi_{\min,L}$ from Lemma~\ref{l:optimizing-psi} with $L := 1+ \delta^{-1}+ \sqrt{\delta^{-2}+ 2\delta^{-1}}>2+\sqrt{3}$ so that $\frac{2L}{(L-1)^2}=\delta$. As a consequence of this choice, $ \nor{\psi_{\min,L}'}{L^2(\R)}^2 = 2 \left( 1+ \frac{L}{(L-1)^2}\right)=2+\delta$, $\nor{\psi_{\min,L}'}{L^1(\R)} = \sqrt{2} \left(1+ \frac{L}{L-1}\right) \leq 2\sqrt{2}$ and $\nor{\psi_{\min,L}}{L^1(\R)} = \frac{L}{\sqrt{2}} \leq 2^{-1/2}(2+ \sqrt{2})\delta^{-1}$  uniformly  for $\delta \in (0,1)$.
\enp

\subsection{Improved remainder with regularity assumptions}
\label{s:remainder-regularity}
The following is a variant of Theorem~\ref{t:resolvent-implies-long-time} in which we assume a parametrix for $(\P_h -\tau)^{N}$, $\tau \notin I_\eps$, and improve the remainder term.
\begin{theorem}
\label{t:resolvent-implies-long-time-smooth}
Under the assumptions of Theorem~\ref{t:resolvent-implies-long-time}, we suppose further that $\eps$ is $h$-independent and $\psi \in C^\infty_c(0,+\infty)$. Let $N\in \N^*$ and assume that there are $C_N>0$ and $S_N, R_N \in \L(\H)$ such that 
\begin{align}
\label{param-alg-N-bis-1}
(\P_h -\tau)^{N} S_N = \bar\Pi_\eps^\pm + R_N, \quad \text{with} \\
\label{param-alg-N-bis-2}  \nor{S_N}{\L(\H)}  \leq \frac{C_N}{\langle \tau \rangle} , \quad \nor{R_N}{\L(\H)}  \leq  C_N  \frac{h^N}{\langle \tau \rangle} , \quad \text{ for all } h \in (0,1) , \tau \in \R\setminus I_\eps,
\end{align}
where $\bar\Pi_\eps^\pm = \bar{\chi}_\eps^\pm(P_h)$ (see~\eqref{def:Pieps}) with $\bar{\chi}^\pm$ defined as $\chi^\pm$ and such that  $\bar{\chi}^\pm \chi^\pm =\chi^\pm$. 
If~\eqref{e:hyp-resP} holds, then for all $h \in (0,h_0)$ and $u\in \H$, we have 
\begin{align*}
\frac{1}{T_h}\int_\R \left|\psi\left( \frac{t}{T_h}\right) \right|^2 \nor{e^{\frac{it}{h}\P_h} \Pi_\eps^\pm u}{\H}^2  dt
 \leq  \left(  \frac{G(h)^2}{T_h^2} \nor{\psi'}{L^2(\R)}^2 
+ \mathsf{C}_{N,\psi}  h^{2N-1} T_h \right) \nor{ \Pi_\eps^\pm u }{\H}^2 , \\
\nor{e^{\frac{i \theta T_h }{h}\P_h} \Pi_\eps^\pm u}{\H}^2 
  \leq \left( \frac{G(h)^2}{T_h^2} \frac{\nor{\psi'}{L^2(\R)}^2}{\nor{\psi}{L^2(0,\theta)}^2}   + \frac{\mathsf{C}_{N,\psi} }{\nor{\psi}{L^2(0,\theta)}^2}    h^{2N-1} T_h \right) \nor{ \Pi_\eps^\pm u }{\H}^2 , \quad \text{for all }\theta>0, \\
  \nor{e^{\frac{i T_h }{h}\P_h} \Pi_\eps^\pm u}{\H}^2 
  \leq \left(  (2+\delta)\frac{G(h)^2}{T_h^2}   +C_{N,\delta} h^{2N-1} T_h  \right) \nor{ \Pi_\eps^\pm u }{\H}^2 , \quad \text{for all }\delta \in (0,1),
\end{align*}
with $\mathsf{C}_{N,\psi} =  C_N^2 \left(\nor{\psi^{(N)}}{L^1(\R)} + \nor{\psi}{L^1(\R)} \right)^2$.
\end{theorem}

\bnp[Proof of Theorem~\ref{t:resolvent-implies-long-time-smooth}]
We follow the proof of Theorem~\ref{t:resolvent-implies-long-time} until \eqref{e:split-param}. Instead of~\eqref{e:split-param}, we use the refined Parametrix of Assumption~\eqref{param-alg-N-bis-1}. 
We decompose 
\begin{align}
\label{e:T1+T2-bis}
&\sqrt{2\pi h} \ \hat{w}(\tau)  = \int_\R  \psi(t/T_h) e^{\frac{it}{h}(\P_h-\tau)} \bar\Pi_\eps^\pm \Pi_\eps^\pm u \, dt  = T_1(\tau) - T_2 (\tau), \quad \text{with }\\
&T_1(\tau) =  \int_\R  \psi(t/T_h) e^{\frac{it}{h}(\P_h-\tau)} (\P_h-\tau)^N S_N \Pi_\eps^\pm u \, dt, \qquad T_2(\tau) = \int_\R  \psi(t/T_h) e^{\frac{it}{h}(\P_h-\tau)} R_N  \Pi_\eps^\pm u \, dt , \nonumber
\end{align}
and we study the two terms separately.

Concerning the first term $T_1(\tau)$, we remark that $(\P_h-\tau)^N e^{\frac{it}{h}(\P_h-\tau)}= \left(\frac{h}{i}\d_t\right)^N e^{\frac{it}{h}(\P_h-\tau)}$. Hence, for $\tau \in \R$, integrating by parts, we obtain
\begin{align*}
T_1(\tau) & =  \int_\R  \psi(t/T_h) \left(\frac{h}{i}\d_t\right)^N e^{\frac{it}{h}(\P_h-\tau)} S_N \Pi_\eps^\pm u \, dt
 = \int_\R \left[ \left(-\frac{h}{i}\d_t\right)^N \psi(t/T_h) \right] e^{\frac{it}{h}(\P_h-\tau)} S_N \Pi_\eps^\pm u \, dt \\
 & = \left(\frac{ih}{T_h}\right)^N  \int_\R \psi^{(N)}(t/T_h)e^{\frac{it}{h}(\P_h-\tau)} S_N \Pi_\eps^\pm u \, dt .
\end{align*}
Therefore, we have 
\begin{align*}
\nor{T_1(\tau)}{\H} & \leq \left(\frac{h}{T_h}\right)^N \int_\R |\psi^{(N)}(t/T_h)| dt   \nor{S_N \Pi_\eps^\pm u}{\H}  =  \left(\frac{h}{T_h}\right)^N T_h   \nor{\psi^{(N)}}{L^1(\R)}  \nor{S_N \Pi_\eps^\pm u}{\H} .
\end{align*}
Using~\eqref{param-alg-N-bis-2}, we deduce for all $\tau \in \R\setminus I_\eps$
\begin{align}
\label{e:estim-T1-bis}
\nor{T_1(\tau)}{\H} \leq \frac{C_N}{\langle\tau\rangle} \left(\frac{h}{T_h}\right)^N T_h   \nor{\psi^{(N)}}{L^1(\R)}  \nor{\Pi_\eps^\pm u}{\H}   .
\end{align}
We estimate $T_2$ roughly for  $\tau \in  \R\setminus I_\eps$, using~\eqref{param-alg-N-bis-2}: 
\begin{align}
\label{e:estim-T2-bis}
\nor{T_2(\tau)}{\H} \leq  \int_\R  |\psi(t/T_h) | dt \nor{R_N  \Pi_\eps^\pm u}{\H} = T_h \|\psi \|_{L^1(\R)}  \nor{R_N  \Pi_\eps^\pm u}{\H} \leq  \frac{C_N}{\langle\tau\rangle} h^N T_h   \nor{\psi}{L^1(\R)}  \nor{\Pi_\eps^\pm u}{\H}  .
\end{align}
Combining~\eqref{e:T1+T2-bis},~\eqref{e:estim-T1-bis} and~\eqref{e:estim-T2-bis}, we have obtained, for $\tau \in  \R\setminus I_\eps$,
$$
\sqrt{2\pi h} \nor{\hat{w}(\tau) }{\H} \leq \frac{C_N}{\langle\tau\rangle}T_h  \left( \left(\frac{h}{T_h}\right)^N  \nor{\psi^{(N)}}{L^1(\R)}  +  h^N \nor{\psi}{L^1(\R)}   \right) \nor{\Pi_\eps^\pm u}{\H} ,
$$
that is, using $T_h \geq 1$,
$$
\nor{\hat{w}(\tau) }{\H} \leq \frac{C_{N,\psi}}{\langle\tau\rangle}   h^{N-1/2} T_h \nor{ \Pi_\eps^\pm u }{\H} , \quad \tau \in  \R\setminus I_\eps .
$$
with $C_{N,\psi} =C_N \left(\nor{\psi^{(N)}}{L^1(\R)} + \nor{\psi}{L^1(\R)} \right)$.
We thus have 
\begin{align}
\label{e:final-I-eps-comp-bis}
\int_{\R\setminus I_\eps }\nor{\hat{w}(\tau) }{\H}^2 d\tau 
& \leq  \int_{\R\setminus I_\eps } \frac{C_{N,\psi}^2}{\langle\tau\rangle^2}   h^{2N-1} T_h^2  \nor{ \Pi_\eps^\pm u }{\H}^2 d\tau  \leq \pi C_{N,\psi}^2 h^{2N-1} T_h^2 \nor{ \Pi_\eps^\pm u }{\H}^2 .
\end{align}
We finally combine~\eqref{e:estim-wz-Ieps} and \eqref{e:final-I-eps-comp-bis} to obtain 
\begin{align*}
\int_{\R}\nor{\hat{w}(\tau)}{\H}^2 d\tau   
& \leq  \left(  \frac{G(h)^2}{T_h} \nor{\psi'}{L^2(\R)}^2 
+  \pi C_{N,\psi}^2 h^{2N-1} T_h^2 \right) \nor{ \Pi_\eps^\pm u }{\H}^2.
\end{align*}
The vector-valued Plancherel theorem finally yields 
\begin{align*}
\int_\R \left|\psi\left( \frac{t}{T_h}\right) \right|^2 \nor{e^{\frac{it}{h}\P_h} \Pi_\eps^\pm u}{\H}^2  dt
&  = \int_{\R}\nor{w(t)}{\H}^2 dt  = \int_{\R}\nor{\hat{w}(\tau)}{\H}^2 d\tau   \\
& \leq  \left(  \frac{G(h)^2}{T_h} \nor{\psi'}{L^2(\R)}^2 
+  \pi C_{N,\psi}^2 h^{2N-1} T_h^2 \right) \nor{ \Pi_\eps^\pm u }{\H}^2,
\end{align*}
 and hence  the sought result when dividing by $T_h$.
\enp

\subsection{Parametrices in the case of smooth damping}
\label{s:parametrice}
We now explain how to construct the parametrix in~\eqref{param-alg-N-bis-1}-\eqref{param-alg-N-bis-2} in the case of the wave/Klein-Gordon equation, using pseudodifferential calculus (what we have tried to avoid so far).
Throughout the section, the energy window is assumed to be of fixed width and symmetric, \ie, for some $0< \eps_0 < 1$,
\begin{align}
\label{e:defIpm}
I_\eps^\pm := [-1-\eps_0, -1+\eps_0] \cup [1-\eps_0, 1+\eps_0] , \quad  \chi_\eps^\pm(s) = \chi_\eps(s)+\chi_\eps(-s) . 
\end{align}
We further assume (see Section~\ref{s:first-order-system}) that
\begin{align}
\label{e:def-Ph-encore}
P_h = \begin{pmatrix}
h\Lambda_m & 0\\
0 & - h\Lambda_m \end{pmatrix} ,
\end{align}
with $m\geq0$, and that 
\begin{align}
\label{e:defPipm}\Pi_\eps^\pm = \chi_\eps^\pm(P_h)
=\begin{pmatrix}
\chi_\eps^\pm(h\Lambda_m) & 0 \\
0 & \chi_\eps^\pm(-h\Lambda_m)
\end{pmatrix}
= \begin{pmatrix}
\chi_\eps(h\Lambda_m) & 0 \\
0 & \chi_\eps(h\Lambda_m)
\end{pmatrix}
=\chi_\eps(h\Lambda_m)  I_2 ,
\end{align}
since $\chi_\eps^\pm(-s) = \chi_\eps^\pm(s)$ even. 
\begin{definition}
\label{d:class-tau}
For the purposes of this section, we introduce, given a subset $J$ of $\R$, the family $S^{m}(\langle \tau \rangle^k, T^*M; \C^{2\times2})$ of matrices of symbols whose entries depend on two parameters $h \in (0,1),\tau \in J$ and satisfy, in coordinates 
\bnan
\label{e:symb-estim-Srho}
|\d_x^\alpha \d_\xi^\beta a(x,\xi , h,\tau)| \leq C_{\alpha,\beta} \langle \tau \rangle^k  \langle \xi \rangle^{m-|\beta|} , \quad \text{uniformly for } h \in (0,1)  ,\tau \in J.
\enan
 We denote by $\Psi^{m}(\langle \tau \rangle^k, M; \C^{2\times2})$ (or $\Psi^{m}(\langle \tau \rangle^k)$ for short) the associated class of pseudodifferential operators. We remark that $S^{m}(\langle \tau \rangle^k, T^*M; \C^{2\times2}) = \langle \tau \rangle^k S^{m}(1, T^*M; \C^{2\times2})$, where $S^{m}(1, T^*M; \C^{2\times2})$ consists in usual symbols in $S^{m}(T^*M; \C^{2\times2})$ depending on $\tau$ as a parameter, with all seminorms uniformly bounded in $\tau\in J$. Hence, the pseudodifferential calculus of $\Psi^{m}(\langle \tau \rangle^k, M; \C^{2\times2})$ is the same as that of $\langle \tau \rangle^k \Psi^{m}( M; \C^{2\times2})$. Note that the gain in the calculus is $h\langle \xi \rangle^{-1}$ (\ie there is no gain in terms of powers of $\tau$).
 \end{definition}
 Our main goal here is the construction of the following parametrix, to be combined later on with Theorem~\ref{t:resolvent-implies-long-time-smooth}.
 \begin{proposition}
\label{p:parametrix}
Assume that $Q_h \in \Psi^{0}( M; \C^{2\times2})$, $P_h$ is defined by~\eqref{e:def-Ph-encore}, and $\P_h = P_h + ihQ_h$. Then, with $I_\eps^\pm,\Pi_\eps^\pm$ defined in~\eqref{e:defIpm}-\eqref{e:defPipm}, for all $N,N' \in \N^*$ there exist $\mathsf{S}_{N'}\in \Psi^{-\infty}(\langle \tau \rangle^{-N})$ and $\mathsf{R}_{N'} \in \Psi^{-\infty}(h^{N'} \langle \tau \rangle^{-1})$ defined for $\tau \notin I_\eps$ such that 
$$
(\P_h - \tau)^N \mathsf{S}_{N'}  = \Pi_\eps^\pm +\mathsf{R}_{N'} \Pi_\eps^\pm  , \quad \text{ for all } \tau \notin I_\eps.
$$
\end{proposition}
Note that~\eqref{param-alg-N-bis-1}-\eqref{param-alg-N-bis-2} are then satisfied if we take $N=N'$ in this proposition and 
\begin{align}
\label{e:def-Q-b}
Q_h = \frac{b}{2} \begin{pmatrix}
1& 1 \\
1 & 1
\end{pmatrix} ,
\end{align}
and further require that $b \in C^\infty(M)$ so that $Q_h\in \Psi^0(M;\C^{2\times2})$.
As a preparation, we first prove the following lemma.
\begin{lemma}
\label{l:solving}
Assume that $Q_h \in \Psi^{0}( M; \C^{2\times2})$, $P_h$ is defined by~\eqref{e:def-Ph-encore}, and $\P_h = P_h + ihQ_h$.
For any $N\in \N$, $k \in \R$ and $A \in \Psi^0(\langle \tau \rangle^k)$, there exist $S \in \Psi^{-\infty}(\langle \tau \rangle^{k-N})$ and $R \in \Psi^{-\infty}(h \langle \tau \rangle^{k-1})+ \Psi^{-\infty}(h^\infty\langle \tau \rangle^k)$ defined for $\tau \notin I_\eps$ such that 
$$
 (\P_h-\tau)^N S = A \Pi_\eps^\pm + R \Pi_\eps^\pm , \quad \text{ for all } \tau \notin I_\eps^\pm.
$$
If moreover $A=1$ (hence $k=0$), then one can choose $S$ so that $R \in \Psi^{-\infty}(h \langle \tau \rangle^{-1})$.
\end{lemma}

\bnp
We first define $\tilde{\Pi}_\eps^\pm = \tilde{\chi}_\eps^\pm (P_h)$ with $\tilde{\chi}_\eps^\pm$ having the same properties as $\chi_\eps^\pm$ and equal to one in a neighborhood of $\supp\chi_\eps^\pm$. Then, for $\tau \notin I_\eps^\pm$, we may define the operator $(P_h-\tau)^{N}\tilde{\Pi}_\eps^\pm$ via functional calculus (see the discussion preceding~\eqref{e:estim-elliptic}). It satisfies (see again the discussion preceding~\eqref{e:estim-elliptic})
\begin{equation*}
\nor{(P_h-\tau)^{-N} \tilde{\Pi}_\eps^\pm y }{L^2(\M;\C^2)} \leq \left(\frac{4}{\alpha \eps_0 \langle\tau\rangle}\right)^{N}  \nor{\tilde{\Pi}_\eps^\pm y }{L^2(\M;\C^2)}, \quad \tau \in  \R\setminus I_\eps^\pm.
\end{equation*}
According to Theorem~\ref{t:thm-Kuster} (and using the remark in Definition~\ref{d:class-tau} that the calculus in $\Psi^{m}(\langle \tau \rangle^k, T^*M; \C^{2\times2})$ reduces to the calculus in $\Psi^{m}(M; \C^{2\times2})$  with symbols depending on $\tau$ as a parameter, with all derivatives uniformly bounded in $\tau \in \R\setminus I_\eps^\pm$), we have $(P_h-\tau)^{-N} \tilde{\Pi}_\eps^\pm \in \Psi^{-\infty}(\langle \tau \rangle^{-N})$ for $\tau \notin I_\eps^\pm$ with (according to~\eqref{e:def-Ph-encore}-\eqref{e:defPipm}) principal symbol
$$
 \begin{pmatrix}
(|\xi|_x-\tau)^{-N} \tilde{\chi}_\eps^\pm(|\xi|_x)& 0\\
0 & (- |\xi|_x-\tau)^{-N} \tilde{\chi}_\eps^\pm(|\xi|_x)
\end{pmatrix}
.$$
We then set $S := (P_h-\tau)^{-N}\tilde{\Pi}_\eps^\pm A \Pi_\eps^\pm$. We notice that $\P_h-\tau = P_h-\tau + i h Q_h$ where $(\P_h-\tau)\tilde{\Pi}_\eps^\pm\in \Psi^{-\infty}(\langle \tau \rangle)$ and $Q_h\in \Psi^0(M; \C^{2\times2})$, and hence, according to the product rule of pseudodifferential calculus, 
$$
(\P_h-\tau)^N \tilde{\Pi}_\eps^\pm= (P_h-\tau + i h Q_h)^N\tilde{\Pi}_\eps^\pm = (P_h-\tau)^N\tilde{\Pi}_\eps^\pm + R_N \tilde{\Pi}_\eps^\pm, \quad R_N\tilde{\Pi}_\eps^\pm \in \Psi^{-\infty}( h\langle \tau \rangle^{N-1}) .
$$
As a consequence, we obtain 
$$
(\P_h-\tau)^N S=  \left((P_h-\tau)^N + R_N \right)\tilde{\Pi}_\eps^\pm(P_h-\tau)^{-N}A \Pi_\eps^\pm = \tilde{\Pi}_\eps^\pm A \Pi_\eps^\pm + R_N \tilde{\Pi}_\eps^\pm(P_h-\tau)^{-N}A \Pi_\eps^\pm .
$$
Since $\tilde{\chi}= 1$ on a neighborhood of $\supp(\chi)$, we have from functional calculus and the product rule 
\begin{align*}
 \tilde{\Pi}_\eps^\pm A \Pi_\eps^\pm = A \Pi_\eps^\pm + R' \Pi_\eps^\pm, \quad R' \in \Psi^{-\infty}(h^\infty \langle \tau \rangle^k) , \\
R'' :=  R_N \tilde{\Pi}_\eps^\pm(P_h-\tau)^{-N}A  \in  \Psi^{-\infty}( h\langle \tau \rangle^{N-1}\langle \tau \rangle^{-N}\langle \tau \rangle^{k}).
\end{align*}
This concludes the proof of the first statement with $R=R'+R''$.  To prove the second statement, we only need to remark that if $A=1$, then $R' =0$ and $R=R''\in  \Psi^{-\infty}( h \langle \tau \rangle^{-1})$.
\enp

We next prove Proposition~\ref{p:parametrix}. This is a classical parametrix construction, see \eg~\cite[Proposition~E.32]{DZ:book}.
 \bnp[Proof of Proposition~\ref{p:parametrix}]
 The proof proceeds by induction on $N'$. The first step follows from the last statement in Lemma~\ref{l:solving}, namely, there are $S_1,R_1$ such that
 \begin{align}
\label{e:split-param-bis}
 (\P_h - \tau)^N S_1 = \Pi_\eps^\pm - R_1\Pi_\eps^\pm , \quad \text{ with } \quad 
  S_1   \in \Psi^{-\infty}(\langle \tau \rangle^{-N}) , \quad R_1   \in \Psi^{-\infty}(h \langle \tau \rangle^{-1}).
\end{align}
Next, we solve the errors away by induction  with Lemma~\ref{l:solving}. Given $j \in \N^*$ and $R_{j}  \in \Psi^{-\infty}(h^j \langle \tau \rangle^{-1})$, Lemma~\ref{l:solving} gives the existence of $S_{j+1}\in \Psi^{-\infty}(h^j \langle \tau \rangle^{-N-1})$ and $R_{j+1} \in \Psi^{-\infty}(h^{j+1} \langle \tau \rangle^{-1})$ such that  
 \begin{align*}
(\P_h - \tau)^N S_{j+1} = R_j \Pi_\eps - R_{j+1}\Pi_\eps .
\end{align*}
Setting $\mathsf{S}_{N'} = \sum_{j=1}^{N'} S_j \in \Psi^{-\infty}(\langle \tau \rangle^{-N})$, we have obtained
\begin{align*}
(\P_h - \tau)^N \mathsf{S}_{N'} & = (\P_h - \tau)^N S_1 + \sum_{j=1}^{N'-1} (\P_h - \tau)^N S_{j+1} = \Pi_\eps^\pm - R_1\Pi_\eps^\pm + \sum_{j=1}^{N'-1} R_j \Pi_\eps - R_{j+1}\Pi_\eps \\
& = \Pi_\eps^\pm - R_{N'} \Pi_\eps^\pm ,
\end{align*}
where $R_{N'} \in \Psi^{-\infty}(h^{N'} \langle \tau \rangle^{-1})$, which concludes the proof of the proposition.
 \enp

 \subsection{End of the proof of Theorem~\ref{t:semiclassic-intro}}
 \label{s:proof-thm-smooth-h}
As a direct corollary of Theorem~\ref{t:resolvent-implies-long-time-smooth} and Proposition~\ref{p:parametrix}, we deduce a proof of Theorem~\ref{t:semiclassic-intro}.
Note that, for a general $b\in L^\infty$, we could also reformulate (without using Proposition~\ref{p:parametrix}) Theorem~\ref{t:resolvent-implies-long-time} (instead of Theorem~\ref{t:resolvent-implies-long-time-smooth}) in terms of the semigroup $(e^{t\A_m})_{t\in \R}$, that is to say for solutions to the damped Klein-Gordon/wave equations. We do not formulate the associated result for the sake of concision.

\bnp[Proof of Theorem~\ref{t:semiclassic-intro}]
We apply Theorem~\ref{t:resolvent-implies-long-time-smooth} to $\H = L^2(M;\C^2)$ and $\P_h= h \P_m$, that is, $P_h$ given by \eqref{e:def-Ph-encore}, $Q_h$ by~\eqref{e:def-Q-b}, $I_\eps^\pm$ by~\eqref{e:defIpm}, and $\Pi_\eps^\pm$ by~\eqref{e:defPipm}.
From Item~\ref{i:resolvantes} in Corollary~\ref{corollary-Pm-Am-reso}, we have 
\begin{align}
\label{e:totototo}
\nor{(z - \A_m)^{-1}}{\L\big(H^1_m \times L^2\big)} =  \nor{(z  - i\P_m)^{-1}}{\L\big(L^2(M;\C^2)\big)}= \nor{(z  - i\P_h/h)^{-1}}{\L\big(L^2(M;\C^2)\big)} =h \nor{(zh  - i\P_h)^{-1}}{\L\big(L^2(M;\C^2)\big)} . 
\end{align}
Notation~\eqref{e:asspt-G} thus implies (taking $z=i\frac{\tau}{h}$ in the above identity)  
\begin{align}
\label{e:totototo-2}
h \nor{(\tau  - \P_h)^{-1}}{\L\big(L^2(M;\C^2)\big)}  =\nor{(i\tau/h - \A_m)^{-1}}{\L\big(H^1_m \times L^2\big)}  \leq G(h), \quad \text{ for } \tau \in I_\eps^\pm , h \in(0,h_0) ,
\end{align}
which is Assumption~\eqref{e:hyp-resP} of Theorem~\ref{t:resolvent-implies-long-time-smooth} (on $I_\eps^\pm$).
From Item~\ref{i:Energ-am}  in Corollary~\ref{corollary-Pm-Am-reso} (taken at $t=0$) we have
$ \nor{V}{H^1_m\times L^2} =  \nor{\Sigma L_m V}{L^2(M;\C^2)}$ for all $V \in H^1_m\times L^2$. Moreover, according to~\eqref{e:defPipm}, we have $\Pi_\eps^\pm = \chi_\eps^\pm(h\Lambda_m) I_2$ and hence $\Pi_\eps^\pm \Sigma L_m U_0 =\Sigma L_m \Pi_\eps^\pm U_0$ (recall the definition of $L_m$ and $\Sigma$ in~\eqref{e:def-Lm} and~\eqref{e:def-sigma}).
As a consequence, given $U_0 \in H^1_m\times L^2$, with $(u,\d_t u)(t)$ the solution  to~\eqref{eq: stabilization} with data $\Pi_\eps^\pm U_0$ and $V = \Pi_\eps^\pm\Sigma L_m U_0$, we have using again Item~\ref{i:Energ-am} in Corollary~\ref{corollary-Pm-Am-reso}, 
\begin{align*}
\nor{e^{\frac{it}{h}\P_h} V}{\H} &= \nor{e^{it\P_m}\Sigma L_m \Pi_\eps^\pm U_0}{L^2(M;\C^2)} =  \nor{e^{t\A_m} \Pi_\eps^\pm U_0}{H^1_m\times L^2} = \sqrt{2E_m(u(t))},\\
\nor{V }{\H} &= \nor{\Sigma L_m \Pi_\eps^\pm U_0}{L^2(M;\C^2)} = \nor{\Pi_\eps^\pm U_0}{H^1_m\times L^2}=\sqrt{2E_m(u(0))}. 
\end{align*}
The parametrix assumption~\eqref{param-alg-N-bis-1}--\eqref{param-alg-N-bis-2} in this setting is a consequence Proposition~\ref{p:parametrix}.
Theorem~\ref{t:resolvent-implies-long-time-smooth} applies and its conclusion thus rewrites, with $(u,\d_t u)(t)$ the solution to~\eqref{eq: stabilization} with data $\Pi_\eps^\pm U_0 = \big(\chi_\eps^\pm(h\Lambda_m)u_0,\chi_\eps^\pm(h\Lambda_m)u_1\big)$,
\begin{align*}
\frac{1}{T_h}\int_\R \left|\psi\left( \frac{t}{T_h}\right) \right|^2 E_m(u(t)) dt
 \leq  \left(  \frac{G(h)^2}{T_h^2} \nor{\psi'}{L^2(\R)}^2 
+ \mathsf{C}_{N,\psi}  h^{2N-1} T_h \right) E_m(u(0)) , \\
E_m(u(T_h))
  \leq  \left( \frac{G(h)^2}{T_h^2} \frac{\nor{\psi'}{L^2(\R)}^2}{\nor{\psi}{L^2(0,1)}^2}    + \frac{\mathsf{C}_{N,\psi} }{\nor{\psi}{L^2(0,1)}^2}    h^{2N-1} T_h \right) E_m(u(0)) , \\
   E_m(u(T_h))
  \leq  \left( (2+\delta) \frac{G(h)^2}{T_h^2}     +C_{N,\delta}   h^{2N-1} T_h \right) E_m(u(0)) .
\end{align*}
This concludes the proof of of Theorem~\ref{t:semiclassic-intro}.
\enp
We now state and prove an analogue of Theorem~\ref{t:semiclassic-intro} for the wave equation, \ie~\eqref{e:DWE} or equivalently~\eqref{eq: stabilization} with $m=0$. 
\begin{theorem}
\label{t:semiclassic-++}
Assume that $b \in C^\infty(M;\R^+)$, let $\eps \in (0,1)$, and defined $I_\eps$ and $\chi_\eps$ as in Theorem~\ref{t:semiclassic-intro}.
Then, for any $N',N \in \N$, any $\psi \in H^1_{\comp}(\R)$ with $\supp \psi \subset \R^+$, and any $\delta \in (0,1)$, there is $C>0$ such that 
for all $(u_0 , u_1) \in \chi_\eps(h \Lambda_0)(H^1\times L^2)$ and $u$ associated solution to~\eqref{eq: stabilization}, 
for all $h \in (0,1)$ and all $T_h \leq h^{-N}$
\begin{align}
\label{e:toto-++}
\frac{1}{T_h}\int_\R \left|\psi\left( \frac{t}{T_h}\right) \right|^2 E(u(t)) dt
 \leq  \left(  \frac{G(h)^2}{T_h^2} \nor{\psi'}{L^2(\R)}^2 
+ C h^{N'}\right) E(u(0)) , \\
    E(u(T_h)) \leq  \left( (2+\delta) \frac{G(h)^2}{T_h^2}     +C h^{N'} \right) E(u(0)) ,
\end{align}
where 
\begin{align}
\label{e:asspt-G++}
G(h) = \sup_{\tau \in \R , |\tau| \in I_\eps} \nor{(i\tau/h - \A_+)^{-1}}{\L\big(H^1_+ \times L^2\big)}  .
\end{align}
\end{theorem}

\bnp[Proof of Theorem~\ref{t:semiclassic-++}]
We set $\H=\L^2_+$, $\P_h=h\P_+$, that is, $P_h$ is still given by \eqref{e:def-Ph-encore}  (but acting $\H^1_+\to \L^2_+$), $Q_h$ by~\eqref{e:def-Q-b}, $I_\eps^\pm$ by~\eqref{e:defIpm}, and $\Pi_\eps^\pm$ by~\eqref{e:defPipm}, all taken with $m=0$.
Notation~\eqref{e:asspt-G++} together with Item~\ref{i:egal-resol+} of Corollary~\ref{c:prop-op-+} (and the fact that $\L^2_+$ is normed by $\|\cdot\|_{L^2(M;\C^2)}$) yield
$$
h \nor{(\tau-\P_h)^{-1}}{\L\big(L^2(M;\C^2)\big)}  =\nor{(i\tau/h - \A_+)^{-1}}{\L\big(H^1_+ \times L^2\big)}  \leq G(h), \quad \text{ for } \tau \in I_\eps^\pm , h \in(0,h_0) ,
$$
which is Assumption~\eqref{e:hyp-resP} of Theorem~\ref{t:resolvent-implies-long-time-smooth} (on $I_\eps^\pm$).
From Item~\ref{i:++enrgy}  in Corollary~\ref{c:prop-op-+} (taken at $t=0$) we have
$ \nor{V}{H^1_+\times L^2} =  \nor{\Sigma L_+ V}{L^2(M;\C^2)}$ for all $V \in H^1_+\times L^2$. Moreover, according to~\eqref{e:defPipm}, we have $\Pi_\eps^\pm = \chi_\eps^\pm(h\Lambda_0) I_2$ and hence $\Pi_\eps^\pm \Sigma L_+ U_0 =\Sigma L_+ \Pi_\eps^\pm U_0$ (recall the definition of $L_+$ and $\Sigma$ in~\eqref{e:def-L+} and~\eqref{e:def-sigma}).
As a consequence, given any $U_0 \in H^1\times L^2$, with $(u,\d_t u)(t)$ the solution  to~\eqref{eq: stabilization} with data $\Pi_\eps^\pm U_0 \in H^1_+\times L^2$ and $V = \Pi_\eps^\pm\Sigma L_m U_0$, we have using again Item~\ref{i:++enrgy}  in Corollary~\ref{c:prop-op-+}, 
\begin{align*}
\nor{e^{\frac{it}{h}\P_h} V}{\H} &= \nor{e^{it\P_+}\Sigma L_+ \Pi_\eps^\pm U_0}{L^2(M;\C^2)} =  \nor{e^{t\A_+} \Pi_\eps^\pm U_0}{H^1_+\times L^2} = \sqrt{2E(u(t))},\\
\nor{V }{\H} &= \nor{\Sigma L_+ \Pi_\eps^\pm U_0}{L^2(M;\C^2)} = \nor{\Pi_\eps^\pm U_0}{H^1_+\times L^2}=\sqrt{2E(u(0))}. 
\end{align*}
The parametrix assumption~\eqref{param-alg-N-bis-1}--\eqref{param-alg-N-bis-2} in this setting is a consequence of Proposition~\ref{p:parametrix}.
Theorem~\ref{t:resolvent-implies-long-time-smooth} applies and its conclusion thus rewrites, with $(u,\d_t u)(t)$ the solution to~\eqref{eq: stabilization} with data $\Pi_\eps^\pm U_0 = \big(\chi_\eps^\pm(h\Lambda_0)u_0,\chi_\eps^\pm(h\Lambda_0)u_1\big)$
\begin{align*}
\frac{1}{T_h}\int_\R \left|\psi\left( \frac{t}{T_h}\right) \right|^2 E(u(t)) dt
 \leq  \left(  \frac{G(h)^2}{T_h^2} \nor{\psi'}{L^2(\R)}^2 
+ \mathsf{C}_{N,\psi}  h^{2N-1} T_h \right) E(u(0)) , \\
E(u(T_h))
  \leq  \left( \frac{G(h)^2}{T_h^2} \frac{\nor{\psi'}{L^2(\R)}^2}{\nor{\psi}{L^2(0,1)}^2}    + \frac{\mathsf{C}_{N,\psi} }{\nor{\psi}{L^2(0,1)}^2}    h^{2N-1} T_h \right) E(u(0)) , \\
   E(u(T_h))
  \leq  \left( (2+\delta) \frac{G(h)^2}{T_h^2}     +C_{N,\delta}   h^{2N-1} T_h \right) E(u(0)) .
\end{align*}
This concludes the proof of of Theorem~\ref{t:semiclassic-++}.
\enp

\section{Egorov theorems for non-selfadjoint operators in Ehrenfest time}
\label{s:egorov}
In this section, we prove the Egorov property needed to conclude the proof of Theorem~\ref{t:theorem-holder-0}.
We first collect classical/symbolic estimates linked to the expansion of the homogeneous geodesic flow $\varphi^t$ in Section~\ref{sub:classical-estimates}, and then state and prove in Section~\ref{sub:Egorov} the Egorov property needed in the main part of the proof.

\subsection{Classical estimates}
\label{sub:classical-estimates}
Define $\varphi^t$ the {\em homogeneous} geodesic flow on $T^*M \setminus 0$, that is, the Hamiltonian flow of the symbol $(x,\xi) \mapsto |\xi|_x = \sqrt{g_x^*(\xi,\xi)}$ (where $g^*$ is the metric on $T^*M$). Writing $M_\mu(x,\xi) = (x,\mu \xi)$ for the multiplication in the fibers, we have $M_\mu \circ \varphi^t = \varphi^t \circ M_\mu$ for all $\mu>0, (x,\xi) \in T^*M \setminus 0$. That is to say, $\varphi^t$ is $0$-homogeneous (whence its denomination). Remark that if $g^t$ denotes the usual geodesic flow on $T^*M$, that is, the Hamiltonian flow of the symbol $(x,\xi) \mapsto\frac12|\xi|_x^2 = \frac12 g_x^*(\xi,\xi)$ (see \eg~\cite[Appendix~B1]{DLRL:13} for properties of these flows), then we have $g^t(x,\xi) = \varphi^{t|\xi|_x}(x,\xi)$ for all $(x,\xi) \in T^*M \setminus 0$. In particular, these two flows coincide on the energy layer $\{|\xi|_x=1\}$.
We also denote by $\Upsilon_{\max}$ the maximal expansion rate of the flow, namely,
\begin{equation}
\label{e:def-lambdamax}
 \Upsilon_{\max} = \limsup_{|t|\to + \infty} \frac{1}{|t|} \log \sup_{(x, \xi) \in T^*M, |\xi|_x = 1} \|d \varphi^t(x,\xi)\| ,
 \end{equation}
where $ \|d \varphi^t(x,\xi)\| $ is the operator norm of $d \varphi^t(x,\xi) : T_{(x,\xi)} T^*M \to T_{\varphi^t(x,\xi)} T^*M $ with respect to any smooth norm on the fibers of $T(T^*M)$.
Note that $\Upsilon_{\max} \in [0,+ \infty)$.
We also use energy layers of the form 
\begin{equation}
\label{e:def-Keps}
K_\Gamma = \{(x,\xi) \in T^*M , \quad \Gamma \leq |\xi|_{x}\leq \Gamma^{-1}\} , \quad \text{ for } \Gamma \in(0,1) ,
\end{equation}
that is to say, avoiding the zero section and infinity in the fibers.

The following lemma is a reformulation in the present context of~\cite[Lemma~C.1]{DG:14} (see also~\cite[Equation~(5.6)]{AnNon}). 
\begin{lemma}
\label{l:DG}
For any $\Gamma>0$ and $\Upsilon_1>  \Upsilon_{\max}$, the following statement holds. For all $k\in \N$, there is $C_{k,\Upsilon_1}>0$ such that for all $a \in C^\infty(T^*M)$ with $\supp(a) \subset \Int(K_\Gamma)$, and all $t\in \R$, we have 
 \bnan
 \label{e:estimate-flow}
 \nor{a\circ \varphi^t}{C^k(K_\Gamma)} \leq C_{k,\Upsilon_1} e^{k\Upsilon_1 |t|} \nor{a}{C^k(K_\Gamma)}  .
 \enan
\end{lemma}
We recall the definition of $S^{m}_\rho(T^*M)$ in Appendix~\ref{app:symbol-classes} and deduce from Lemma~\ref{l:DG} the following proposition. 
\begin{proposition}
\label{p:symb-log-times}
Take $\mu, \nu , \rho \geq 0$, $a \in S^{-\infty}_\rho(T^*M)$ such that $\supp(a) \subset \Int(K_\Gamma)$,  $q \in S^{0}_\nu(T^*M)$ such that $\Re(q)\geq0$ on $T^*M$, and set
\bna
a_t(x,\xi)=  a\circ \varphi^{t_3}(x,\xi) e^{-\int_{t_1}^{t_2} q\circ \varphi^s(x,\xi) ds} , \quad t=(t_1,t_2,t_3) .
\ena
Then, for $\mu \geq 0$, we have $$
\Big\{a_t ,  |t_3| \leq \mu \log(h^{-1}) ,- \mu \log(h^{-1}) \leq t_1\leq t_2 \leq \mu \log(h^{-1}) \Big\} \subset S^{-\infty}_{\mu \Upsilon_1 + \max(\rho, \nu)} (T^*M) ,
$$
and this set is a bounded subset of $S^{-\infty}_{\mu \Upsilon_1 + \max(\rho, \nu)} (T^*M)$.
\end{proposition}
This means that $a_t$ satisfies the symbolic estimates of $S^{-\infty}_{\mu \Upsilon_1 + \max(\rho, \nu)} (T^*M)$, uniformly for $(t_1,t_2,t_3)$ in the above range, and in particular all bounded by $\mu \log(h^{-1})$.
Of course, in view of pseudodifferential calculus (see Appendix~\ref{s:calcul-pseudo}), this is only useful/used for those $(\mu ,\rho, \nu)$ such that $\mu \Upsilon_1 + \max(\rho, \nu) <\frac12$.

We need the following lemma in the proof of Proposition~\ref{p:symb-log-times}.
\begin{lemma}
\label{l:deriv-expo}
Let $\mu,\nu \geq0$, and assume $q \in S^0_\nu(T^*M)$. Then, for all $\alpha \in \N^{2n}$, there is a function $r_{t_1,t_2,\alpha} \in C^\infty(T^*M)$ such that
\begin{align}
\label{e:def-ralpha}
\d^\alpha \left(  e^{-\int_{t_1}^{t_2} q\circ \varphi^s ds} \right) = r_{t_1,t_2,\alpha}  \ e^{-\int_{t_1}^{t_2} q\circ \varphi^s ds} , \end{align}
and which satisfies
\begin{align}
\label{e:estim-ralpha}
\nonumber  &\text{for all }\chi \in C^\infty_c(\Int(K_\Gamma)),  \text{ and all }k \in \N ,  \text{ there is }  C_{\alpha,\chi,k}>0  \text{ such that } \\
& \nor{ \chi r_{t_1,t_2,\alpha}}{C^k(K_\Gamma)} \leq C_{\alpha,\chi,k} h^{-(|\alpha| + k)(\mu \Upsilon_1 + \nu)} ,
 \quad  \quad \text{for all }-\mu \log(h^{-1})  \leq t_1 \leq t_2 \leq \mu \log(h^{-1}) . \end{align}
\end{lemma}
\bnp[Proof of Lemma~\ref{l:deriv-expo}]
We prove the result in local charts by induction on $|\alpha|$. Note first that for $|\alpha|=0$, \ie $\alpha = (0, \cdots, 0)$, we have $r_{t_1,t_2,\alpha} =1$ which satisfies~\eqref{e:estim-ralpha}.
Assume now~\eqref{e:def-ralpha}-\eqref{e:estim-ralpha} for $|\alpha|=m$. Then, taking $\alpha ,\beta \in \N^n$ multiindices with $|\alpha|=m$ and $|\beta| = 1$ (that is, $\d^\beta$ is a first order derivative),  using the Leibniz rule and the induction assumption, we have
\begin{align*}
\d^{\alpha+\beta} \left(  e^{-\int_{t_1}^{t_2} q\circ \varphi^s ds} \right) = \d^\beta \left(  r_{t_1,t_2,\alpha}  e^{-\int_{t_1}^{t_2} q\circ \varphi^s ds} \right) 
=  \left( \d^\beta r_{t_1,t_2,\alpha}  - r_{t_1,t_2,\alpha} \int_{t_1}^{t_2}\d^\beta \big( q\circ \varphi^s \big) ds  \right) e^{-\int_{t_1}^{t_2} q\circ \varphi^s ds} ,
\end{align*}
that is to say $r_{t_1,t_2,\alpha+\beta}=\d^\beta  r_{t_1,t_2,\alpha}  - r_{t_1,t_2,\alpha} \int_{t_1}^{t_2}\d^\beta \big( q\circ \varphi^s \big) ds$. For $k \in \N$, we now estimate (recall that $|\beta|=1$)
\begin{align}
\label{e:interm-Ck-estim}
 \nor{\chi r_{t_1,t_2,\alpha+\beta}}{C^k(K_\Gamma)} & \leq \nor{\chi \d^\beta  r_{t_1,t_2,\alpha} }{C^k(K_\Gamma)} + \nor{\chi r_{t_1,t_2,\alpha} \int_{t_1}^{t_2}\d^\beta \big( q\circ \varphi^s \big) ds }{C^k(K_\Gamma)} \nonumber \\
&  \leq \nor{\tilde{\chi} r_{t_1,t_2,\alpha} }{C^{k+1}(K_\Gamma)} + \int_{t_1}^{t_2} \nor{\chi r_{t_1,t_2,\alpha} \d^\beta \big( q\circ \varphi^s \big) }{C^k(K_\Gamma)}ds ,
 \end{align}
 where we choose (for later use) $\tilde{\chi}$ such that
  \begin{align}
 \label{e:def-tilde-chi}
 \nonumber
&  \tilde{\chi}(x, \xi) = \eta (|\xi|_x)\text{ with }\eta \in C^\infty_c(\R), \supp(\eta) \subset (\Gamma,\Gamma^{-1}) \text{ (and hence $\tilde{\chi} \in C^\infty_c(\Int(K_\Gamma))$)}\\
 & \text{and } \tilde{\chi} = 1 \text{ in a \nhd of }\supp(\chi) .
 \end{align}
The first term in~\eqref{e:interm-Ck-estim} is directly estimated by the induction assumption~\eqref{e:estim-ralpha} as 
$$\nor{\tilde{\chi}r_{t_1,t_2,\alpha} }{C^{k+1}(K_\Gamma)}  \leq  C h^{-(|\alpha| + k+ 1)(\mu \Upsilon_1 + \nu)} =  C  h^{-(m+1 + k)(\mu \Upsilon_1 + \nu)}.$$
Concerning the second term in~\eqref{e:interm-Ck-estim}, we use again the Leibniz rule to obtain
\begin{align*}
\nor{\chi r_{t_1,t_2,\alpha} \d^\beta \big( q\circ \varphi^s \big) }{C^k(K_\Gamma)} & \leq C \sum_{\ell = 0}^k \nor{\tilde{\chi}r_{t_1,t_2,\alpha} }{C^{k-\ell}(K_\Gamma)}  \nor{ \chi \d^\beta \big( q\circ \varphi^s \big) }{C^\ell(K_\Gamma)} \\
&  \leq C \sum_{\ell = 0}^k \nor{\tilde{\chi} r_{t_1,t_2,\alpha} }{C^{k-\ell}(K_\Gamma)}  \nor{\tilde{\chi} q\circ \varphi^s  }{C^{\ell+1}(K_\Gamma)}  
 \end{align*}
 The induction assumption~\eqref{e:estim-ralpha} implies 
 $$\nor{\tilde{\chi}r_{t_1,t_2,\alpha} }{C^{k-\ell}(K_\Gamma)}  \leq C h^{-(|\alpha| + k-\ell )(\mu \Upsilon_1 + \nu)}.
 $$ Concerning the other term, we recall the choice of $\tilde{\chi}$ in \eqref{e:def-tilde-chi} which implies that $\tilde{\chi} q\circ \varphi^s = (\tilde{\chi}q)\circ \varphi^s$. Hence, this second term is estimated with~\eqref{e:estimate-flow} as (recall $t_2\geq t_1$ and $s \in [t_1,t_2] \subset [ - \mu \log (h^{-1}) , \mu \log (h^{-1})]$)
 \begin{align*}
 \nor{\tilde{\chi} q\circ \varphi^s}{C^{\ell+1}(K_\Gamma)} \leq C e^{(\ell+1)\Upsilon_1 |s|} \nor{q}{C^{\ell+1}(K_\Gamma)}  \leq   C e^{(\ell+1)\Upsilon_1 |s|} h^{-\nu (\ell+1)},
\end{align*}
where, in the last inequality we used $q \in S^0_\nu(T^*M)$. Combining the last four estimates in~\eqref{e:interm-Ck-estim} now implies
\begin{align*}
 \nor{\chi r_{t_1,t_2,\alpha+\beta}}{C^k(K_\Gamma)}
   & \leq C  h^{-(m+1 + k)(\mu \Upsilon_1 + \nu)} + C \sum_{\ell = 0}^k  h^{-(m+ k-\ell )(\mu \Upsilon_1 + \nu)}\int_{t_1}^{t_2} e^{(\ell+1)\Upsilon_1 |s|} h^{-\nu (\ell+1)} ds \\
   & \leq C  h^{-(m+1 + k)(\mu \Upsilon_1 + \nu)} + C \sum_{\ell = 0}^k  h^{-(m + k-\ell )(\mu \Upsilon_1 + \nu)} e^{(\ell+1)\Upsilon_1 \max(|t_1|,|t_2|)} h^{-\nu (\ell+1)}  .
    \end{align*}
    Using that $|t_1|, |t_2| \leq \mu \log(h^{-1})$, we obtain
    \begin{align*}
 \nor{\chi r_{t_1,t_2,\alpha+\beta}}{C^k(K_\Gamma)}
   & \leq C  h^{-(m+1 + k)(\mu \Upsilon_1 + \nu)} + C \sum_{\ell = 0}^k  h^{-(m+ k-\ell )(\mu \Upsilon_1 + \nu)} h^{-(\ell+1)\mu \Upsilon_1} h^{-\nu (\ell+1)} \\
   &\leq C  h^{-(m+1 + k)(\mu \Upsilon_1 + \nu)} ,
    \end{align*}
    which is the sought estimate~\eqref{e:estim-ralpha} for $r_{t_1,t_2,\alpha+\beta}$, with $|\alpha + \beta| = m+1$. This concludes the proof of the lemma.
\enp
We can now return to the proof of Proposition~\ref{p:symb-log-times}.
\bnp[Proof of Proposition~\ref{p:symb-log-times}]
In local charts, we write the Leibniz formula, with $\d = \d_{(x,\xi)}$ and $\beta \in \N^{2n}$: 
\begin{align}
\label{e:estim-dbetaa}
|\d^\beta a_t | \leq C \sum_{\gamma + \alpha = \beta } \left| \d^\gamma \big(  a\circ \varphi^{t_3} \big) \right| \left| \d^\alpha \big(   e^{-\int_{t_1}^{t_2} q\circ \varphi^s ds}\big) \right|
=  C \sum_{\gamma + \alpha = \beta } \left| \d^\gamma \big(  a\circ \varphi^{t_3} \big) \right|  \left| r_{t_1,t_2,\alpha}  e^{-\int_{t_1}^{t_2} q\circ \varphi^s ds} \right| ,
\end{align}
where we used~\eqref{e:def-ralpha} for introducing $r_{t_1,t_2,\alpha} $. Using that $\Re(q)\geq 0$ and $t_2 \geq t_1$, we have $\left| e^{-\int_{t_1}^{t_2} q\circ \varphi^s ds} \right| \leq 1$. 
Now, remark that all terms in the sum are compactly supported in $\Int(K_\Gamma)$, since this is the case for $a\circ \varphi^{t_3}$ ($K_\Gamma$ is invariant by $\varphi^s$) and all its derivatives. We introduce an additional energy cutoff $\chi$ such that
  \begin{align}
 \label{e:def-chi-chi}
 \nonumber
& \chi(x, \xi) = \eta (|\xi|_x)\text{ with }\eta \in C^\infty_c(\R), \supp(\eta) \subset (\Gamma, \Gamma^{-1}) \text{ (and hence $\chi \in C^\infty_c(\Int(K_\Gamma))$)}\\
 & \text{and } \chi = 1 \text{ in a \nhd of }\supp(a) .
 \end{align}
In particular, $\chi \circ \varphi^s = \chi$ for all $s\in \R$, and thus $\chi a\circ \varphi^{t_3} = a\circ \varphi^{t_3}$ for all $t_3\in \R$. 
Therefore, we may introduce $\chi$ in~\eqref{e:estim-dbetaa} and then apply~\eqref{e:estim-ralpha} (with $k=0$) to obtain 
\begin{align*}
|\d^\beta a_t |
 \leq C \sum_{\gamma + \alpha = \beta } \left| \d^\gamma \big(  a\circ \varphi^{t_3} \big) \right|  \left|\chi  r_{t_1,t_2,\alpha}  e^{-\int_{t_1}^{t_2} q\circ \varphi^s ds} \right| 
 \leq C \sum_{\gamma + \alpha = \beta } \left| \d^\gamma \big(  a\circ \varphi^{t_3} \big) \right| h^{-|\alpha|(\mu \Upsilon_1 + \nu)} \quad \text{on } T^*M  ,
\end{align*}
uniformly for $-\mu \log(h^{-1})\leq t_1 \leq t_2 \leq \mu \log(h^{-1})$.
Using~\eqref{e:estimate-flow} in this estimate together with the fact that $|t_3| \leq \mu \log(h^{-1})$, we now obtain
\begin{align*}
\langle \xi \rangle^N |\d^\beta a_t | 
\leq  C_{N} \sum_{\gamma + \alpha = \beta } C e^{|\gamma| \Upsilon_1 |t_3|} \nor{a}{C^{|\gamma|}(K_\Gamma)}h^{-|\alpha|(\mu \Upsilon_1 + \nu)} 
\leq C \sum_{\gamma + \alpha = \beta } h^{-\mu \Upsilon_1|\gamma|}h^{-\rho |\gamma|}  h^{-|\alpha|(\mu \Upsilon_1 + \nu)} , 
\end{align*}
where, in the last inequality we also used $a \in S^0_\rho(T^*M)$. This implies in particular that 
\begin{align*}
\langle \xi \rangle^N |\d^\beta a_t | \leq C_{N,\beta} h^{-|\beta| \big( \mu \Upsilon_1+ \max(\rho, \nu) \big)}  \quad \text{on } T^*M  ,
\end{align*}
and hence that  $a_t \in S^{-\infty}_{\mu \Upsilon_1 + \max(\rho, \nu)} (T^*M)$ uniformly for $|t_3| \leq \mu \log(h^{-1})$ and $-\mu \log(h^{-1})\leq t_1\leq t_2\leq \mu \log(h^{-1})$. This concludes the proof of the proposition.
\enp

\subsection{Egorov theorem in Ehrenfest time}
\label{sub:Egorov}
We are now ready to prove an Egorov theorem in Ehrenfest time for the group $\left(e^{it\tilde{P}_Q}\right)_{t\in \R}$ generated by ($i$ times) the operator 
$$
\tilde{P}_Q = \Lambda \tilde{\chi}_0(h\Lambda) + i Q .
$$
where $Q \in \Psi^0_\nu(M)$ and $\tilde{\chi}_0 \in C^\infty(\R;[0,1])$ is such that $\supp( \tilde{\chi}_0) \subset (0, \infty)$ and $ \tilde{\chi}_0=1$ on a neighborhood of $(\Gamma,\Gamma^{-1})$. We rather rewrite this as $(e^{\frac{it}{h}P_Q})_{t\in \R}$ for 
$$
P_Q =h \tilde{P}_Q  = h\Lambda \tilde{\chi}_0(h\Lambda) + ih Q .
$$
Note that we have $h\Lambda \tilde{\chi}_0(h\Lambda) \in \Psi^{-\infty}(M)$ with principal symbol $|\xi|_x \tilde{\chi}_0(|\xi|_x)$ on account to functional calculus~\cite[Theorem~14.9]{Zworski:book} (applied to $f(-h^2\Delta)$ with $f(s) =\sqrt{s}  \tilde{\chi}_0(\sqrt{s}) \in C^\infty_c(\R)$ for $\supp( \tilde{\chi}_0) \subset (0,+\infty)$). In particular, for $(x,\xi) \in K_\Gamma$ (defined in~\eqref{e:def-Keps}), this symbol is $|\xi|_x \tilde{\chi}_0(|\xi|_x) = |\xi|_x$.

In the {\em selfadjoint case} (\ie in case $Q=0$ in the above operator $P_Q$), there are now several classical references for the Egorov theorem in Ehrenfest time (that is to  say, in time of order $\kappa \log(h^{-1})$). We mention the pionnering work of Bouzouina-Robert~\cite{BouzouinaRobert}, the appearance of the dynamical constant $\kappa <\Upsilon_{\max}^{-1}$ in the time-range $[0,\kappa \log(h^{-1})]$ in~\cite[Section~5.2]{AnNon}, \cite[Theorem~7.1]{Riv} and~\cite[Section~3.3 and Appendix~C]{DG:14} and the textbook~\cite[Theorem~11.12]{Zworski:book}. A detailed proof is written in~\cite[Section~2.2.2 and Appendix]{DJN:20} in the selfadjoint case on a compact manifold, and we shall use some of the estimates of this reference.

In the non-selfadjoint semiclassical setting, Egorov theorems in $h$-independent time have been proven by Royer in~\cite{Royer:these}, \cite{Royer:10CPDE} and \cite{Royer:10JDE}. A related Egorov theorem in $h$-independent time for {\em non-autonomous} non-selfadjoint operators is available in~\cite[Appendix~A.2]{LL:16}.

Finally, a version of an Egorov theorem in the non-selfadjoint semiclassical setting in Ehrenfest time is proven in~\cite[Appendix~A.3]{Riviere:12} and \cite[Section~4.1]{Riviere:14}. However, in the last two references, the symbols $a,q$ are taken in $S^{0}_0(T^*M)$.
Theorem~\ref{t:egorov} below relaxes this assumption to symbols in $S^{0}_\nu(T^*M)$ and follows slightly more precisely the constant $\kappa$ in the time range of validity $[0,\kappa \log(h^{-1})]$.

\medskip
We recall that $\Upsilon_{\max}$ is defined in~\eqref{e:def-lambdamax} and $K_\Gamma$ in~\eqref{e:def-Keps} (we shall take $\Gamma = 10^{-1}$ in the following to fix ideas), and that the definition of a quantization procedure $\Op_h$ on $M$ is given in Appendix~\ref{app:op-quant}.
\begin{theorem}
\label{t:egorov}
Let $\Upsilon_1> \Upsilon_{\max}$, and $\mu, \nu , \rho \geq 0$ such that 
$$
\delta := \mu \Upsilon_1 + \max(\rho, \nu) <\frac12.
$$
Assume $Q \in \Psi^{0}_\nu (M)$ has principal symbol $q \in S^{0}_\nu(T^*M)$  such that $\Re(q) \geq 0$ on $T^*M$. 
Then, for all $a \in S^{-\infty}_\rho(T^*M)$ such that $\supp(a) \subset \Int(K_{10^{-1}})$, 
the following hold:
\begin{itemize}
\item the symbols 
$$
a_t(x,\xi) := a \circ \varphi^{-t} (x,\xi) e^{- \int_0^t 2\Re(q) \circ \varphi^{-s}(x,\xi)ds}
$$ belong to a bounded set of $S^{-\infty}_{\delta} (T^*M)$ uniformly for $0\leq t \leq \mu \log(h^{-1})$.
\item the estimate  
\begin{align}\label{e:EGOROV}
  \nor{\left(e^{\frac{it}{h}P_Q}\right)^* \Op_h(a) \left(e^{\frac{it}{h}P_Q}\right) - \Op_h \left(a_t  \right)}{\L(L^2)} \leq C't h^{1- 2\delta} 
\end{align}
holds for all $h \in (0,h_0)$, and $0\leq t \leq \mu \log(h^{-1})$.
\end{itemize}
\end{theorem}
\begin{remark}
\label{e:Upsilon=infty}
Recall that $\Upsilon_{\max} \in [0,+ \infty)$. If $\Upsilon_{\max}=0$, this Egorov statement holds up to times $0\leq t \leq \mu \log(h^{-1})$ for any $\mu>0$ (indeed, the requirement on $\Upsilon_1$ is $\Upsilon_1>0$). Anyways, in this situation, the logarithmic bound in Theorems~\ref{t:theorem-holder-0} and~\ref{t:main-res} is not expected to be relevant (see e.g. the discussion related to~\cite{AL:14,Stahn:17,LeLe:17,Sun:21} on tori in Section~\ref{s:literature}).
\end{remark}

\bnp
For $w \in T^*M$ and $\tau \in [0,t]$, we set:
$$
\tilde{a}_t(\tau ,w) = a \circ \varphi^{\tau-t}(w) e^{- \int_\tau^t 2\Re(q) \circ \varphi^{\tau-s} (w)ds} = a \circ \varphi^{\tau-t}(w) e^{- \int_0^{t-\tau} 2\Re(q) \circ \varphi^{-s} (w)ds} ,
$$
so that $a_t(w) = \tilde{a}_t(0 ,w)$.
According to Proposition~\ref{p:symb-log-times}, for $\mu \geq 0$, the family of symbols $\{\tilde{a}_t(\tau, \cdot) , 0\leq \tau \leq t \leq \mu\log(h^{-1})\}$ is a bounded subset of $S^{-\infty}_{\delta} (T^*M)$. 
Moreover, differenciating with respect to $\tau$ the definition of $\tilde{a}_t$ (see~\eqref{e:symbol-tau-egorov} below and Proposition~\ref{p:symb-log-times}) shows that 
\begin{align}
\label{e:dtau-bdd}
\{\d_\tau \tilde{a}_t(\tau, \cdot) , 0\leq \tau \leq t \leq \mu\log(h^{-1})\} \text{ is a bounded subset of } h^{-\rho} S^{-\infty}_{\delta} (T^*M). 
\end{align}
We have, for $z= \varphi^\tau(w) \in T^*M$,  $\tilde{a}_t(\tau ,\varphi^{-\tau}(z)) = a \circ \varphi^{-t}(z) e^{- \int_\tau^t 2 \Re(q) \circ \varphi^{-s} (z)ds}$ and hence 
$$
\frac{d}{d\tau} \left(\tilde{a}_t(\tau ,\varphi^{-\tau}(z)) \right)  =  2\Re(q) \circ \varphi^{-\tau} (z)\tilde{a}_t(\tau ,\varphi^{-\tau}(z)) \in S^{-\infty}_{\delta} (T^*M) ,
$$
as well as (with $p(x,\xi) = |\xi|_x$)
$$
\frac{d}{d\tau} \left(\tilde{a}_t(\tau ,\varphi^{-\tau}(z)) \right)  = (\d_\tau \tilde{a}_t) (\tau ,\varphi^{-\tau}(z)) - \{p ,\tilde{a}_t(\tau ,\cdot ) \}(\varphi^{-\tau}(z)) .
$$
(Note that both terms in the right-hand side of this identity are only in $ h^{-\rho} S^{-\infty}_{\delta} (T^*M)$, whereas their difference is in $S^{-\infty}_{\delta} (T^*M)$).
As a consequence of these two identities, and writing $z= \varphi^\tau(w)$ we obtain 
\begin{equation}
\label{e:symbol-tau-egorov}
 (\d_\tau \tilde{a}_t) (\tau ,w) - \{p ,\tilde{a}_t(\tau ,\cdot ) \}(w) =  2\Re(q)(w) \tilde{a}_t(\tau ,w) .
\end{equation}
Now, assuming $\delta < \frac12$, we quantify $\tilde{a}_t(\tau ,\cdot)$ into $\Op_h( \tilde{a}_t(\tau ,\cdot))$ according to~\eqref{eq: def psiDO on M}. These operators remain in a bounded set of $\Psi^{-\infty}_{\delta} (M)$ (see Appendix~\ref{app:op-quant} for a definition) for $0\leq \tau \leq t \leq \mu\log(h^{-1})$, and, in particular, are bounded on $L^2(M)$ (see Appendix~\ref{s:calcul-pseudo}).
Setting 
$$
A_\tau = \left(e^{\frac{i\tau}{h}P_Q}\right)^* \Op_h( \tilde{a}_t(\tau ,\cdot)) \left(e^{\frac{i\tau}{h}P_Q}\right) ,
$$
and writing $P_0=h\Lambda  \tilde{\chi}_0(h\Lambda)$, we have
\begin{align}
\label{e:estimate-ddtatau}
\frac{d}{d\tau} A_\tau & = \left(e^{\frac{i\tau}{h}P_Q}\right)^* \left(\left(\frac{i}{h}P_Q\right)^*  \Op_h( \tilde{a}_t(\tau ,\cdot)) +  \Op_h( \tilde{a}_t(\tau ,\cdot)) \left(\frac{i}{h}P_Q\right) 
+\Op_h(\d_\tau \tilde{a}_t(\tau ,\cdot))  \right) \left(e^{\frac{i\tau}{h}P_Q}\right)  \nonumber\\
& = \left(e^{\frac{i\tau}{h}P_Q}\right)^* \left(- \frac{i}{h}\left(P_0 - ihQ^* \right) \Op_h( \tilde{a}_t(\tau ,\cdot))  \right. \nonumber \\
& \quad \left. +  \frac{i}{h} \Op_h( \tilde{a}_t(\tau ,\cdot)) \left(P_0 + ihQ \right) 
+\Op_h(\d_\tau \tilde{a}_t(\tau ,\cdot))  \right) \left(e^{\frac{i\tau}{h}P_Q}\right) \nonumber \\
& = \left(e^{\frac{i\tau}{h}P_Q}\right)^* B_h  \left(e^{\frac{i\tau}{h}P_Q}\right) ,
\end{align}
with 
\begin{equation}
\label{e:op-in-egor}
B_h :=\frac{i}{h}\left[  \Op_h( \tilde{a}_t(\tau ,\cdot)) , P_0\right]
- Q^* \Op_h( \tilde{a}_t(\tau ,\cdot)) - \Op_h( \tilde{a}_t(\tau ,\cdot)) Q 
+\Op_h(\d_\tau \tilde{a}_t(\tau ,\cdot)) .
\end{equation}
Applying Corollary~\ref{c:DJN-A6} to $b\in S^{-\infty}(T^*M)$ such that $\supp(b) \subset \{|\xi|_x\leq 10\}$ and $P_0 = \Op_h(b) + O_{\L(L^2(M))}(h^{\infty})$ together with $\tilde{a}_t(\tau ,\cdot)$, which remain in a bounded subset of $S^{-\infty}_{\delta} (T^*M)$ with $\supp (\tilde{a}_t(\tau ,\cdot)) \subset K_{10^{-1}}\subset \{|\xi|_x<10\}$), we obtain that  $\frac{i}{h}\left[  \Op_h( \tilde{a}_t(\tau ,\cdot)) , P_0\right] \in h^{-\delta} \Psi^{-\infty}_\delta(M)$ together with
$$\frac{i}{h}\left[  \Op_h( \tilde{a}_t(\tau ,\cdot)) , P_0\right] = \Op_h \left( \{ \tilde{a}_t(\tau ,\cdot) ,p \} \right)+ O_{\L(L^2(M))}(h^{1-2\delta}) . 
$$
This, together with application of the composition rule (see Appendix~\ref{s:calcul-pseudo}) in the class $\Psi^{-\infty}_{\delta} (M)$ and \eqref{e:op-in-egor} then imply 
$$
B_h = \Op_h \big( \{ \tilde{a}_t(\tau ,\cdot) ,p \} - \ovl{q} \tilde{a}_t(\tau ,\cdot) - q \tilde{a}_t(\tau ,\cdot) + \d_\tau \tilde{a}_t(\tau ,\cdot)  \big) + O_{\L(L^2(M))}(h^{1-2\delta}) ,
$$
for $0\leq \tau \leq t \leq \mu\log(h^{-1})$.
According to~\eqref{e:symbol-tau-egorov}, this is 
$$
B_h = O_{\L(L^2(M))}(h^{1-2\delta}) , \quad \text{ for }0\leq \tau \leq t \leq \mu\log(h^{-1}) .
$$
With Lemma~\ref{l:bound-etip} applied to $ \left(e^{\frac{i\tau}{h}P_Q}\right)$ and $\left(e^{\frac{i\tau}{h}P_Q}\right)^* =  e^{-\frac{i\tau}{h}P_Q^*}$, and the fact that $\Re(q)\geq 0$, 
we can now estimate the operator norm of the right hand-side of~\eqref{e:estimate-ddtatau} as 
$$\nor{\frac{d}{d\tau} A_\tau}{\L(L^2)} \leq  \nor{e^{\frac{i\tau}{h}P_Q^*}}{\L(L^2)} \nor{B_h}{\L(L^2)}\nor{e^{\frac{i\tau}{h}P_Q}}{\L(L^2)} 
 \leq  Ch^{1- 2\delta} e^{2C_0h^{1-2\nu} t}  
\leq C' h^{1- 2\delta},
$$
uniformly for $h \in (0,h_0)$ and $0\leq \tau \leq t \leq \mu \log(h^{-1})$.
Integrating on $[0,t]$ then in particular implies that 
$$
\nor{A_t -A_0}{\L(L^2)}
\leq C't h^{1- 2\delta} \quad  \text{for all }h \in (0,h_0), 0\leq t \leq \mu \log(h^{-1}) ,
$$
which is precisely~\eqref{e:EGOROV}.
\enp

\section{From propagation in logarithmic times to logarithmic lower resolvent bounds}
\label{s:proof-main}
In this section, we conclude the proof of Theorems~\ref{t:theorem-holder-0} and formulate a more precise version in Theorem~\ref{t:main-res} (see also Corollary~\ref{c:ralalalala}). 
To this aim, we combine the abstract result of Theorem~\ref{t:resolvent-implies-long-time} (resolvent estimate implies long time bound) together with long time (logarithmic) propagation of a coherent state, using the Egorov Theorem~\ref{t:egorov}.
Recall that the open set $\omega \subset M$ satisfies the Geometric Control Condition (GCC) if one of the following equivalent conditions is satisfied ($\pi$ denotes the canonical projection $S^*M\to M$): 
\begin{align}
\label{e:GCC-T}
&\text{there is }T>0 \text{ such that for all } \zeta \in S^*M, \text{ there is } t \in (0,T) \text{ such that } \pi \circ\varphi_t(\zeta) \in \omega , \\
\label{e:GCC-infty}
& \text{for all } \zeta \in S^*M, \text{ there is } t >0 \text{ such that } \pi \circ\varphi_t(\zeta) \in \omega .
\end{align}
That~\eqref{e:GCC-infty} implies~\eqref{e:GCC-T} follows from a classical compactness argument (see \eg~\cite[Section~5.1]{LeLe:17}). Recall that $\Upsilon_{\max}$ is defined in~\eqref{e:def-lambdamax}.
\begin{theorem}
\label{t:main-res}
Assume that $b\in C^0(M;\R_+)$ admits a nondecreasing modulus of continuity $\omega$ such that $\omega(\eps)\log(\eps)\to 0$ as $\eps\to 0^+$, and that the set $\{b>0\}$ does not satisfy the Geometric Control Condition. Let $\rho , \mu$ such that
\begin{align}
\label{e:schnorcombre}
\rho \geq 0 ,\quad  \mu>0, \quad \mu \Upsilon_{\max} +  \rho < \frac12 .
\end{align}
Then, for all $\eta>0$, there is $h_0 >0$ such that 
\begin{align*}
 \sup_{\tau \in I_{h^\rho}}\nor{(i\tau/h - \A_m)^{-1}}{\L\big(H^1_m \times L^2\big)} 
\geq \frac{1}{\sqrt{2}}(\mu-\eta) \log(h^{-1}), \quad \text{ for } h \in (0,h_0) ,\\
 \sup_{\tau \in I_{h^\rho}}\nor{(i\tau/h - \A_+)^{-1}}{\L\big(H^1_+ \times L^2\big)} 
\geq \frac{1}{\sqrt{2}}(\mu-\eta) \log(h^{-1}), \quad \text{ for } h \in (0,h_0) ,
\end{align*}
with $I_{h^\rho} := [1-h^\rho, 1+h^\rho]$ if $\rho>0$ and $I_{h^\rho} := [1-\eps, 1+\eps]$ if $\rho=0$, where $\eps>0$ is $h$-independent (but the constant $h_0 >0$ then depends on $\eps$).
\end{theorem}
Setting $\lambda = h^{-1},\lambda_0 = h_0^{-1}$ and $z=\frac{\tau}{h}$, the following two statements become particular cases of Theorem~\ref{t:main-res}.
\begin{corollary}
\label{c:ralalalala}
Under the assumptions of Theorem~\ref{t:main-res}, for  $\bullet=m$ or $\bullet=+$, the following two statements hold. For all $\eps>0$, there are $\rho,\lambda_0>0$ such that 
$$
 \sup_{z \in [\lambda (1-\lambda^{-\rho}) , \lambda (1 + \lambda^{-\rho})]} \nor{(iz - \A_\bullet)^{-1}}{\L\big(H^1_\bullet \times L^2\big)} \geq \left(\frac{1}{2\sqrt{2}\Upsilon_{\max}}-\eps\right) \log(\lambda), \quad \text{ for } \lambda \geq \lambda_0 .
 $$
For all $\eps\in (0,1/2]$, there are $c_\eps, \lambda_0>0$ such that 
$$
 \sup_{z \in [\lambda (1-\lambda^{-\frac12+\eps}) , \lambda (1 + \lambda^{-\frac12+\eps})]} \nor{(iz - \A_\bullet)^{-1}}{\L\big(H^1_\bullet \times L^2\big)}  \geq c_\eps \log(\lambda), \quad \text{ for } \lambda \geq \lambda_0 . 
 $$
\end{corollary} 
The first statement consists in fixing in~\eqref{e:schnorcombre} $\mu < \frac{1}{2\Upsilon_{\max}}$ as close as possible to $\frac{1}{2\Upsilon_{\max}}$; then deduce $\rho$.
 The second statement consists in fixing in~\eqref{e:schnorcombre} $\rho=\frac12-\eps$; then deduce $\mu$. 
 Note that in the first statement of Corollary~\ref{c:ralalalala}, we have implicitly assumed $\Upsilon_{\max}>0$; in case $\Upsilon_{\max}=0$ we refer Remark~\ref{e:Upsilon=infty}. 
 
 \bnp[Proof of Theorem~\ref{t:main-res}]
We first use the formulation of the equation as a first order hyperbolic system as in Theorem~\ref{t:semiclassic-intro} (case $m>0$) and Theorem~\ref{t:semiclassic-++} (case $m=0$). We only give the details for the case $m>0$ for simplicity, the case $m=0$ being treated as in the proof of Theorem~\ref{t:semiclassic-++}.
We put the problem in the setting of Theorem~\ref{t:resolvent-implies-long-time} with $\P_h := h\P_m$ where $\P_m$ is defined in~\eqref{e:damp-wave-hyp}. That is to say,
$$
\H = L^2(M; \C^2) , \quad P_h = h \Lambda_m I_\pm , \quad D(P_h)  =H^2(M;\C^2) ,
 \quad h Q_h = h \frac{b}{2}
 \begin{pmatrix}
1& 1 \\
1 &1
\end{pmatrix} .
$$
Moreover, we write \begin{align}
\label{e:def-eps}
\eps(h) = h^{\rho} , \text{ if }\rho >0 , \quad \text{ or } \quad \eps(h) =\eps \in (0,1)\text{ fixed,  if } \rho=0 , 
\end{align}
 and set
$
T_h := \mu \log(h^{-1}) .
$
We further let $\chi, \tilde{\chi} \in C^\infty_c(\R; [0,1])$ with $\chi, \tilde{\chi} \geq0$, $\supp\chi ,  \tilde{\chi}\subset (-1,1)$, $\chi(0)=1$, and $\tilde{\chi} = 1$ in a neighborhood of $\supp(\chi)$. We then set 
$$
\chi_h(s) = \chi \left((s-1)\eps(h)^{-1}\right) , \quad \tilde{\chi}_h(s) = \tilde{\chi} \left((s-1)\eps(h)^{-1}\right) ,
$$
so that $\supp(\chi_h) \subset \Int(I_\eps)$. We then define (using functional calculus for selfadjoint operators) as in~\eqref{def:Pieps}, 
 \begin{equation}
 \label{e:proj-pipi}
\chi_h \left( P_h \right)  = \chi_h (h \Lambda_m I_\pm ) 
 = 
  \begin{pmatrix}
\chi_h (h \Lambda_m) & 0 \\
0 &\chi_h (- h \Lambda_m)
\end{pmatrix}
 = 
  \begin{pmatrix}
\chi_h (h \Lambda_m) & 0 \\
0 & 0
\end{pmatrix}  ,
\end{equation}
where we have used in the last equality that $\supp \chi_h \subset \R_+^*$ together with $\Sp(-h\Lambda_m) \subset \R^-$.
Note that we have similarly
 \begin{equation*}
\tilde{\chi}_h(P_h)  = 
  \begin{pmatrix}
\tilde{\chi}_h (h \Lambda_m) & 0 \\
0 & 0
\end{pmatrix}  ,
\end{equation*}
and they satisfy $\tilde{\chi}_h( P_h)  \chi_h(P_h)=   \chi_h(P_h)$.
According to Theorem~\ref{t:thm-Kuster}, we have $ \chi_h(P_h),\tilde{\chi}_h( P_h) \in \Psi^0_{\rho}(M;\C^{2\times2})$ with respective principal symbols
$$
\sigma_h \left( \chi_h ( P_h )  \right) = 
  \begin{pmatrix}
\chi_h (|\xi|_x) & 0 \\
0 & 0
\end{pmatrix} ,
\quad 
\sigma_h \left( \tilde{\chi}_h ( P_h )  \right) = 
  \begin{pmatrix}
\tilde{\chi}_h (|\xi|_x) & 0 \\
0 & 0
\end{pmatrix} .
$$
Application of estimate~\eqref{e:estim-pointwise-opt} in Theorem~\ref{t:resolvent-implies-long-time} in this context yields for all $u \in L^2(M;\C^2)$
\begin{align}
\label{e:estim-bis} 
\nor{e^{\frac{i T_h }{h}\P_h}  \chi_h ( P_h )  u}{L^2(M;\C^2)}^2 
  \leq \left(  (2+\tilde{\delta})\frac{G(h)^2}{T_h^2}   + \frac{C_b}{\tilde{\delta}^2}  \frac{h T_h}{\eps(h)^2} \right) \nor{  \chi_h ( P_h )  u }{L^2(M;\C^2)}^2 , \quad \text{for all }\tilde{\delta} \in (0,1),
\end{align}
with, according to~\eqref{e:hyp-resP}
\begin{align}
\label{e:sup-res}
G(h) = h \sup_{\tau \in I_\eps} \nor{(\tau  - \P_h)^{-1}}{\L\big(L^2(M;\C^2)\big)} =  h \sup_{\tau \in I_\eps} \nor{(\tau  - h\P_m)^{-1}}{\L\big(L^2(M;\C^2)\big)}  = \sup_{\tau \in I_\eps}\nor{(i\tau/h - \A_m)^{-1}}{\L\big(H^1_m \times L^2\big)} 
 \end{align}
where we have used~\eqref{e:totototo}-\eqref{e:totototo-2} in the last equality.
Our goal is to produce a lower bound for the left-hand side and an upper bound for the right hand-side of Estimate~\eqref{e:estim-bis} 
for a well chosen family of functions $u=u_h$.  To this aim, we proceed in several steps.
 
\medskip
\noindent
\textbf{Step 1: replacing $b$ by $b_{h^\nu}$.}
We first replace the damping function $b$ by its regularization $b_{h^\nu}$ given by Corollary~\ref{cor:def-b-eps}, for some $0<\nu< 1/2$ (to be fixed later on). We denote by 
$$
\P^\nu := \Lambda_m I_\pm+ i \frac{b_{h^\nu}}{2}
 \begin{pmatrix}
1& 1 \\
1 &1
\end{pmatrix} .
$$
and using Lemma~\ref{l:perturbation} and Corollary~\ref{cor:def-b-eps}, we have that
\begin{align}
\label{e:replace-b-bnu}
\nor{e^{it\P} - e^{it\P^\nu}}{\L(L^2)} \leq t \nor{b-b_{h^\nu}}{L^\infty(M)}  \leq  t C_0 \omega(\kappa h^\nu) , \quad t \geq 0 . 
\end{align}
The interest is that now $b_{h^\nu}\in \Psi^0_\nu(M)$. 

\medskip
\noindent
\textbf{Step 2: Replacing $\P^\nu$ by its diagonal part $\P_{\diag}$ and reducing to a scalar equation.} 
According to Corollary~\ref{c:energy-cutoff-final}, we have 
$$
\nor{e^{it\P^\nu} \chi_0(h\Lambda_m) I_2 -  e^{it\P_{\diag}} \chi_0(h\Lambda_m) I_2}{\L(L^2)}  \leq C\left(th^{1-2\nu}+h\right) ,   
$$
where $\chi_0,\tilde{\chi}_0$ are fixed cutoff functions defined in~\eqref{e:def-cut-off-energy-0} and  
 \begin{align}
\label{e:def-Pd}
\P_{\diag} =
\begin{pmatrix}
\Lambda_m \tilde{\chi}_0(h\Lambda_m) &0  \\
0 &-\Lambda_m \tilde{\chi}_0(h\Lambda_m)
\end{pmatrix}   
    + i
 \begin{pmatrix}
\tilde{\chi}_0(h\Lambda_m) b_{h^\nu} \tilde{\chi}_0(h\Lambda_m) &0  \\
0 &\tilde{\chi}_0(h\Lambda_m) b_{h^\nu} \tilde{\chi}_0(h\Lambda_m)
\end{pmatrix} .
\end{align}
Moreover, notice that $\chi_0(h\Lambda_m)\chi_h(h\Lambda_m) =\chi_h(h\Lambda_m)$ and thus, for all $v\in L^2(M;\C^2)$
$$
\nor{e^{it\P^\nu} \chi_h(h\Lambda_m) v -  e^{it\P_{\diag}} \chi_h(h\Lambda_m) v}{L^2(M;\C^2)}  \leq C\left(th^{1-2\nu}+h\right) \nor{\chi_h(h\Lambda_m) v}{L^2(M;\C^2)} ,  
$$
As a consequence, writing $u=(u_1,u_2)^t$ and remarking that $\chi_h( P_h) u = \chi_h(h\Lambda_m)I_2 (u_1,0)^t$ on account to~\eqref{e:proj-pipi}, 
we have $$e^{it\P^\nu} \chi_h( P_h) u =  e^{it\P^\nu} \chi_h(h\Lambda_m)\begin{pmatrix}
 u_1\\
0
\end{pmatrix},
$$ and thus 
$$
e^{it\P^\nu} \chi_h( P_h) u = 
e^{it\P_{\diag}} \chi_h(h\Lambda_m) (u_1,0)^t
+ O\left(th^{1-2\nu}+h\right) \nor{u}{L^2(M;\C^2)}
$$
Recalling the expression of $\P_{\diag}$ in~\eqref{e:def-Pd}, and combining with~\eqref{e:replace-b-bnu} we have now obtained, for all $u=(u_1,u_2)\in L^2(M;\C^2)$
\begin{align}
\label{e:numerobis}
\nor{e^{it\P} \chi_h( P_h) u}{L^2(M;\C^2)} = 
\nor{e^{it \P_{\hw}} \chi_h(h\Lambda_m) u_1 }{L^2(M)}
+ O\left(th^{1-2\nu}+h + t\omega(\kappa h^\nu)\right) \nor{u}{L^2(M;\C^2)} ,
\end{align}
with 
$$
\P_{\hw} = \Lambda_m \tilde\chi_0(h\Lambda_m) + i \tilde\chi_0(h\Lambda_m) b_{h^\nu} \tilde\chi_0(h\Lambda_m) .
$$

\medskip
\noindent
\textbf{Step 3: Using the Egorov theorem.} 
 From now on, we set $\delta= \max(\rho,\nu)$.
Using Lemma~\ref{l:1-BB}, we write for $t\geq 0$
$$
\nor{e^{it \P_{\hw}} \chi_h(h\Lambda_m) -\chi_h(h\Lambda_m)e^{it \P_{\hw}} }{\L(L^2)} \leq t \nor{ [\P_{\hw} ,  \chi_h(h\Lambda_m) ]}{\L(L^2)}
\leq t \nor{ [b_{h^\nu},  \chi_h(h\Lambda_m) ]}{\L(L^2)} \leq  t h^{1-2\delta},
$$
on account to the $\Psi^0_{\delta}(M)$ calculus. Hence, $e^{it \P_{\hw}} \chi_h(h\Lambda_m) = \chi_h(h\Lambda_m)e^{it \P_{\hw}} + O(t h^{1-2\delta})$.
Next, we evaluate  
$$
\nor{\chi_h(h\Lambda_m)e^{it \P_{\hw}} u_1}{L^2(M)}^2
= \left( \left(e^{it \P_{\hw}}\right)^\ast \chi_h(h\Lambda_m)^2 \left( e^{it \P_{\hw}} \right) u_1 , u_1\right)_{L^2(M)} .
$$
Using the Egorov Theorem~\ref{t:egorov}, together with the fact that $(x,\xi)\mapsto |\xi|_x$ is invariant by $\varphi^{-t}$, we have for any $\mu$ such that $\mu \Upsilon_1 +\delta <\frac12$, $h \in (0,h_0)$, and $0\leq t \leq \mu \log(h^{-1})$,
\begin{align*}
\left(e^{it \P_{\hw}}\right)^\ast \chi_h(h\Lambda_m)^2 \left( e^{it \P_{\hw}} \right)
=  \Op_h \left(\chi_{h^\rho}^2(|\xi|_x)e^{- 2 \int_0^t  b_{h^\nu} \circ \pi \circ \varphi^{-s} (x,\xi)ds} \right)  + R'_t ,  \quad
 \nor{R'_t}{\L(L^2)} \leq C' t h^{1- 2(\mu \Upsilon_1 + \delta)} . 
\end{align*}
This yields 
\begin{align}
\label{e:egorov-use}
\nor{e^{it \P_{\hw}} \chi_h(h\Lambda_m) u_1 }{L^2(M)}^2
 =  \left( \Op_h \left(\chi_h^2(|\xi|_x)e^{- 2 \int_0^t  b_{h^\nu} \circ \pi \circ \varphi^{-s} (x,\xi)ds} \right) u_1 , u_1 \right)_{L^2} + \left( R'_t u_1 , u_1 \right)_{L^2} ,
\end{align}
with $\nor{R'_t}{\L(L^2)} \leq C' t h^{1- 2(\mu \Upsilon_1 +\delta)}$. 
 
\medskip
\noindent
\textbf{Step 4: Using the lack of GCC and a coherent state.}
Now, the assumption that GCC is not satisfied ensures (see~\eqref{e:GCC-infty}) the existence of $\zeta_0 = (x_0, \xi_0) \in T^*M$ such that $|\xi_0|_{x_0} = 1$ and $b \circ \pi \circ \varphi^{-s} (\zeta_0) = 0$ for all $s \in \R_+$. 
As a consequence of Corollary~\ref{cor:def-b-eps}, we have $b_{h^\nu}=0$ on $\{b=0\}$, and hence 
\begin{align}
\label{e:bnu-0}
b_{h^\nu} \circ \pi \circ \varphi^{-s} (\zeta_0) = 0 \quad \text{ for all } s \in \R_+ .
\end{align}
According to Corollary~\ref{c:coherent-state}, there is an $h$-dependent family of functions $u_h^{\zeta_0} \in C^\infty(M)$ such that $\|u_h^{\zeta_0}\|_{L^2(M)}=1$ and for all $\delta \in [0,1/2)$ and  $a \in S^m_\delta(T^*M)$, we have
\begin{align*}
\left| \left(\Op_h(a)  u_h^{\zeta_0} , u_h^{\zeta_0} \right)_{L^2} - a(\zeta_0) \right| \leq C_a h^{\frac12-\delta} , \quad h \in (0,h_0) ,
\end{align*}
where $C_a$ depends on a finite number of seminorms of $a$ in $S^m_\delta(T^*M)$. 
Taking $u_1 = u_h^{\zeta_0}$ in Equation~\eqref{e:egorov-use} yields, for $\mu$ such that $\mu \Upsilon_1 + \delta <\frac12$, and $0\leq t \leq T_h= \mu \log(h^{-1})$,
\begin{align}
\label{e:asymptotics-log-t}
\nor{e^{it \P_{\hw}} \chi_h(h\Lambda_m) u_h^{\zeta_0} }{L^2(M)}^2
& =  \left( \Op_h \left( \chi_h^2(|\xi|_x)e^{- 2 \int_0^t  b_{h^\nu} \circ \pi \circ \varphi^{-s} (x,\xi)ds} \right) u_h^{\zeta_0} , u_h^{\zeta_0} \right)_{L^2} +  O\left( t h^{1- 2(\mu \Upsilon_1 + \delta)} \right)\nonumber \\
& = \chi_h^2(|\xi_0|_{x_0})e^{- 2 \int_0^t  b_{h^\nu} \circ \pi \circ \varphi^{-s} (x_0,\xi_0)ds} + O\left(h^{\frac12-\delta}\right) +  O\left( t h^{1- 2(\mu \Upsilon_1 + \delta)} \right) \nonumber \\
& = 1 + O\left(h^{\frac12-(\mu \Upsilon_1 + \delta)}\right) +  O\left( t h^{1- 2(\mu \Upsilon_1 + \delta)} \right) ,
\end{align}
where we have used that $|\xi_0|_{x_0}=1$ together with~\eqref{e:bnu-0}. We have also used in the second line that the symbol $\chi_h^2(|\xi|_x)e^{- 2 \int_0^t  b_{h^\nu} \circ \pi \circ \varphi^{-s} (x,\xi)ds}$ remains in a bounded set of $S^{-\infty}_{\mu \Upsilon_1 + \delta} (T^*M)$ uniformly for $0\leq t \leq \mu \log(h^{-1})$, as a consequence of the first item in Theorem~\ref{t:egorov}.

Combined with~\eqref{e:numerobis}, this now reads for $u=(u_h^{\zeta_0},0)^t$,
\begin{align}
\label{e:numeroter}
\nor{u}{L^2(M;\C^2)} = 1 , \qquad 
\nor{e^{it\P} \chi_h( P_h) u}{L^2(M;\C^2)}
= 1 + O\left(h^{\frac12-(\mu \Upsilon_1 + \delta)} + t h^{1- 2(\mu \Upsilon_1 + \delta)} +  t\omega(\kappa h^\nu)  \right) ,
\end{align}
for  $\mu \Upsilon_1 + \delta <\frac12$ and $0\leq t \leq T_h= \mu \log(h^{-1})$. 

\medskip
\noindent
\textbf{Step 5: Application of~\eqref{e:estim-bis} and conclusion of the proof.}
We come back to~\eqref{e:estim-bis}: we notice that~\eqref{e:numeroter} furnishes an asymptotics of the left-hand side for $t= T_h= \mu \log(h^{-1})$, whereas $\nor{  \chi_h( P_h) u}{L^2(M;\C^2)} \leq \nor{u}{L^2(M;\C^2)} =1$. 
Applying~\eqref{e:estim-bis} yields, for all $\tilde{\delta}\in (0,1)$, 
$$
1 - O\left(h^{\frac12-(\mu \Upsilon_1 + \delta)} + \log(h^{-1}) h^{1- 2(\mu \Upsilon_1 + \delta)} +  \log(h^{-1})\omega(\kappa h^\nu)  \right)
  \leq    (2+\tilde{\delta})\frac{G(h)^2}{\mu^2 \log(h^{-1})^2}   + \frac{C_b \mu}{\tilde{\delta}^2}  \frac{h \log(h^{-1})}{\eps(h)^2}  
$$
where, recall,  $\eps(h)$ is given by~\eqref{e:def-eps}. In any case, using $(2+\tilde{\delta})^{-1} \geq \frac{1-\tilde{\delta}}{2}$, this yields, uniformly in $\tilde{\delta}\in (0,1)$, asymptotically as $h \to 0^+$,
\begin{equation}
\label{e:tatata}
\frac{1-\tilde{\delta}}{2}\left( 1 - O\left(h^{\frac12-(\mu \Upsilon_1 + \delta)} + \log(h^{-1}) h^{1- 2(\mu \Upsilon_1 + \delta)} +  \log(h^{-1})\omega(\kappa h^\nu) +  \tilde{\delta}^{-2}   h^{1-2\rho} \log(h^{-1}) \right)\right)
  \leq \frac{G(h)^2}{\mu^2 \log(h^{-1})^2}  .
\end{equation}
According to the assumption, we have $\log(h^{-1}) \omega(\kappa h^\nu)  = \nu^{-1} \log(\kappa)\omega(\kappa h^\nu) - \nu^{-1}\log(\kappa h^\nu) \omega(\kappa h^\nu) \to 0$ as $h \to 0$. 
Hence, choosing \eg $\tilde{\delta}=h^{\frac{1-2\rho}{4}}$, we have obtained that for any $\Upsilon_1>\Upsilon_{\max}$, any $\nu \in(0,1/2),\rho \in [0,1/2)$ and $\mu$ such that $\mu \Upsilon_1+ \max(\nu,\rho)< \frac12$ we have
\begin{equation}\label{e:ponpon}
\frac{1}{2}- o(1)  \leq \frac{G(h)^2}{\mu^2 \log(h^{-1})^2}  , \quad \text{ as }h\to0^+ .
\end{equation}
where $o(1)$ depends on the choice of $\Upsilon_1,\rho,\nu, \mu$.  If $\rho>0$, we choose $\nu =\rho$ and we see that if~\eqref{e:schnorcombre} is satisfied, then we may choose $\Upsilon_1>\Upsilon_{\max}$ such that $\mu \Upsilon_1+ \max(\nu,\rho)< \frac12$ is satisfied.
If $\rho=0$,  we see that if~\eqref{e:schnorcombre} is satisfied, then we may choose $\Upsilon_1>\Upsilon_{\max}$ and $\nu >0$ (small enough) such that $\mu \Upsilon_1+ \max(\nu,\rho)< \frac12$ is satisfied. In both cases,~\eqref{e:ponpon} holds.

Recalling the definition of $G(h)$ in~\eqref{e:sup-res} then concludes the proof of the theorem.
\enp

\begin{remark}
It also follows from~\eqref{e:tatata} that there exists $\gamma_0>0$ such that if $\limsup_{\eps\to 0^+} \omega(\eps)\log(\eps) <\gamma_0$, then we have $\limsup_{h\to 0^+} \frac{G(h)^2}{\log(h^{-1})^2}  >0$.
\end{remark}

\section{Resolvent estimates and semigroups: abstract setting}
\label{s:abstract-section}
In this section, we first introduce an abstract setting for proving results like Theorem~\ref{t:thm-classiq-A-0}. Second, we put the Klein-Gordon/wave equations in this setting and check the abstract assumptions (using semiclassical analysis), to finally deduce a proof of Theorem~\ref{t:thm-classiq-A-0}. Note that this section somehow presents a non-semiclassical counterpart of Section~\ref{s:semiclassical-resolvent} and is independent from Sections~\ref{s:semiclassical-resolvent},~\ref{s:egorov} and~\ref{s:proof-main}.
Let $\scrH$ be a complex Hilbert space, let 
\begin{itemize}
\item $\scrP_0: D(\scrP_0) \subset \scrH \to \scrH$ be a selfadjoint operator,
\item $\scrB: D(\scrB) \subset \scrH \to \scrH$ be a closed {\em accretive} operator on $\scrH$, in the sense $\Re(\scrB u, u)_{\scrH} \geq 0$ for all $u \in D(\scrB).$
\end{itemize} 
We assume further that $D(\scrP_0) \subset D(\scrB) \subset \scrH$, so that ($\scrB$ has dense domain and), with 
$$
\scrP = \scrP_0 +i  \scrB, 
$$
we have $D(\scrP)= D(\scrP_0)$. Note also that 
\begin{align*}
\Im(\scrP u , u)_{\scrH} = \Re(\scrB u , u)_{\scrH} \geq 0, \quad \text{ for all }u \in D(\scrP_0).
\end{align*}
 We assume that $i\scrP$ generates a semigroup $(e^{it\scrP})_{t\in \R_+}$ (which, given accretivity of $\scrB$ and selfadointness of $\scrP_0$, is equivalent to $i\scrP$ being {\em maximal} dissipative). This is \eg the case by classical perturbative arguments if there is $a \in [0,1)$ and $b>0$ such that $\|\scrB u \|_\scrH \leq a \|\scrP_0 u\|_\scrH+ b \|u\|_\scrH$ for all $u \in D(\scrP_0)$, see \eg~\cite[Chapter~3, Theorem~3.2]{Pazy:83}.
The assumptions on $\scrP$ imply that $\{z\in \C,  \Im(z)<0\} \subset \rho(\scrP)$ and that $(e^{it\scrP})_{t\in \R_+}$ is a contraction semigroup.

\begin{remark}
Note that if one starts with an operator $\scrP$ such that $i\scrP$ is maximal dissipative, there is no uniqueness in the decomposition $\scrP=\scrP_0 +i  \scrB$ with $\scrP_0 $ and $\scrB$ satisfying the above assumptions.
A natural choice might be to take $\P_0=\Re(\P)=\frac{1}{2}(\scrP+\scrP^*)$ and $\scrB =\Im(\scrP) = \frac{1}{2i}(\scrP-\scrP^*)$.
\end{remark}
%

\subsection{Modified resolvent estimates}
\label{s:modified-resolvent}
Now, given a function $\mathsf{M} \in C^0(\R;\R_+^*)$, bounded from below, we define $\mathsf{M}(\scrP_0)$, with domain $D(\mathsf{M}(\scrP_0))$, and $\mathsf{M}(\scrP_0)^{-1}= \left(\frac{1}{\mathsf{M}}\right)(\scrP_0)$ from continuous calculus of selfadjoint operators.
Note that $\mathsf{M}$ is bounded from below so that $\mathsf{M}(\scrP_0)^{-1}$ is a bounded operator. 

For concision, we set $\mathsf{M}_\eps\left(\cdot\right) = \mathsf{M}\left(\frac{\cdot}{1-\eps}\right)$
and hence $\mathsf{M}_\eps\left(\scrP_0\right) = \mathsf{M}\left(\frac{\scrP_0}{1-\eps}\right)$.

 \begin{lemma}
\label{l:res-est-res}
Under the above structure Assumptions on $\scrP$, we assume further that 
\begin{enumerate} 
\item\label{asspt-res-estim} $\R \subset \rho(\scrP)$ and there is an even function $\mathsf{M} \in C^0(\R;\R_+^*)$, nondecreasing on $\R_+$, such that $\nor{(\tau - \scrP)^{-1}}{\L(\scrH)}\leq \mathsf{M}(\tau)$ for all $\tau \in \R$;
\item\label{asspt-reg-B} there exist $\eps \in (0,1),\delta \in (0,1]$ and $C>0$ such for all $u \in D(\scrP_0)$,
$$
 \nor{\langle \scrP_0 \rangle^{-1+\delta} \mathsf{M}_\eps(\scrP_0)^{-1} \scrB u}{\scrH} \leq C \nor{\mathsf{M}_\eps(\scrP_0)^{-1} u}{\scrH} .
$$
\end{enumerate}
Then there exists $C_0 >0$ such that 
\begin{align}
\nor{\mathsf{M}_\eps\left(\scrP_0\right)^{-1}  u }{\scrH} \leq C_0 \nor{(\scrP+ i\alpha - \tau)u }{\scrH} , \quad \text{ for all } u \in D(\scrP_0),\tau \in \R , \alpha \in [0,1]\label{e:bdd-res-op-1}  .
\end{align}
\end{lemma}
Assumption~\ref{asspt-res-estim} is a standard resolvent estimate. Assuming $\R \subset \rho(\scrP)$, the smallest function $\mathsf{M}$ satisfying Assumption~\ref{asspt-res-estim} is 
$$
\mathsf{M}_\scrP (\tau) := \sup_{r\in \R, |r|\leq |\tau|}\nor{(r - \scrP)^{-1}}{\L(\scrH)}, \quad \tau \in \R .
$$
Assumption~\ref{asspt-reg-B} may be formally rephrased as $\langle \scrP_0 \rangle^{-1+\delta} \mathsf{M}_\eps\left(\scrP_0\right)^{-1} \scrB \mathsf{M}_\eps\left(\scrP_0\right)$ extending as a bounded operator $\scrH \to \scrH$. It will be satisfied for instance if  $\scrB$ commutes reasonably well with $\scrP_0$ (and hence with $\mathsf{M}_\eps\left(\scrP_0\right)$), and if $\langle \scrP_0 \rangle^{-1+\delta} \scrB$ is bounded (that is, if $\scrB$ is not too unbounded with respect to $\scrP$).
We will then need to produce conditions under which an upper bound function $\mathsf{M}$ on the resolvent (\ie satisfying Assumption~\ref{asspt-res-estim}) also satisfies Assumption~\ref{asspt-reg-B}. This will be achieved in Section~\ref{s:Mand-asspt-2} below.


\bnp
Let us first check that Assumption~\ref{asspt-res-estim} implies 
\begin{align}
\label{e:resolvent-half-plane}
\nor{(\scrP+i \alpha - \tau)^{-1}}{\L(\scrH)} \leq 2 \mathsf{M}(\tau) \quad \text{for all } \tau \in \R, \alpha \geq 0 .
\end{align}
Indeed, from the resolvent identity, we have for all $\alpha \geq 0$
\begin{align*}
\nor{(\scrP+i \alpha - \tau)^{-1}}{\L(\scrH)}& \leq \nor{(\scrP-\tau)^{-1}}{\L(\scrH)} + \alpha\nor{(\scrP+i \alpha - \tau)^{-1}}{\L(\scrH)}  \nor{(\scrP-\tau)^{-1}}{\L(\scrH)}\\
& \leq \mathsf{M}(\tau) \left( 1+ \alpha \nor{(\scrP+i \alpha - \tau)^{-1}}{\L(\scrH)} \right) .
\end{align*}
As a consequence, we obtain
\begin{align*}
\nor{(\scrP+i \alpha - \tau)^{-1}}{\L(\scrH)} \leq \frac{\mathsf{M}(\tau)}{1-\alpha \mathsf{M}(\tau)} , \quad \text{ for all }\tau \in \R , \alpha \in \left[0, \mathsf{M}(\tau)^{-1}\right) ,
\end{align*}
and in particular, 
\begin{align*}
\nor{(\scrP+i \alpha - \tau)^{-1}}{\L(\scrH)} \leq 2 \mathsf{M}(\tau), \quad \text{ for all }\tau \in \R , \alpha \in \left[0, (2\mathsf{M}(\tau))^{-1}\right] .
\end{align*}
Since $i\scrP$ generates a contraction semigroup, we also have $\nor{(\scrP+i \alpha - \tau)^{-1}}{\L(\scrH)}  \leq \frac{1}{\alpha} \leq 2 \mathsf{M}(\tau)$ for all $\alpha \geq (2\mathsf{M}(\tau))^{-1}$ and~\eqref{e:resolvent-half-plane} follows from the last two inequalities.


Now let $\eps\in (0,1)$ be given by Assumption~\ref{asspt-reg-B} and define $\chi \in C^\infty_c(\R; [0,1])$ such that $\supp\chi \subset (1-\eps, 1+\eps)$ and $\chi= 1$ on $(1-\eps/2, 1+\eps/2)$. We write 
\begin{align}
\label{e:estimMA_0}
\nor{\mathsf{M}_\eps\left(\scrP_0\right)^{-1} u }{\scrH} \leq\nor{\chi\left(\scrP_0/\tau\right)\mathsf{M}_\eps\left(\scrP_0\right)^{-1}u}{\scrH}+\nor{\left(1-\chi\right)\left(\scrP_0/\tau\right)\mathsf{M}_\eps\left(\scrP_0\right)^{-1}u}{\scrH}.
\end{align}
Concerning the first term, on $\supp\chi(s/ \tau)$, we have $|s/\tau|=s/\tau\geq1-\eps>0$, hence $\frac{|s|}{1-\eps}\geq |\tau|$, and thus $\mathsf{M}(\tau)=\mathsf{M}(|\tau|)  \leq \mathsf{M}(\frac{|s|}{1-\eps}) = \mathsf{M}(\frac{s}{1-\eps})$ since $\mathsf{M}$ is even and non-decreasing on $\R_+$. Inverting this inequality implies 
$\mathsf{M}(\frac{\scrP_0}{1-\eps})^{-1}\chi\left(\scrP_0/\tau\right) \leq \mathsf{M}(\tau)^{-1}\chi\left(\scrP_0/\tau\right)$ in  the sense of selfadjoint operators. As a consequence, we have
$$
\nor{\mathsf{M}_\eps(\scrP_0)^{-1}\chi\left(\scrP_0/\tau\right) u}{\scrH} \leq \mathsf{M}(\tau)^{-1} \nor{\chi\left(\scrP_0/\tau\right)u}{\scrH} \leq \mathsf{M}(\tau)^{-1} \nor{u}{\scrH}.
$$
From~\eqref{e:resolvent-half-plane}, we deduce $\nor{u}{\scrH}\leq 2\mathsf{M}(\tau) \nor{(\scrP+i\alpha-\tau)u}{\scrH}$ for all $u \in D(\scrP_0)$, $\tau \in \R, \alpha \geq 0$, which, combined with the previous line, yields
\begin{equation}
\label{e:MchiA_0est}
\nor{\mathsf{M}_\eps(\scrP_0)^{-1}\chi\left(\scrP_0/\tau\right) u}{\scrH} \leq 2\nor{(\scrP+ i \alpha-\tau )u}{\scrH}. 
\end{equation}
We now estimate the second term in~\eqref{e:estimMA_0}. To this aim, notice first that there is a constant $C_\eps>0$ such that
$$
C_\eps |z-1| \geq |z| + 1 ,\quad \text{ for all } z \in \R \setminus [1-\eps/2 , 1+\eps/2] .
$$
Recalling that $\chi=1$ on $(1-\eps/2, 1+\eps/2)$ and $\chi \geq 0$, this implies 
$$
1-\chi(z) \leq C_\eps \frac{|z-1|}{|z|+1} (1-\chi(z))  ,\quad \text{ for all } z \in \R, 
$$
and thus 
$$
1-\chi(s/\tau) \leq C_\eps\frac{|s-\tau|}{|s| + |\tau|}\left( 1-\chi(s/\tau)  \right) , \quad s,\tau \in \R , \tau\neq 0.
$$
As a consequence, we have, for all $|\tau| \geq 1$,
\begin{align*}
\nor{\left(1-\chi\right)\left(\scrP_0/\tau\right)\mathsf{M}_\eps(\scrP_0)^{-1}u}{\scrH} 
& \leq C_\eps \nor{\left(1-\chi\right)\left(\scrP_0/\tau\right)(\scrP_0-\tau)(|\scrP_0| + |\tau| )^{-1}\mathsf{M}_\eps(\scrP_0)^{-1}u}{\scrH}\\
& \leq C_\eps \nor{(|\scrP_0| + |\tau| )^{-1}\mathsf{M}_\eps(\scrP_0)^{-1}(\scrP_0-\tau)u}{\scrH} .
\end{align*}
Writing $(\scrP_0-\tau)u = (\scrP+i \alpha-\tau)u - i(\scrB+\alpha) u$, and using the boundedness of $(|\scrP_0| + |\tau|)^{-1}$ for $|\tau|\geq 1$ and $\mathsf{M}_\eps(\scrP_0)^{-1}$, this implies 
\begin{align*}
\nor{\left(1-\chi\right)\left(\scrP_0/\tau\right)\mathsf{M}_\eps(\scrP_0)^{-1}u}{\scrH} 
& \leq C_\eps \nor{(\scrP+i\alpha-\tau) u}{\scrH}
 +  C_\eps \nor{(|\scrP_0| + |\tau| )^{-1}\mathsf{M}_\eps(\scrP_0)^{-1} \scrB u}{\scrH} +C_\eps\nor{(|\scrP_0| + |\tau| )^{-1}\mathsf{M}_\eps(\scrP_0)^{-1} u}{\scrH} ,
\end{align*}
for $|\alpha| \leq 1$.
Now, from the Young inequality, we remark that $|\tau|^\delta \langle s \rangle^{1-\delta} \leq |\tau| + \langle s \rangle \leq 2(|\tau|+|s|)$ for $|\tau|\geq 1$. Hence using that $(|s| + |\tau| )^{-1} \leq 2 |\tau|^{-\delta}\langle s \rangle^{-1+\delta}$, we obtain, for all $\delta \in [0,1]$,
\begin{align*}
\nor{\left(1-\chi\right)\left(\scrP_0/\tau\right)\mathsf{M}_\eps(\scrP_0)^{-1}u}{\scrH} 
& \leq C_\eps \nor{(\scrP+i\alpha-\tau) u}{\scrH}
 +  2C_\eps |\tau|^{-\delta} \nor{\langle \scrP_0 \rangle^{-1+\delta} \mathsf{M}_\eps(\scrP_0)^{-1} \scrB u}{\scrH} +C_\eps |\tau|^{-\delta}\nor{\mathsf{M}_\eps(\scrP_0)^{-1} u}{\scrH}\\
 &  \leq C_\eps \nor{(\scrP+i\alpha-\tau) u}{\scrH}
 +   C_\eps' |\tau|^{-\delta} \nor{\mathsf{M}_\eps(\scrP_0)^{-1} u}{\scrH} ,
\end{align*}
for $\tau \geq 1,|\alpha| \leq 1$, where, in the last line, we have used Assumption~\ref{asspt-reg-B}.

Combining the last two lines with~\eqref{e:estimMA_0}-\eqref{e:MchiA_0est}, we have obtained, for all $|\tau|\geq 1, \alpha \in [0,1]$,  
$$
\nor{\mathsf{M}_\eps(\scrP_0)^{-1} u }{\scrH} \leq (1+ C_\eps)  \nor{(\scrP+i\alpha-\tau) u}{\scrH}
 +  C_\eps'' |\tau|^{-\delta} \nor{\mathsf{M}_\eps(\scrP_0)^{-1} u}{\scrH}.
$$
which proves the existence of $\tau_0\geq1$ large such that~\eqref{e:bdd-res-op-1} holds for all $|\tau| \geq \tau_0$.

For $|\tau|\leq \tau_0$, we use the boundedness of $\mathsf{M}_\eps(\scrP_0)^{-1}$ together with~\eqref{e:resolvent-half-plane} to write
$$
\nor{\mathsf{M}_\eps(\scrP_0)^{-1}  u }{\scrH} \leq C \nor{u}{\scrH} \leq C \mathsf{M}(\tau_0) \nor{(\scrP+i\alpha-\tau) u}{\scrH} \quad \text{ for all }|\tau| \leq \tau_0 , \alpha \geq0.
$$
This concludes the proof of~\eqref{e:bdd-res-op-1} for all $\tau \in \R$.
\enp

\subsection{From resolvents to integrated semigroup}
\label{s:from-res-to-integrated}

The next result is essentially that of~\cite[Theorem~4.7]{BCT:16} or~\cite[Theorem~2.3]{Miller:12} in the present context. It is also very close to~\cite{BZ:04} and some proofs of the Gearhart-Huang-Pr\"uss-Greiner Theorem, see \eg~\cite[Chapter~V.1]{EN:book} or~\cite{HS:10}. We formulate the result with an abstract operator $\scrC$; we will only use it with $\scrC = \mathsf{M}_\eps\left(\scrP_0\right)^{-1}$ through Lemma~\ref{l:res-est-res} (for which Assumption~\eqref{e:bdd-res-op-1-bis} below is precisely~\eqref{e:bdd-res-op-1}).

\begin{lemma}
\label{l:Psi-T-classical}
Let $\scrP$ be such that $\R \subset \rho(\scrP)$, $i\scrP$ generates a contraction semigroup $(e^{it\scrP})_{t\in\R_+}$, and $\scrC$ be a bounded operator. Assume there is $\alpha_0>0$ such that 
\begin{align}
 \label{e:bdd-res-op-1-bis}
\nor{\scrC u }{\scrH} \leq C_0 \nor{(\scrP- \tau +i \alpha)u }{\scrH} , \quad \text{ for all } u \in D(\scrP_0),\tau \in \R, \alpha \in [0,\alpha_0] .
\end{align}
Then, for all $\Psi \in H^1_{\loc}(\R)$ such that $\supp(\Psi) \subset [0,+\infty)$ and $\Psi' \in L^2(\R)$, for all $T>0$ and all $u \in \scrH$, we have
\begin{align}
\label{e:Psi-T-classical}
\frac{1}{T} \int_\R \Psi \left(\frac{t}{T}\right)^2 \nor{\scrC e^{it\scrP}u}{\scrH}^2 dt \leq \frac{C_0^2}{T^2}\nor{\Psi'}{L^2(\R)}^2 \nor{u}{\scrH}^2 .
\end{align}

\end{lemma}

Note that the choice made in~\cite[Theorem~4.7]{BCT:16} is to take $\Psi(t) =  \mathds{1}_{\R_+}(t) \min(t,1)$, so that $\Psi' = \mathds{1}_{[0,1]}$. In this case,~\eqref{e:Psi-T-classical} reads
$$
\frac{1}{T} \int_0^T t^2\nor{\scrC e^{it\scrP}u}{\scrH}^2 dt 
+ T  \int_T^{\infty} \nor{\scrC e^{it\scrP}u}{\scrH}^2 dt\leq C_0^2 \nor{u}{\scrH}^2 , \quad \text{ for all }T >0, u \in \scrH .
$$

Another way of formulating the result of Lemma~\ref{l:Psi-T-classical} is simply to write 
 \begin{align*}
\nor{ \psi \scrC e^{it\scrP}}{L^2(\R;\L(\scrH))} 
\leq \nor{\psi'}{L^2(\R)}\sup_{\tau \in \R, \alpha \in [0,\alpha_0]}\nor{\scrC(\scrP- \tau+i\alpha)^{-1}}{\L(\scrH)} ,
\end{align*}
 for all $\psi \in H^1_{\loc}(\R)$ such that $\supp(\psi) \subset [0,+\infty)$ and $\psi' \in L^2(\R)$.

\bnp[Proof of Lemma~\ref{l:Psi-T-classical}]
We first take any test function $\psi \in H^1_{\loc}(\R)\subset C^0(\R)$ such that $\supp(\psi) \subset [0,\infty)$ and $\psi' \in L^2(\R)$. 
Next, taking $\alpha \in (0,\alpha_0]$, $u \in D(\scrP)$, we set $v_\alpha(t) := e^{-\alpha t}e^{it\scrP}u = e^{it(\scrP + i\alpha)}u \in C^0(\R^+; D(\scrP))\cap C^1(\R^+;\scrH)$ defined for $t\geq0$, and $w_\alpha(t) = \psi(t)v_\alpha(t)$, continued by $0$ for $t\leq 0$.
We have $\| w_\alpha(t)\|_{\scrH} \leq |\psi(t)|e^{-\alpha t}\|u\|_{\scrH}$ and $\| \scrP w_\alpha(t)\|_{\scrH} \leq |\psi(t)|e^{-\alpha t}\|\scrP u\|_{\scrH}$ for all $t\in \R$, so that $w_\alpha \in C^0(\R; D(\scrP))\cap C^1(\R;\scrH) \cap L^p(\R ; D(\scrP)) \cap W^{1,p}(\R ; \scrH) $ for all $p \in [1,+\infty]$, together with 
$$
\left(\frac{\d_t}{i} - (\scrP+i\alpha) \right)v_\alpha(t)=0, \quad \text{ for } t\in \R_+^* , 
$$
whence
\begin{equation}
\label{e:eq-trunk-time-A}
\left(\frac{\d_t}{i} - (\scrP+i\alpha) \right)w_\alpha(t)= \frac{1}{i}\psi'(t) v_\alpha (t) \quad \text{ for }t\in \R . 
\end{equation}
We take the Fourier transform in time of~\eqref{e:eq-trunk-time-A} to obtain 
\begin{equation}
\left(\tau - \scrP -i\alpha \right) \hat{w}_\alpha(\tau)= \frac{1}{i} \F(\psi' v_\alpha)(\tau)  . 
\end{equation}
Now, applying~\eqref{e:bdd-res-op-1-bis} to $\hat{w}_\alpha(\tau) \in D(\scrP)$, we deduce
$$
\nor{\scrC \hat{w}_\alpha(\tau) }{\scrH} \leq C_0 \nor{(\scrP+i\alpha-\tau) \hat{w}_\alpha(\tau) }{\scrH} = C_0\nor{\F(\psi' v_\alpha)(\tau) }{\scrH} , \quad \text{ for all } \tau \in \R 
$$
With the Plancherel theorem, and the fact that $(e^{it\scrP})_{t\in\R_+}$ is a contraction semigroup, we have, for all $\alpha \in [0,\alpha_0]$,  
$$
\int_\R \nor{\F(\psi' v_\alpha)(\tau) }{\scrH}^2  d\tau = \int_\R \nor{\psi'(t) v_\alpha(t)}{\scrH}^2  dt = \int_\R |\psi'(t)|^2 e^{-2\alpha t}\nor{e^{it\scrP} u}{\scrH}^2 dt \leq \nor{\psi'}{L^2(\R)}^2 \nor{u}{\scrH}^2 .
$$
The Plancherel theorem also yields 
$$
\int_\R \nor{\scrC \hat{w}_\alpha(\tau) }{\scrH}^2 d\tau =  \int_\R \nor{\scrC w_\alpha(t) }{\scrH}^2 dt .
$$
Together with the previous two inequalities, this implies for all $\alpha \in [0,\alpha_0]$, 
$$
\int_\R  |\psi(t)|^2e^{-2\alpha t}\|  \scrC e^{it\scrP}u \|_{\scrH}^2 dt  =  \int_\R \nor{\scrC w_\alpha(t) }{\scrH}^2 dt  \leq C_0^2 \nor{\psi'}{L^2(\R)}^2 \nor{u}{\scrH}^2 .
$$
Letting $\alpha\to 0^+$, we deduce that $t\mapsto \psi(t)\scrC e^{it\scrP}u \in L^2(\R ; \scrH)$ with 
$$
 \int_\R |\psi(t)|^2 \nor{\scrC e^{it\scrP}u}{\scrH}^2 dt \leq C_0^2 \nor{\psi'}{L^2(\R)}^2 \nor{u}{\scrH}^2 .
$$
Note that this is true for any function $\psi$ such that $\psi' \in L^2(\R)$ and $\supp(\psi)\subset \R_+$. 
Choosing $\psi(t)= \Psi \left(\frac{t}{T}\right)$ for $T>0$ and $\Psi \in H^1_{\comp}(\R)$ with $\supp(\Psi)\subset(0,\infty)$, and noticing that 
$$
\nor{\psi'}{L^2(\R)}^2 = \frac{1}{T^2}\int_\R \left| \Psi' \left(\frac{t}{T}\right) \right|^2 dt = \frac{1}{T} \nor{\Psi'}{L^2(\R)}^2 ,
$$
we obtain~\eqref{e:Psi-T-classical}, which concludes the proof of the lemma for $u \in D(\scrP)$. The result for $u \in \scrH$ follows by density.
\enp

\subsection{Admissible $\mathsf{M}$ and Assumption~\ref{asspt-reg-B} of Lemma~\ref{l:res-est-res}}
\label{s:Mand-asspt-2}
In this section, we describe sufficient conditions on the operators $\scrP_0,\scrB$ and the function $\mathsf{M}$ for Assumption~\ref{asspt-reg-B} in Lemma~\ref{l:res-est-res} to be satisfied.
We first formulate them in an abstract way. To this aim, we introduce a dyadic partition of unity:
There exist two even functions $\phi , \varphi \in C^\infty(\R ; [0,1])$ such that $\supp(\phi)\subset (-1,1)$, $\supp(\varphi)\subset (-2,-1/2)\cup (1/2,2)$ and 
\begin{align*}
\phi(s) + \sum_{j\in \N} \varphi(2^{-j} s) = 1,  \text{ for all } s\in \R . 
\end{align*}
We also consider another even function $\tilde{\varphi} \in C^\infty(\R ; [0,1])$ satisfying $\tilde{\varphi}=1$ on a neighborhood of $\supp(\varphi)$ and $\supp(\tilde{\varphi})\subset (-2,-1/2)\cup (1/2,2)$.
The parameter $\eps$ does not play any role here, and we thus write $\mathsf{M}_\eps = \mathsf{M}(\frac{\cdot}{1-\eps})$. 

\begin{proposition}
\label{p:dyadic}
Let $\delta\in(0,1]$, $\scrP_0 : \D(\scrP_0)\subset \scrH \to \scrH$ be a selfadjoint operator, $\scrB: \D(\scrB)\subset \scrH \to \scrH$ such that $D(\scrP_0)\subset D(\scrB)$ with continuous embedding, and $\mathsf{M}_\eps \in C^0(\R;\R_+^*)$ be an even function  which is nondecreasing on $\R_+$.  Assume that 
\begin{align}
\label{e:sum-M-N}
& \sum_{k\in \N}  \mathsf{M}_\eps (2^{k+1}) \nor{(\id - \tilde{\varphi}(2^{-k} \scrP_0))\scrB \varphi(2^{-k} \scrP_0)}{\L(\scrH)}  < +\infty  ,\\
\label{e:diag-terms}
&  \sum_{k\in \N}  2^{-(1-\delta)k} \mathsf{M}_\eps (2^{k-1})^{-1}\mathsf{M}_\eps (2^{k+1})  \nor{\tilde{\varphi}(2^{-k} \scrP_0) \scrB \varphi(2^{-k} \scrP_0) }{\L(\scrH)} <+\infty  .
\end{align}
Then the operator $\langle\scrP_0\rangle^{-1+\delta}\mathsf{M}_\eps(\scrP_0)^{-1}\scrB\mathsf{M}_\eps(\scrP_0)$ extends from $\left\{\chi(\scrP_0)w,\chi\in C^\infty_c(\R),w\in\scrH\right\}$ as a bounded operator $\scrH \to \scrH$ (and hence Assumption~\ref{asspt-reg-B} in Lemma~\ref{l:res-est-res} is satisfied).
\end{proposition}
Note that the continuity in the embedding $D(\scrP_0)\subset D(\scrB)$ implies that 
\begin{align}
\label{e:bdd-phi}\scrB\chi(\scrP_0)\in\L(\scrH),\quad \text{ for all }t\chi\in C^\infty_c(\R),
\end{align} so that each term in the sums~\eqref{e:sum-M-N}--\eqref{e:diag-terms} makes sense.
The following is a direct corollary of Proposition~\ref{p:dyadic}.
\begin{corollary}
\label{c:dyadic}
Assume $\B$ is bounded and there are $N,C,\lambda_0>0$ and $\eta\in(0,1]$ such that 
\begin{align}
\label{e:sum-M-N-bis}
& \mathsf{M} (\lambda) \leq C \lambda^{N-\eta} , \qquad  \nor{(\id - \tilde{\varphi}(\lambda^{-1}\scrP_0))\scrB \varphi(\lambda^{-1}\scrP_0)}{\L(\scrH)} \leq C\lambda^{-N} , \quad \text{ for all } \lambda \geq \lambda_0 , \\
\label{e:diag-terms-bis}
& \mathsf{M} (4\lambda) \leq C \lambda^{1-\eta} \mathsf{M} (\lambda) ,\quad \text{ for all } \lambda \geq \lambda_0  .
\end{align}
Then, for all $\eps \in [0,1)$ and for all $\delta\in (0,\eta)$, Assumption~\ref{asspt-reg-B} in Lemma~\ref{l:res-est-res} is satisfied. 
\end{corollary}
Note that in~\eqref{e:diag-terms-bis}, the ``gain'' $\lambda^{1-\eta}$ (as compared to $\eta=1$) is neither optimal, nor seemingly useful in applications.
The important remaining assumption is the second part of~\eqref{e:sum-M-N-bis}. 
It quantifies how the operator $\scrB$ ``mixes'' the frequencies of $\scrP_0$ at ``high-frequency''; in application to the damped wave operator, it is henceforth naturally related to the regularity of the function $b$.
We study its validity in that context via semiclassical analysis in Section~\ref{s:app-damped-semiclass} below.

\bnp[Proof of Corollary~\ref{c:dyadic}]
Note first that the boundedness of $\scrB$ directly implies~\eqref{e:bdd-phi}. Next,~\eqref{e:sum-M-N-bis} implies that for $k$ sufficiently large
$\mathsf{M}_\eps (2^{k+1})\leq C\left( \frac{2^{k+1}}{1-\eps}\right)^{N-\eta}$ and 
$$
 \sum_{k\in \N}  \mathsf{M}_\eps (2^{k+1}) \nor{(\id - \tilde{\varphi}(2^{-k} \scrP_0))\scrB \varphi(2^{-k} \scrP_0)}{\L(\scrH)} \leq  \sum_{k\in \N} C_\eps 2^{(N-\eta)k} C2^{-kN} < +\infty , 
$$
hence~\eqref{e:sum-M-N} (for all $\eps \in [0,1)$). Finally,~\eqref{e:diag-terms-bis} implies for $k$ sufficiently large
$\mathsf{M}_\eps (2^{k+1})\leq C\left( \frac{2^{k-1}}{1-\eps}\right)^{1-\eta}\mathsf{M}_\eps (2^{k-1})$. Using boundedness of $\scrB$, this yields
$$
 \sum_{k\in \N}  2^{-(1-\delta)k} \mathsf{M}_\eps (2^{k-1})^{-1}\mathsf{M}_\eps (2^{k+1})  \nor{\tilde{\varphi}(2^{-k} \scrP_0) \scrB \varphi(2^{-k} \scrP_0) }{\L(\scrH)} 
 \leq  \sum_{k\in \N}  2^{-(1-\delta)k}  C_\eps 2^{(1-\eta)k} < +\infty ,
 $$
 as soon as $\delta\in (0,\eta)$. Application of Proposition~\ref{p:dyadic} then concludes the proof.
\enp

\bnp[Proof of Proposition~\ref{p:dyadic}]
We decompose
\begin{align*}
\langle \scrP_0 \rangle^{-1+\delta} \mathsf{M}_\eps (\scrP_0)^{-1} \scrB \mathsf{M}_\eps (\scrP_0) & = T_1 + T_2 , \quad \text{with}\\
T_1 &= \langle \scrP_0 \rangle^{-1+\delta} \mathsf{M}_\eps (\scrP_0)^{-1} \scrB \phi(\scrP_0) \mathsf{M}_\eps (\scrP_0),  \\
T_2 &=  \langle \scrP_0 \rangle^{-1+\delta} \mathsf{M}_\eps (\scrP_0)^{-1} \scrB (\id- \phi(\scrP_0)) \mathsf{M}_\eps (\scrP_0) .
\end{align*}
The first term is bounded by assumption since, using $\supp(\phi) \subset (-1,1)$, the monotonicity of $\mathsf{M}_\eps$, and~\eqref{e:bdd-phi}, we have
$$
\nor{\langle \scrP_0 \rangle^{-1+\delta} \mathsf{M}_\eps (\scrP_0)^{-1} \scrB \phi(\scrP_0) \mathsf{M}_\eps (\scrP_0)}{\L(\scrH)} \leq \nor{\langle \scrP_0 \rangle^{-1+\delta} \mathsf{M}_\eps (\scrP_0)^{-1}}{\L(\scrH)} \nor{\scrB \phi(\scrP_0)}{\L(\scrH)}\mathsf{M}_\eps(1) <+\infty .
$$
The second term is
\begin{align*}
T_2 & =  \sum_{k\in \N}  \langle \scrP_0 \rangle^{-1+\delta} \mathsf{M}_\eps (\scrP_0)^{-1} \scrB \varphi(2^{-k} \scrP_0)\mathsf{M}_\eps (\scrP_0) .
\end{align*}
We write $\scrB \varphi(2^{-k} \scrP_0) =  \tilde{\varphi}(2^{-k} \scrP_0) \scrB \varphi(2^{-k} \scrP_0) + (\id - \tilde{\varphi}(2^{-k} \scrP_0))\scrB \varphi(2^{-k} \scrP_0)$,
and decompose accordingly 
$
T_2 = T_{2,1} + T_{2,2},
$
with 
\begin{align*}
T_{2,1} & =  \sum_{k\in \N}  \langle \scrP_0 \rangle^{-1+\delta} \mathsf{M}_\eps (\scrP_0)^{-1} \tilde{\varphi}(2^{-k} \scrP_0) \scrB \varphi(2^{-k} \scrP_0)\mathsf{M}_\eps (\scrP_0) ,\\
T_{2,2} & =  \sum_{k\in \N}  \langle \scrP_0 \rangle^{-1+\delta} \mathsf{M}_\eps (\scrP_0)^{-1} (\id - \tilde{\varphi}(2^{-k} \scrP_0))\scrB \varphi(2^{-k} \scrP_0)\mathsf{M}_\eps (\scrP_0) .
\end{align*}
Using that $\mathsf{M}_\eps$ is (even and) nondecreasing on $\R_+$ and that $\supp \varphi \subset(-2,2)$, we obtain
\begin{align*}
\nor{T_{2,2}}{\L(\scrH)} & \leq  \sum_{k\in \N} \nor{ \langle \scrP_0 \rangle^{-1+\delta} \mathsf{M}_\eps (\scrP_0)^{-1} }{\L(\scrH)} \nor{(\id - \tilde{\varphi}(2^{-k} \scrP_0))\scrB \varphi(2^{-k} \scrP_0)}{\L(\scrH)} \mathsf{M}_\eps (2^{k+1})   <+ \infty .
\end{align*}
by assumption~\eqref{e:sum-M-N}.

Concerning the term $T_{2,1}$, we remark that the function $\langle \cdot \rangle^{-1+\delta} \mathsf{M}_\eps (\cdot)^{-1}$ is (even and) nonincreasing on $\R_+$. Using further that $\supp \tilde\varphi \subset(-\infty,-1/2) \cup (1/2,+ \infty)$ and $\supp \varphi \subset(-2,2)$, we obtain
\begin{align*}
\nor{T_{2,1}}{\L(\scrH)} &\leq  \sum_{k\in \N} \langle 2^{k-1} \rangle^{-1+\delta} \mathsf{M}_\eps (2^{k-1})^{-1} \nor{\tilde{\varphi}(2^{-k} \scrP_0) \scrB \varphi(2^{-k} \scrP_0) }{\L(\scrH)} \mathsf{M}_\eps (2^{k+1}) \\
& \leq  \sum_{k\in \N}  2^{-(1-\delta)(k-1)} \mathsf{M}_\eps (2^{k-1})^{-1}\mathsf{M}_\eps (2^{k+1})  \nor{\tilde{\varphi}(2^{-k} \scrP_0) \scrB \varphi(2^{-k} \scrP_0) }{\L(\scrH)} <+\infty ,
\end{align*}
according to~\eqref{e:diag-terms}.
\enp

\subsection{Application to damped Klein-Gordon equations}
\label{s:app-damped-semiclass}
In this section, we put the Klein-Gordon equation in the framework of the above sections, and check Condition~\eqref{e:sum-M-N-bis} assuming regularity of $b$ (using semiclassical analysis).
We start with the following lemma.
\begin{lemma}
\label{l:b-mix-freq}
Assume $b\in C^\infty(M)$ and let $m\geq0$, then for all $N\in \N$, there is $C_N>0$ such that 
\begin{equation}
\label{e:b-mix-infty}
\nor{(1 - \tilde{\varphi}(h\Lambda_m)) b  \varphi(h\Lambda_m)}{\L(L^2(M))} \leq C_N h^{N} , \quad \text{ for all } h \in (0,h_0) .
\end{equation}
Assume $b\in C^0(M)$ has continuity modulus $\omega$, then for all $\nu<\frac12$, there are $C,\kappa>0$ such that 
$$
\nor{(1 - \tilde{\varphi}(h\Lambda_m)) b  \varphi(h\Lambda_m)}{\L(L^2(M))} \leq C\omega(\kappa h^\nu) , \quad \text{ for all } h \in (0,h_0) .
$$
\end{lemma}

\bnp
The first property directly follows from the fact that $\supp(1-\tilde{\varphi}) \cap \supp \varphi = \emptyset$, together with properties of functional calculus in Corollary~\ref{c:functional-calc}.

Concerning the second property, we introduce the regularized function $b_{h^\nu}$ defined in Appendix~\ref{e:reg-b} and write 
$$
\nor{(1 - \tilde{\varphi}(h\Lambda_m)) b  \varphi(h\Lambda_m)}{\L(L^2(M))} \leq \nor{(1 - \tilde{\varphi}(h\Lambda_m)) b_{h^\nu}  \varphi(h\Lambda_m)}{\L(L^2(M))} 
+ \nor{b -b_{h^\nu}}{L^\infty(M)} .
$$
According the properties of functional calculus in Corollary~\ref{c:functional-calc}, we have for all $N \in \N$,
$$
\nor{(1 - \tilde{\varphi}(h\Lambda_m)) b_{h^\nu}  \varphi(h\Lambda_m)}{\L(L^2(M))} \leq C_N h^{N} .
$$
Now using Corollary~\ref{cor:def-b-eps}, we have $ \nor{b -b_{h^\nu}}{L^\infty(M)} \leq C\omega(\kappa h^\nu)$, which, combined with the previous two inequalities concludes the proof of the second statement. 
\enp

\begin{theorem}
\label{t:thm-classiq-A}
Let $m>0$. Assume $b\in C^0(M;\R^+)$ is positive on a non-empty open set, $\mathsf{M}  \in C^0(\R;\R_+^*)$ is an even function  which is nondecreasing on $\R^+$. Assume that there exist $C,\lambda_0>0$ (possibly large) and $\eta \in (0,1]$ (small) such that
\begin{enumerate}
\item $\nor{(i\tau \id- \A_m)^{-1}}{\L(H^1_m\times L^2)} \leq \mathsf{M}(\tau)$ for all $\tau \in \R$;
\item \label{i:growth-M-temperate}  $\mathsf{M} (4\lambda) \leq C \lambda^{1-\eta} \mathsf{M} (\lambda)$ for all $\lambda \geq \lambda_0$;
\item \label{i:reg-of-b}
\begin{enumerate}
\item either $b\in C^\infty(M)$ and there is $N>0$ such that $\mathsf{M} (\lambda) \leq C \lambda^{N}$ for $\lambda\geq \lambda_0$;
\item or $b$ is $\alpha$-H\"older continuous and $\mathsf{M} (\lambda) \leq C \lambda^{\frac{\alpha}{2}-\eta}$ for $\lambda\geq \lambda_0$;
\end{enumerate}
\end{enumerate}
Then, for all $\eps \in (0,1)$, there is $C_0>0$ such that for all $\Psi \in H^1_{\loc}(\R)$ with $\supp(\Psi) \subset [0,+\infty)$ and $\Psi' \in L^2(\R)$, for all $T>0$ and all $U \in L^2(M;\C^2)$, $V\in H^1_m \times L^2$, we have
\begin{align}
\label{e:Psi-T-classical-bis}
\frac{1}{T} \int_\R \Psi \left(\frac{t}{T}\right)^2 \nor{ \mathsf{M}\left(\frac{\Lambda_m}{1-\eps}\right)^{-1} e^{it\P_m}U}{L^2(M;\C^2)}^2 dt \leq \frac{C_0^2}{T^2}\nor{\Psi'}{L^2(\R)}^2 \nor{U}{L^2(M;\C^2)}^2 ,\\
\label{e:Psi-T-classical-ter}
\frac{1}{T} \int_\R \Psi \left(\frac{t}{T}\right)^2 \nor{ \mathsf{M}\left(\frac{\Lambda_m}{1-\eps}\right)^{-1} e^{t\A_m}V}{H^1_m \times L^2}^2 dt \leq \frac{C_0^2}{T^2}\nor{\Psi'}{L^2(\R)}^2 \nor{V}{H^1_m \times L^2}^2 .
\end{align}
\end{theorem}
Note first that the assumption that $b$ is nontrivial implies $\R\subset \rho(\P_m)$. The theorem applies in particular to the function $\mathsf{M}=\mathsf{M}_{\P_m}$.

Now, if $(u,\d_t u)(t) = e^{t\A_m}(u_0,u_1)$ denotes the solution to~\eqref{eq: stabilization} for $m>0$, then~\eqref{e:Psi-T-classical-ter} rewrites as~\eqref{e:Psi-T-classical-quad-0} (see Item~\ref{i:Energ-am} in Corollary~\ref{corollary-Pm-Am-reso}) 
Therefore, Theorem~\ref{t:thm-classiq-A} in the case $b\in C^\infty(M)$ implies Theorem~\ref{t:thm-classiq-A-0} in the introduction.

We do not formulate the analoguous result in the case of the wave equation $m=0$ for the sake of concision. The proof would be a combination of those of Theorems~\ref{t:thm-classiq-A} and~\ref{t:semiclassic-++}.

\bnp[Proof of Theorem~\ref{t:thm-classiq-A}]
We first prove~\eqref{e:Psi-T-classical-bis}. 
For this, we apply the results of the above sections to $\scrH =L^2(M;\C^2)$, $\scrP = \P_m$ defined in~\eqref{e:def-Pm-KG}, that is $\scrP_0 = \Lambda_m \begin{pmatrix}
1 & 0\\
0& -1
\end{pmatrix}$ and $\scrB = \frac{b}{2}\begin{pmatrix}
1 & 1\\
1& 1
\end{pmatrix}$.
According to Corollary~\ref{corollary-Pm-Am-reso} and the assumption, we have $\nor{(\tau \id- \P_m)^{-1}}{\L( L^2(M;\C^2))} = \nor{(i\tau \id- \A_m)^{-1}}{\L(H^1_m\times L^2)} \leq \mathsf{M}(\tau)$ for all $\tau \in \R$.
Next, since the functions $\varphi,\tilde{\varphi}$ and $\mathsf{M}$ are even, we have
$$
\varphi(\lambda^{-1}\scrP_0) = \varphi(\lambda^{-1}\Lambda_m) I_2 , \quad \tilde{\varphi}(\lambda^{-1}\scrP_0) =  \tilde{\varphi}(\lambda^{-1}\Lambda_m) I_2 ,
\quad \mathsf{M}_\eps(\scrP_0)^{-1} = \mathsf{M}_\eps(\Lambda_m)^{-1} I_2 .
$$
As a consequence, we have
$$
(\id - \tilde{\varphi}(\lambda^{-1}\scrP_0))\scrB \varphi(\lambda^{-1}\scrP_0) =  (1 - \tilde{\varphi}(\lambda^{-1}\Lambda_m)) b  \varphi(\lambda^{-1}\Lambda_m) \frac{1}{2}\begin{pmatrix}
1 & 1\\
1& 1
\end{pmatrix} .
$$
Condition~\eqref{e:sum-M-N-bis} is thus a consequence of Lemma~\ref{l:b-mix-freq} together with Assumption~\ref{i:reg-of-b}. Condition~\eqref{e:diag-terms-bis} is Item~\ref{i:growth-M-temperate}. Corollary~\ref{c:dyadic} thus applies. The assumptions of Lemma~\ref{l:res-est-res} are hence satisfied, and therefore those of Lemma~\ref{l:Psi-T-classical} with $\scrC = \mathsf{M}_\eps \left(\P_m\right)^{-1} = \mathsf{M}_\eps(\Lambda_m)^{-1} I_2$, which  proves~\eqref{e:Psi-T-classical-bis}.
Finally,~\eqref{e:Psi-T-classical-ter} is a direct consequence of Corollary~\ref{corollary-Pm-Am-reso} Item~\ref{i:fPm}.
\enp

\appendix

\section{Geometry and pseudodifferential calculus}
\label{appendix}
\subsection{Local charts, pullbacks, and partition of unity}
\label{sec: local charts}
For a diffeomorphism $\phi$ between two open sets, $\phi: U_1 \to U_2$,
the associated pullback (here stated for continuous functions) is
\begin{align*}
  \phi^\ast : C^0 (U_2) &\to C^0 (U_1),\\
    u &\mapsto  u \circ \phi.
\end{align*}
For a function defined on phase-space, \eg a symbol, the pullback is
given by
\begin{align}
  \label{eq: pullback phasespace}
  \phi^\ast u (y,\eta) = u (\phi(y),\transp \big(\phi'(y)\big)^{-1} \eta), 
  \quad y \in U_1, \eta \in T^\ast_yU_1, \quad u \in C^0(T^\ast U_2).
\end{align}
The compact manifold $M$ is of dimension $n$ and is furnished with a finite
atlas $(U_j, \phi_j)$, $j \in J$, with $\phi_j: U_j \to \Ut_j =\phi_j (U_j)\subset \R^{n}$ a smooth diffeomorphism.  
For later use, we denote by $(\psi_j)_j$ a partition of unity subordinated to this covering:
 $$
 \psi_j \in C^\infty(M),\quad  \supp(\psi_j) \subset U_j,  \quad 0\leq \psi_j \leq 1, \quad \sum_j \psi_j =1.
 $$ 

\subsection{Symbol classes}
\label{app:symbol-classes}
We introduce the following mildly exotic but now classical class of symbols (see~\cite[Chapter~7]{DS:book} or~\cite{Zworski:book}). We say that $a \in S^m_\rho(T^*\R^n)$ if it is an $h$-dependent family of $C^\infty(\R^n \times \R^n)$ functions and, for all $\alpha, \beta \in \N^n$, there is $C_{\alpha,\beta}>0$ such that for all $h \in (0,h_0)$,
\bnan
\label{e:symb-estim-Srho}
|\d_x^\alpha \d_\xi^\beta a(x,\xi , h)| \leq C_{\alpha,\beta} h^{-\rho(|\alpha| + |\beta|)} \langle \xi \rangle^{m-|\beta|} .
\enan
The best constants $C_{\alpha,\beta}$ constitute a family of seminorms, endowing $S^m_\rho(T^*\R^n)$ with the topology of a Fr\'echet space. 
Given a bounded open set $\tilde{U} \subset \R^n$, we say that $a \in S^m_\rho(T^*\tilde{U})$ if $\chi a \in S^m_\rho(T^*\R^n)$ for all $\chi \in  C^\infty_c(\tilde{U})$. 
Finally, given $a \in C^\infty(T^*M)$, we say that $a \in S^m_\rho(T^*M)$ if $(\phi^{-1})^*a \in S^m_\rho(T^*\phi(U))$ for any coordinate patch $(U,\phi)$ of $M$, a property which has only to be checked for an atlas~\cite[Theorem~9.4 and Section~14.2.3]{Zworski:book}.

\subsection{Operators and quantization}
\label{app:op-quant}
We recall the usual left quantization of symbols in $\R^n$: for $a \in S^m_\rho(T^*\R^n)$, 
$$
\left(\Op_h^\ell (a)  u \right) (x)= \frac{1}{(2\pi h)^n}\int_{\R^{2n}} e^{\frac{i}{h}(x-y)\cdot \xi} a\left( x , \xi\right) u(y) dy d\xi .
$$
We denote by $\Psi^m_\rho(\R^n)$ the set of operators $\Op_h^\ell (a)$ with $a \in S^m_\rho(T^*\R^n)$.

\begin{definition}
\label{def:pseudo}
We say that an operator $A : C^\infty(M) \to \D'(M)$ belongs to $\Psi^m_\rho(M)$ if the following two items are satisfied:
\begin{enumerate}
\item \label{i:pseudolocal} for any $\chi_1,\chi_2 \in C^\infty_c(M)$ with $\supp(\chi_1) \cap \supp(\chi_2) = \emptyset$, we have for all $N\in \N$, $\nor{ \chi_1 A \chi_2}{\L(H^{-N},H^N)}  = O(h^\infty)$;
\item \label{i:pseudo} for any coordinate patch $(U,\phi)$ of $M$ and for all $\chi_1,\chi_2 \in C^\infty_c(U)$, we have $\big(\phi^{-1}\big)^\ast \chi_1 A \chi_2 \phi^\ast \in \Psi^m_\rho(\R^n)$.
\end{enumerate}
\end{definition}
We refer \eg to~\cite[Definition~E.12 and Proposition~E.13]{DZ:book}.
Now, we quantize a symbol $a \in S^m_\rho(T^*M)$ as 
 \begin{align}
\label{eq: def psiDO on M}
\Op_h(a) = \sum_{j\in J} \check{\psi}_j \phi_j^\ast  \Op_h^\ell(  \tilde{a}_j ) 
  \big(\phi_j^{-1}\big)^\ast \check{\psi}_j , \quad \text{with} \quad  \tilde{a}_j = \big(\phi_j^{-1}\big)^\ast (a \psi_j) , \quad j \in J, 
\end{align}
where $\check{\psi}_j \in C^\infty_c(U_j)$ is such that $\check{\psi}_j =1$ in a neighborhood of $\supp(\psi_j)$. According to~\cite[Theorem~14.1]{Zworski:book}, if $a \in S^m_\rho(T^*M)$ then $\Op_h( a) \in \Psi^m_\rho(M)$.

\subsection{Pseudodifferential calculus}
\label{s:calcul-pseudo}
The general idea is that the gain in the $\Psi^m_\rho(M)$ pseudodifferential calculus is $h^{1-2\rho}$ (together with a derivative). From now on, we assume that $0\leq\rho<\frac12$.
We recall that the principal symbol map (see~\cite[Propositions~E.14-E.16]{DZ:book}) 
$$
\sigma_h : \Psi^m_\rho(M) \to S^m_\rho(T^*M)/h^{1-2\rho}S^{m-1}_\rho(T^*M)  , 
\quad A \mapsto \sigma_h(A) ,
$$
satisfies the following first properties:
\begin{enumerate}
\item 
The kernel of $\sigma_h$ is $h^{1-2\rho}\Psi^{m-1}_\rho(M)$, \ie $\sigma_h(A) = 0$ if and only if $A \in h^{1-2\rho}\Psi^{m-1}_\rho(M)$.
\item 
For all $a\in S^m_\rho(T^*M)$, we have $\sigma_h(\Op_h(a)) = a$ in $S^m_\rho(T^*M)/h^{1-2\rho}S^{m-1}_\rho(T^*M)$ (in particular, $\sigma_h$ is onto).
\item 
For all $A\in \Psi^m_\rho(M)$, we have $A - \Op_h(\sigma_h(A) ) \in  h^{1-2\rho} \Psi^{m-1}_\rho(M)$.
\item \label{i:act-coord-patch}
For $A\in \Psi^m_\rho(M)$ and for any coordinate patch $(U,\phi)$ of $M$ and $\chi \in C^\infty_c(U)$, we have 
 $ \big(\phi^{-1}\big)^\ast \chi A\chi \phi^\ast - \Op_h\left(  \big(\phi^{-1}\big)^\ast (\chi^2 \sigma_h(A))\right) \in h^{1-2\rho} \Psi^{m-1}_\rho(\R^n)$.
\end{enumerate}
Operators in $ \Psi^0_\rho(M)$ are $L^2$-bounded, with a finite number of derivatives (see \eg~\cite[Theorem~2.5.1]{Lerner:10} or \cite[Corollaire~2]{FK:14}): There is $\kappa=\kappa(n)$ (depending only on the dimension of $M$) such that for all $a\in S^{0}_\rho(T^*M)$, there is $C>0$ such that $$\|\Op_h(a)\|_{\L(L^2)} \leq C\sup_{|\alpha|\leq \kappa} \|h^{|\alpha|}\d^\alpha a\|_{\infty}.$$
 Namely, for $A \in \Psi^m_\rho(M) , B \in \Psi^k_\rho(M)$ we have:
\begin{enumerate}
\item Adjoint (taken in $L^2(M,d\Vol_g)$): $A^* \in  \Psi^m_\rho(M)$ with $\sigma_h(A^*)= \ovl{\sigma_h(A)}$,
\item Product:  $AB \in \Psi^{m+k}_\rho(M)$ with $\sigma_h(AB) =\sigma_h(A)\sigma_h(B)$, 
\item Commutator:  $[A,B] \in h^{1-2\rho} \Psi^{m+k-1}_\rho(M)$ with $\sigma_h\big(h^{2\rho-1}[A,B] \big) = \frac{h^{2\rho}}{i} \{\sigma_h(A), \sigma_h(B)\}$ where equality holds in $$S^{m+k-1}_\rho(T^*M)/h^{1-2\rho}S^{m+k-2}_\rho(T^*M).$$
\end{enumerate}
We refer to \cite[End of Section~E.1.7]{DZ:book}. This last point can be rewritten with the quantization $\Op_h$ as follows: for all $a \in S^{m}_\rho(T^*M), b\in S^{k}_\rho(T^*M)$, we have
$$
\frac{1}{h}[\Op_h(a),\Op_h(b)] =  \Op_h\big(\frac{1}{i}\{a ,b\} \big) + R , \qquad R \in h^{1-4\rho} \Psi^{m+k-2}_\rho(M) ,
$$
One should beware that $\frac{1}{h}[\Op_h(a),\Op_h(b)]$ and $\Op_h\big(\frac{1}{i}\{a ,b\} \big)$ both belong to $h^{-2\rho} \Psi^{m+k-1}_\rho(M)$.
 As noticed in~\cite{DJN:20}, this $O(h^{1-4\rho})$ remainder is not good enough for the proof of an Egorov theorem in time $\left(\frac{1}{2\Upsilon_{\max}}+\eps \right)\log(h^{-1})$ (here Theorem~\ref{t:egorov}).
We follow the remedy in~\cite[Appendix~A]{DJN:20}. The following is part of~\cite[Lemma~A.6]{DJN:20}.
\begin{lemma}
\label{l:DJN-A6}
There exists $\kappa (n) >0$ a constant depending only of the dimension $n$ of $M$ such that the following holds true. For all $a,b \in C^\infty_c(T^*M)$ such that $\supp a, \supp b \subset \{|\xi|_x \leq 10\}$ and all $N\in \N$, we have 
\begin{align}
\label{e:DJN-A6}
[\Op_h(a) , \Op_h(b)] & = \Op_h \left( \frac{h}{i} \{a,b\}+\sum_{j=2}^{N-1} h^j{\bf D}^{2j-4}(d^2a\otimes d^2b)|_{\Diag} \right)+ O_{N} \big(\nor{a \otimes b}{C^{2N+ \kappa(n)}} h^N\big) ,
\end{align}
where 
\begin{itemize}
\item $a\otimes b\in C^\infty_c(T^*M\times T^*M)$ is defined by $(a\otimes b)(\zeta,\zeta') = a(\zeta)b(\zeta')$, 
\item $\Diag$ is the diagonal of $T^*M \times T^*M$, 
\item $d^2b$ denotes the vector $(\d^\alpha b)_{|\alpha|\leq 2}$, 
\item ${\bf D}^{k}c$ denotes the result of applying some partial differential operator of order $k$ to $c$, with coefficients depending only on $(M,g)$, the atlas $(U_j, \phi_j)$ and the cutoff functions $(\psi_j, \check\psi_j)$ used to define the quantization $\Op_h$,
\item $O_N (\cdot)$ is in $\L(L^2(M))$ (and also depends on the same features as ${\bf D}^{k}$).
\end{itemize}
\end{lemma}
Note that Lemma~\ref{l:DJN-A6} is stated only in dimension $n=2$ in~\cite[Lemma~A.6]{DJN:20} (in which case $\kappa(2)=15$ is estimated explicitly); however the proof works as well in higher dimension using that  pseudodifferential calculus (and in particular boundedness in $L^2$, see \eg \cite[Section~2.2]{FK:14}) consumes a finite number (depending only on the dimension) of derivatives.
The point of that lemma is the very particular structure of the lower order terms, which allows to show the following corollary (see the remarks after Lemma~A.6 in~\cite{DJN:20}). 
\begin{corollary}
\label{c:DJN-A6}
Let $\rho \in [0,1/2)$ and assume that $a \in S^{-\infty}_\rho(T^*M)$ and $b \in S^{-\infty}(T^*M)=S^{-\infty}_0(T^*M)$ satisfy $\supp(a), \supp(b)\subset \{|\xi|_x\leq 10\}$. Then we have 
\begin{enumerate}
\item \label{e:poisson-loss} $ \{a,b\} \in h^{-\rho} S^{-\infty}_\rho(T^*M)$,
\item  \label{e:commut-loss}  $[\Op_h(a) , \Op_h(b)]  \in h^{1-\rho} \Psi^{-\infty}_\rho(M) + O_{\L(L^2(M))} \big(h^{\infty} \big)$, 
\item  \label{e:poiss-comm} $[\Op_h(a) , \Op_h(b)]  = \Op_h \left( \frac{h}{i} \{a,b\} \right)+ O_{\L(L^2(M))} \big(h^{2-2\rho} \big)$.
\end{enumerate}
\end{corollary}
Note that the rough $\Psi^{-\infty}_\rho(M)$ calculus only yields the same result with $h^{-\rho}$ replaced by $h^{-2\rho}$. The improvement is due to the fact that one of the symbols is in $S^{-\infty}_0(T^*M)$, together with Lemma~\ref{l:DJN-A6}. 
\bnp[Proof of Corollary~\ref{c:DJN-A6} from Lemma~\ref{l:DJN-A6}]
Item~\ref{e:poisson-loss} directly follows from the symbolic estimates~\eqref{e:symb-estim-Srho} satisfied by $a,b$.
To prove Items~\ref{e:commut-loss} and~\ref{e:poiss-comm}, we examine the remainder term in~\eqref{e:DJN-A6}. From~\eqref{e:symb-estim-Srho} we deduce $\nor{a \otimes b}{C^{2N+ \kappa(n)}} \leq C_N h^{-\rho(2N+ \kappa(n))}$. As a consequence, 
\begin{align}
\label{e:remainder-DJN-1}
O_{N} \big(\nor{a \otimes b}{C^{2N+ \kappa(n)}} h^N\big)  & = O_{N} (h^{N(1-2\rho)-\rho \kappa(n)})  \\
\label{e:remainder-DJN-2}
& = O_{\L(L^2(M))}(h^{2-2\rho}),
\end{align}
where we have fixed $N \in \N, N> \frac{2 + (\kappa(n)-2)\rho}{1-2\rho}$ for the last equality to hold.
Next, we have
\begin{align*}
{\bf D}^{2j-4}(d^2a\otimes d^2b)|_{\Diag} \in h^{-(2j-2)\rho} S^{-\infty}_\rho(T^*M),
\end{align*}
so that 
\begin{align}
\label{e:remainder-DJNb-1}
h^j{\bf D}^{2j-4}(d^2a\otimes d^2b)|_{\Diag}&  \in h^{j(1-2\rho)+2\rho}  S^{-\infty}_\rho(T^*M) ,\\
\label{e:remainder-DJNb-2}
&  \subset 
 h^{2-2\rho} S^{-\infty}_\rho(T^*M) ,  \quad \text{ for all }j \in \{2,\cdots N\} .
\end{align}
Lemma~\ref{l:DJN-A6} together with~\eqref{e:remainder-DJN-1},~\eqref{e:remainder-DJNb-1} and Item~\ref{e:poisson-loss} imply Item~\ref{e:commut-loss}.
Finally,~\eqref{e:remainder-DJNb-2} and $L^2$ boundedness yield
$$
\Op_h \left( h^j{\bf D}^{2j-4}(d^2a\otimes d^2b)|_{\Diag} \right) = O_{\L(L^2(M))}(h^{2-2\rho}).
$$
Lemma~\ref{l:DJN-A6} together with~\eqref{e:remainder-DJN-2} and \eqref{e:remainder-DJNb-2} conclude the proof of Item~\ref{e:poiss-comm}. 
\enp

Finally, we also make use of the following sharp G{\aa}rding inequatity in $\Psi^0_\rho(M;\C^{N\times N})$ (\ie, for matrix-valued operators, see \eg~\cite[Theorem~2.5.4]{Lerner:10}).
\begin{theorem}
\label{t:garding}
Assume 
$$
A \in \Psi^0_\rho(M;\C^{N\times N}),  \quad \sigma_h(A) + \sigma_h(A)^* \geq 0 \text{ on }T^*M , \text{ in the sense of symmetric matrices} ,
$$ then there is $C>0$ such that 
$$ \Re \left(A U,U\right)_{L^2(M;\C^N)} \geq -h^{1-2\rho} C  \nor{U}{L^2(M;\C^N)}^2 \text{ for all }U \in L^2(M;\C^N), h \in (0,h_0) .
$$
\end{theorem}
\subsection{Regularization procedure}
\label{e:reg-b}
We let $N\in C^\infty_c(\R^n)$ be such that $\supp(N)\subset B(0,1)$, $N\geq 0$,  $\int_{\R^n} N(x) dx=1$ and define $N_\eps(x) =\frac{1}{\eps^n}N\left(\frac{x}{\eps}\right)$. 
Notice that, given $b \in C^0_c(\R^n)$, we have $N_\eps \ast b \in C^\infty_c(\R^n)$ with $\d^\alpha(N_\eps \ast b )\leq C_\alpha \eps^{-|\alpha|}$. Hence, for $\eps= h^{\nu}$, we also have $N_\eps \ast b \in S^0_\nu(T^*\R^n)$, and thus the multiplication by the function $N_\eps \ast b$ defines an operator in $\Psi^0_\nu(\R^n)$. We now define a similar regularization procedure for functions $b\in C^0(M)$.
Recall that such a continuous function $b\in C^0(M; \R)$ is uniformly continuous. It hence admits a modulus of continuity, \ie a continuous function $\omega : \R_+ \to \R_+$ with $\omega(0)=0$ such that
\begin{equation}
\label{e:modulus}
| b(x)- b(y) |\leq \omega (\dist_g (x,y)), \quad \text{ for all } x,y\in M ,
\end{equation}
and which we moreover assume to be non-decreasing. Here, $\dist_g$ denotes the Riemannian distance associated to the metric $g$.

Given such a function $b\in C^0(M)$, we define
\begin{equation}
\label{e:def-regu-b}
b_\eps = \sum_{j\in J} \psi_j \phi_j^\ast \left(N_\eps \ast b_j \right), \quad   b_j := \big(\phi_j^{-1}\big)^\ast \left(\check{\psi}_j b \Big(1-\chi \big(\frac{b}{\omega(\eps)}\big) \Big) \right)  ,
\end{equation}
where $\check{\psi}_j \in C^\infty_c(U_j)$ is such that $\check{\psi}_j =1$ in a neighborhood of $\supp(\psi_j)$, and $\chi \in C^\infty_c(\R ;[0,1])$ satisfies $\chi=1$ on $[-1,1]$ and $\supp(\chi) \subset [-2,2]$.

Note that there is an additional cutoff $1-\chi \big(\frac{b}{\omega(\eps)}\big)$ with respect to a usual mollifier. It is aimed at preserving the zero set of $b$.

\begin{proposition}
\label{p:def-b-eps}
There are $C_0, \kappa, C_\alpha>0$ (depending only on $(M,g), \chi$ and the $(U_j,\phi_j), \psi_j , \check{\psi}_j$'s) such that 
for all $b \in C^0(M;\R)$ satisfying $b \geq 0$ and~\eqref{e:modulus} with $\omega$ non-decreasing, the function $b_\eps$ defined by~\eqref{e:def-regu-b} satisfies for all $\eps \in (0,1)$ the following statements:
\begin{enumerate}
\item\label{i:b-pos}$b_\eps$ is real-valued with $b_\eps\geq 0$ on $M$;
\item\label{i:b=0} $b_\eps(x) = 0$ for all $x \in M$ such that $\dist_g(x,\{b=0\} ) \leq \eps$; 
\item \label{i:Linfty-bound} $\|b_\eps\|_{L^\infty(M)} \leq C_0 \|b\|_{L^\infty(M)}$;
\item \label{i:symbolic} $b_\eps \in C^\infty(M)$ and for any coordinate patch $(U,\phi)$ of $M$ and $\check\chi \in  C^\infty_c(\phi(U))$, we have $| \d_x^\alpha \left( \check\chi (\phi^{-1})^* b_\eps \right) | \leq C_\alpha \eps^{-|\alpha|}\|b\|_{L^\infty(M)}$;
\item \label{i:convergence} $\|b_\eps- b\|_{L^\infty(M)} \leq C_0 \omega(\kappa \eps)$.
\end{enumerate}
\end{proposition}
We only use Proposition~\ref{p:def-b-eps} under the following direct corollary (apply Proposition~\ref{p:def-b-eps} to $\eps= h^\nu$).
\begin{corollary}
\label{cor:def-b-eps}
Given $b \in C^0(M;\R)$ satisfying $b \geq 0$ and~\eqref{e:modulus}, and $\nu \in (0,1/2)$, there are $C_0, \kappa>0$ and, for $h \in (0,1)$ a function $b_{h^\nu} \in C^\infty(M; \R) \cap S^0_\nu(T^*M)$ such that $b_{h^\nu} \geq 0$ on $M$, $b_{h^\nu} = 0$ on $\{b=0\}$, $\|b_{h^\nu}\|_{L^\infty(M)} \leq C_0 \|b\|_{L^\infty(M)}$ and $\|b_{h^\nu}- b\|_{L^\infty(M)} \leq C_0 \omega(\kappa h^\nu)$.
\end{corollary}

\begin{proof}[Proof of Proposition~\ref{p:def-b-eps}]
We remark first that $b_j \in C^0_c(\R^n)$ and $b_j \geq 0$, where $b_j$ is defined in~\eqref{e:def-regu-b}.
Item~\ref{i:b-pos} follows from the definition~\eqref{e:def-regu-b} of $b_\eps$ as a sum of convolutions of nonnegative functions. 
To prove Item~\ref{i:b=0}, remark first that~\eqref{e:modulus} implies that $b(x) \leq \omega (\dist_g(x, \{b=0\}))$. Since $\omega$ is non-decreasing, if $x$ is such that we have $\dist_g(x,\{b=0\}) \leq \eps$, this implies
$$
b(x) \leq \omega\big(\dist_g(x,\{b=0\}) \big) \leq \omega(\eps) , \quad \text{ hence } 0\leq  \frac{b(x)}{\omega(\eps) } \leq 1 .
$$
As a consequence, $\chi\big(\frac{b(x)}{\omega(\eps) } \big) = 1$ and~\eqref{e:def-regu-b} implies $b_\eps(x)=0$.

Item~\ref{i:Linfty-bound} is a direct consequence of the definition~\eqref{e:def-regu-b} and the estimate
$$
\|b_\eps\|_{L^\infty(M)} \leq 
C \sum_{j\in J} \nor{N_\eps \ast   b_j }{L^\infty(\R^n)} = C \sum_{j\in J} \nor{b_j}{L^\infty(\R^n)} \leq 
C \|b\|_{L^\infty(M)} .
$$
To prove Item~\ref{i:symbolic}, first remark that $b_\eps \in C^\infty(M)$ since each term in the sum~\eqref{e:def-regu-b} is a convolution of the $C^0$ function $b_j$ with the $C^\infty$ function $N_\eps$. 
Now, given a coordinate patch $(U,\phi)$ of $M$ and $\chi \in  C^\infty_c(\phi(U))$, the only dependence of $\chi (\phi^{-1})^* b_\eps$ on the parameter $\eps$ is in $N_\eps*b_j$. Therefore, in $\d^\alpha \left( \chi (\phi^{-1})^* b_\eps \right)$, all derivatives are uniformly bounded in $\eps$ except those appearing in $\d^\beta \left( N_\eps \ast  b_j \right)$, $|\beta| \leq |\alpha|$, which satisfy
\begin{align*}
\left|\d^\beta \left( N_\eps \ast  b_j \right) (x)\right| & =  \int_{\R^n} \eps^{-|\beta|} \left| (\d^\beta N )\left(\frac{x-y}{\eps} \right)\right| | b_j(y)|\frac{dy}{\eps^n} \\
&   \leq \eps^{-|\beta|}\nor{b_j}{L^\infty(\R^n)} \nor{\d^\beta N}{L^1(\R^n)} \leq C_{\beta}\eps^{-|\beta|}\nor{b}{L^\infty(M)}.
\end{align*}
Summing up these contributions proves Item~\ref{i:symbolic}.

To prove Item~\ref{i:convergence}, first write $b$ as $b = \sum_{j\in J} \psi_j \phi_j^\ast \big(\phi_j^{-1}\big)^\ast \big(\check{\psi}_j b \big)$, so that we have 
\begin{align}
\label{e:mod-infty-b}
\|b_\eps- b\|_{L^\infty(M)} \leq C \max_{j\in J} \nor{\big(\phi_j^{-1}\big)^\ast \big(\check{\psi}_j b \big) - N_\eps \ast b_j}{L^\infty(\R^n)} .
\end{align}
We now write $\tilde{b}_j := \big(\phi_j^{-1}\big)^\ast \big(\check{\psi}_j b \big)$, and, recalling the definition of $b_j$ in~\ref{e:def-regu-b}, decompose the above norm as
\begin{align}
\label{e:interm-convol-reg}
\nor{\tilde{b}_j - N_\eps \ast b_j}{L^\infty(\R^n)} & = \nor{\tilde{b}_j - N_\eps \ast \tilde{b}_j + N_\eps \ast \big(\phi_j^{-1}\big)^\ast \left(\check{\psi}_j b \chi \big(\frac{b}{\omega(\eps)}\big) \right) }{L^\infty(\R^n)} \nonumber \\
& \leq \nor{\tilde{b}_j - N_\eps \ast \tilde{b}_j}{L^\infty(\R^n)}  + \nor{N_\eps \ast \big(\phi_j^{-1}\big)^\ast \left(\check{\psi}_j b \chi \big(\frac{b}{\omega(\eps)}\big) \right) }{L^\infty(\R^n)} .
\end{align}
As far as the second term is concerned, we have 
$$
\nor{N_\eps \ast \big(\phi_j^{-1}\big)^\ast \left(\check{\psi}_j b \chi \big(\frac{b}{\omega(\eps)}\big) \right) }{L^\infty(\R^n)} 
 \leq \nor{N_\eps \ast \big(\phi_j^{-1}\big)^\ast \left(\check{\psi}_j b \chi \big(\frac{b}{\omega(\eps)}\big) \right) }{L^\infty(\R^n)}
 \leq \nor{ b \chi \big(\frac{b}{\omega(\eps)}\big) }{L^\infty(M)} \leq 2\omega(\eps) ,
$$
since $\supp(\chi) \subset (-\infty,2]$.
Let us now consider the first term in~\eqref{e:interm-convol-reg}. 
Using that $\phi_j \in C^1_c(M)$ and $ \check{\psi}_j$ is of class $C^1$, equivalence between Riemannian and Euclidean distances, together with~\eqref{e:modulus} (and the fact that $\omega$ is non-decreasing with $\omega(t) \geq c_0 t$ for $t \in (0,1)$) yield the existence of $C_j, \kappa_j$ such that
$$
| \tilde{b}_j(x)- \tilde{b}_j(y) |\leq C_j \omega (\kappa_j |x -y| ), \quad \text{ for } x,y\in \R^n .
$$
This implies, uniformly for $x \in \R^n$ 
\begin{align*}
 \left| \tilde{b}_j (x)-  N_\eps \ast \tilde{b}_j (x)  \right| & = \left| \tilde{b}_j(x) \int_{\R^n} N_\eps (y) dy- \int_{\R^n}  \tilde{b}_j(x-y) N_\eps (y) dy\right| \\
&  \leq  \int_{\R^n} | \tilde{b}_j(x)- \tilde{b}_j(x-y)| N_\eps (y) dy \leq C_j  \int_{\R^n}  \omega (\kappa_j |y| )N\left(\frac{y}{\eps}\right)\frac{dy}{\eps^n} \\
& \leq C_j  \int_{\R^n}  \omega (\kappa_j \eps |y| )N\left(y\right) dy \leq C_j \omega (\kappa_j \eps) .
\end{align*}
where we used in the last inequality that $\omega$ is non-decreasing and that $|y| \leq 1$ on $\supp(N)$. Plugging this estimate in~\eqref{e:mod-infty-b} finally concludes the proof of Item~\ref{i:convergence} using the finiteness of $J$ and, again, the fact that $\omega$ is non-decreasing. This concludes the proof of the lemma.
\end{proof}

\subsection{Functional calculus}
In the main part of the article, we use functional calculus for the selfadjoint operator $h\Lambda = \sqrt{-h^2\Delta}$. We recall in this section the interplay between functional calculus and pseudodifferential calculus. Classical results (that is, concerning $f(-h^2\Delta)$ for fixed $f$) can be found in~\cite[Section~8]{DS:book} or~\cite[Theorem~14.9]{Zworski:book}. Here, we sometimes use functional calculus for $f_h(-h^2\Delta)$ where $f_h$ satisfies $|\d^\alpha f_h(s)|\leq C_{\alpha} h^{-\alpha \rho}$. This functional calculus is described in~\cite{Kuster:17}. We state here a particular case of~\cite[Theorem~4.4]{Kuster:17}.

\begin{theorem}[Theorem~4.4 in~\cite{Kuster:17}]
\label{t:thm-Kuster}
Let $0\leq \rho <\frac12$ and take $f_h \in C^\infty_c(\R)$ such that $|\d^\alpha f_h(s)|\leq C_{\alpha} h^{-\alpha \rho}$ and $\supp(f_h) \subset K$ where $K\subset \R$ is a compact set independent of $h$. Then, there exist $e_j\in S^{-\infty}_\delta(T^*\R^n)$ for $j \in J$ and $R_h \in \L(L^2(M))$ such that 
\begin{enumerate}
\item $f_h(-h^2\Delta) = \sum_{j\in J} \check{\psi}_j \phi_j^\ast  \Op_h^\ell(  e_j ) \big(\phi_j^{-1}\big)^\ast \check{\psi}_j + R_h$;
\item $R_h = \O_{\L(L^2(M))}(h^\infty)$;
\item \label{i:asympt} for all $N\in \N$, $e_j - \sum_{k=0}^N h^{k(1-2\delta)} e_{j,k} \in h^{(N+1)(1-2\delta)}  S_{\delta}^{-\infty}(T^*\R^n)$ with $e_{j,k} \in S_{\delta}^{-\infty}(T^*\R^n)$ such that 
\item $e_{j,0} = \big(\phi_j^{-1}\big)^\ast \big(f_h(|\xi|_x^2) \psi_j\big)$ for all $j \in J$,
\item $\supp(e_{j,k}) \subset \supp \Big( \big(\phi_j^{-1}\big)^\ast \big(f_h(|\xi|_x^2) \psi_j\big) \Big)$ for all $j \in J$.
\end{enumerate}
\end{theorem}
This result is stated with the Weyl quantization in~\cite[Theorem~4.4]{Kuster:17} but can be equivalently rewritten with the classical quantization according \eg to~\cite[Theorem~4.13]{Zworski:book}.
In particular, we mostly use this result under the following weaker corollary, where we recall that $\Lambda = \sqrt{-\Delta}$.
\begin{corollary}
\label{c:functional-calc}
Let $0\leq \rho <\frac12$, $\chi \in C^\infty_c(\R)$ and set $\chi_h(s) := \chi(\frac{s-1}{h^\rho})$. Then, modulo $\O_{\L(L^2(M))}(h^\infty)$, we have $\chi_h(h\Lambda) \in \Psi_\rho^{-\infty}(M)$ with principal symbol $\chi_h(|\xi|_x) \in S_\rho^{-\infty}(T^*M)$. 

Moreover, let $\tilde{\chi} \in C^\infty_c(\R)$ be equal to one in a neighborhood of $\supp(\chi)$ and $\tilde{\chi}_h(s) = \tilde{\chi}(\frac{s-1}{h^\rho})$. Then, for any $B \in \Psi_\rho^{0}(M)$, we have $(1-\tilde{\chi}_h(h\Lambda)) B \chi_h(h\Lambda) = \O_{\L(L^2(M))}(h^\infty)$.
\end{corollary}
\bnp
We write $\supp(\chi) \subset (-\alpha,\alpha)$ and assume $h<\alpha^{-1}$. Then setting $f_h(s) = \chi_h(\sqrt{s}) = \chi(\frac{\sqrt{s}-1}{h^\rho})$ implies $f_h \in C^\infty(\R)$, $\supp(f_h) \subset [0,2]$ and $|\d^\alpha f_h(s)|\leq C_{\alpha} h^{-\alpha \rho}$. Hence Theorem~\ref{t:thm-Kuster} applies and the first part of the corollary follows. 

Concerning the second part, we write $a(x,\xi) = \chi_h(|\xi|_x)$, $\tilde{a}(x,\xi) = \tilde{\chi}_h(|\xi|_x)$, and notice that, modulo $\O_{\L(L^2(M))}(h^\infty)$, we have  $\chi_h(h\Lambda) = \sum_{j\in J} \check{\psi}_j \phi_j^\ast  \Op_h^\ell(e_j )\big(\phi_j^{-1}\big)^\ast \check{\psi}_j$ with $\supp(e_j) \subset \supp \Big( \big(\phi_j^{-1}\big)^\ast \big(a\psi_j\big) \Big)$ and similarly $\tilde{\chi}_h(h\Lambda) = \sum_{j\in J} \check{\psi}_j \phi_j^\ast  \Op_h^\ell(\tilde{e}_j ) \big(\phi_j^{-1}\big)^\ast \check{\psi}_j $ with $\supp(\tilde{e}_j) \subset  \supp \Big( \big(\phi_j^{-1}\big)^\ast \big(\tilde{a}\psi_j\big) \Big)$. Since $\supp a \cap \supp (1-\tilde{a})= \emptyset$, we deduce that $\supp\big(\big(\phi_i^{-1}\big)^\ast\phi_j^\ast e_j \big) \cap \supp(1-\tilde{e}_i) = \emptyset$ for all $i,j \in J$. Thus the second part of the result follows from Item~\ref{i:asympt} in Theorem~\ref{t:thm-Kuster}, the change of variable formula (see~\cite[Theorem~18.1.17]{Hoermander:V3} or~\cite[Theorem~B.1]{LRLR:13} in the semiclassical setting) together with the $\Psi_\rho^{-\infty}(\R^n)$ calculus.
\enp

\subsection{Coherent states}

\begin{lemma}
\label{l:coherent}
Let $(x_0,\xi_0) \in T^*\R^n$ and set 
\begin{equation}
\label{e:def-coherent-state}
v_h(x) = (\pi h)^{-\frac{n}{4}}e^{-\frac{1}{2h} |x-x_0|^2 + \frac{i}{h}(x-x_0) \cdot \xi_0}.
\end{equation} Then we have $\|v_h\|_{L^2(\R^n)}=1$ and for all $m \in \R$ there is $C>0$ such that for all $a \in S^m_\rho(T^*\R^n)$, we have
\begin{align*}
\left| \left(\Op_h^\ell(a)  v_h , v_h \right)_{L^2} - a(x_0 , \xi_0) \right| \leq Ch^{\frac12-\rho}.
\end{align*}
\end{lemma}
Note that $a \in S^m_\rho(T^*\R^n)$ (and hence in particular $a(x_0 , \xi_0)$) depends on $h$.

\begin{corollary}
\label{c:coherent-state}
Given $\zeta_0 \in T^*M$, there is an $h$-dependent family of functions $u_h^{\zeta_0} \in C^\infty(M)$ such that $\|u_h^{\zeta_0}\|_{L^2(M)}=1$ and for all $A \in \Psi^m_\rho(M)$, we have for all $h \in (0,h_0)$
\begin{align*}
\left| \left(A  u_h^{\zeta_0} , u_h^{\zeta_0} \right)_{L^2} - \sigma_h(A)(\zeta_0) \right| \leq Ch^{\frac12-\rho}.
\end{align*}
\end{corollary}

\bnp[Proof of Corollary~\ref{c:coherent-state}]
Denote by $(y_0,\eta_0) = \zeta_0\in T^*M$. Take first $(U,\phi)$ a coordinate patch of $M$ such that $y_0 \in U$. Write $(x_0, \xi_0) := ( \phi(y_0),\transp\big(\phi'(y_0)\big)^{-1} \eta_0)\in T^*\R^n$, and take $\psi \in C^\infty_c(\phi (U))$ such that $\psi = 1$ in a neighborhood of $x_0$ and $\chi \in C^\infty(U)$ such that $\chi=1$ in a neighborhood of $\supp (\phi^*\psi)$

Now, we denote by $v_h$ the function defined in~\eqref{e:def-coherent-state} associated to $(x_0,\xi_0)$, and  
$v_h^{\zeta_0} := \phi^* \psi v_h \in C^\infty(M)$. 
Given $A\in \Psi^m_\rho(M)$, and denoting by $|g|$ the determinant of the Jacobian, we have
\begin{align*}
\left(Av_h^{\zeta_0} , v_h^{\zeta_0} \right)_{L^2(M)} & =  \left(A \phi^* \psi v_h ,  \phi^* \psi v_h \right)_{L^2(M)} 
=  \left( \chi A \chi \phi^* \psi v_h ,  \phi^* \psi v_h \right)_{L^2(M)} \\
& =  \int_{\R^n} \Op_h^\ell(\tilde{a}) (\psi v_h) \ovl{\psi v_h} \sqrt{|g|} dx ,
\end{align*}
where $\Op_h^\ell(\tilde{a}) = \big(\phi^{-1}\big)^* \chi A \chi \phi^*$ is such that $\tilde{a}\in S^m_\rho (T^*\R^n)$ equals $\big(\phi^{-1}\big)^*\big(\chi^2\sigma_h(A)\big)$ modulo $h^{1-2\rho}S^{m-1}_\rho (T^*\R^n)$, according to Definition~\ref{def:pseudo} and the properties of $\sigma_h$. Now, pseudodifferential calculus in $\R^n$ gives
$\psi  \sqrt{|g|} \Op_h^\ell(\tilde{a}) \psi =  \Op_h^\ell(b)$ where $b\in S^m_\rho (T^*\R^n)$, $b = \sqrt{|g|} \psi^2 \tilde{a}= \sqrt{|g|} \psi^2\big(\phi^{-1}\big)^*\sigma_h(A)$ modulo $h^{1-2\rho}S^{m-1}_\rho (T^*\R^n)$.
Using Lemma~\ref{l:coherent}, we thus obtain that 
\begin{align*}
 \left(A v_h^{\zeta_0} , v_h^{\zeta_0} \right)_{L^2(M)} &  = 
\left(\Op_h^\ell(b) v_h,v_h \right)_{L^2(\R^n)} = b(x_0, \xi_0) + O(h^{\frac12-\rho}) \\
& = \sqrt{|g|}(x_0) \big(\phi^{-1}\big)^*\sigma_h(A)(x_0, \xi_0) + O(h^{\frac12-\rho}) .
\end{align*}
In particular, if $A=\id$, this reads $\nor{v_h^{\zeta_0}}{L^2(M)}^2 = \sqrt{|g|}(x_0)+ O(h^{\frac12-\rho})$.

Hence, finally setting $u_h^{\zeta_0} :=\nor{v_h^{\zeta_0}}{L^2(M)}^{-1} v_h^{\zeta_0}$, we have $\nor{u_h^{\zeta_0}}{L^2(M)}=1$ and for all $A\in \Psi^m_\rho(M)$, 
\begin{align*}
 \left(A u_h^{\zeta_0} , u_h^{\zeta_0} \right)_{L^2(M)}  = \nor{v_h^{\zeta_0}}{L^2(M)}^{-2} \left(A v_h^{\zeta_0} , v_h^{\zeta_0} \right)_{L^2(M)} 
 = \big(\phi^{-1}\big)^*\sigma_h(A)(x_0, \xi_0) + O(h^{\frac12-\rho})  = \sigma_h(A)(\zeta_0) +O(h^{\frac12-\rho}) ,
\end{align*}
which concludes the proof of the corollary.
\enp

\bnp[Proof of Lemma~\ref{l:coherent}]
According to~\cite[p 103]{Zworski:book}, we have
\begin{align}
\label{e:op-gauss}
\left(\Op_h^\ell(a)  v_h , v_h \right)_{L^2} = \frac{2^{\frac{n}{2}}}{(2\pi h)^n} \int_{\R^n} \int_{\R^n}a(x,\xi) e^{\frac{i}{h}(x-x_0)\cdot(\xi-\xi_0)}e^{-\frac{1}{2h}( |x-x_0|^2 + |\xi-\xi_0|^2)} dx d\xi .
\end{align}
We write 
$$
a(x,\xi) - a(x_0, \xi_0) =   \left[\int_0^1 a'\big( t(x,\xi) + (1-t)(x_0,\xi_0) \big) dt \right] \cdot \left( 
\begin{array}{c}
x-x_0 \\
\xi-\xi_0
\end{array}
\right) ,
$$
which, using $a \in S^m_\rho(T^*\R^n)$ and hence~\eqref{e:symb-estim-Srho}, yields $|a(x,\xi) - a(x_0, \xi_0)| \leq C h^{-\rho}\langle \xi\rangle^{m-1} \left(|x-x_0| +|\xi - \xi_0| \right)$. Coming back to~\eqref{e:op-gauss} and using that $\frac{2^{\frac{n}{2}}}{(2\pi h)^n} \int_{\R^n} \int_{\R^n}e^{\frac{i}{h}(x-x_0)\cdot(\xi-\xi_0)}e^{-\frac{1}{2h}( |x-x_0|^2 + |\xi-\xi_0|^2)} dx d\xi = 1$, we now have
\begin{align}
\label{e:-coherent-mean-value}
\left| \left(\Op_h^\ell(a)  v_h , v_h \right)_{L^2} - a(x_0 , \xi_0) \right| & \leq   \frac{Ch^{-\rho}}{h^n} \int_{\R^n} \int_{\R^n} \langle \xi\rangle^{m-1} \left(|x-x_0| +|\xi - \xi_0| \right) e^{-\frac{1}{2h}( |x-x_0|^2 + |\xi-\xi_0|^2)} dx d\xi \nonumber \\
& \leq  Ch^{-\rho} h^{1/2} ,
\end{align}
after a change of variables.
This concludes the proof of the lemma.
\enp

\section{Elementary intermediate lemmata}
\label{appendix-elementary}
\subsection{Semigroup estimates}
\begin{lemma}
\label{l:bound-etip}
Let $N \in \N$ and $\nu <\frac12$. Assume that $\P_Q = P + i Q$ where $P$ is a selfadjoint operator on $L^2(M;\C^N)$ with dense domain, and $Q \in \Psi^{0}_\nu (M; \C^{N\times N})$ with principal symbol $q \in S^{0}_\nu(T^*M; \C^{N\times N})$ such that 
\begin{align}
\label{e:pos-q}
\Re(q) = \frac12 \left( q + q^*\right) \quad \text{ is a nonnegative (hermitian) matrix on }T^*M.
\end{align} 
Then, there are $C_0,h_0>0$ such that $i\P_Q$ generates a strongly continuous semigroup $(e^{i t \P_Q} )_{t\in \R_+}$ with 
\begin{align}
\label{e:estim-pq-sem}
\|e^{\frac{i t}{h}\P_Q} u_0\|_{L^2(M;\C^N)}\leq e^{C_0h^{1-2\nu} t}\|u_0\|_{L^2(M;\C^N)}, \quad \text{ for all } u_0\in L^2(M;\C^N), \quad h \in (0,h_0) , \quad t \geq 0 .
\end{align}
If we assume further $\Re(Q) = \frac12 (Q+ Q^*) \geq 0$ in the sense of operators (which implies~\eqref{e:pos-q}), then $(e^{i t\P_Q} )_{t\in \R_+}$ is a contraction semigroup, \ie~\eqref{e:estim-pq-sem} holds true with $C_0=0$.
\end{lemma}
\bnp
That  $\P_Q$ generates a strongly continuous group of evolution follows from selfadjointness of $P$ together with boundedness of $Q \in \Psi^{0}_\nu (M; \C^{N\times N})$.
For $u \in D(\P_Q) = D(P) \subset L^2(M;\C^N)$, we have  
\begin{align*}
\Re \left( i \P_Q u, u \right)_{L^2(M;\C^N)}= \left( - Q u, u\right)_{L^2(M;\C^N)} +  \left(u, - Q u \right)_{L^2(M;\C^N)} 
 = - \left(\Re(Q) u , u \right)_{L^2(M;\C^N)} ,
\end{align*}
with $\Re(Q)=\frac12(Q+Q^*)$ having principal symbol $\Re(q) \geq 0$. The sharp G{\aa}rding inequality of Theorem~\ref{t:garding}  implies the existence of $C_0>0$ such that 
$$
\left(\Re(Q) u, u\right)_{L^2(M;\C^N)} \geq -C_0 h^{1-2\nu} \|u\|_{L^2(M;\C^N)}^2 ,
$$
As a consequence, $i \P_Q + C_0 h^{1-2\nu}$, is accretive (hence maximal accretive), and thus 
generates a contraction semigroup. This implies the sought result for $i\P_Q$.
\enp

We use several times the following lemma.
\begin{lemma}
\label{l:1-BB}
Assume $i\P$ generates a strongly continuous semigroup $(e^{it\P} )_{t\in \R_+}$ on $\H$, with $\|e^{it\P}\|_{\L(\H)} \leq K_0 e^{\gamma t}$ for some $\gamma, K_0\geq 0$ and for all $t\geq 0$. Assume that $A \in \L(\H)$ is such that $[\P,A]\in  \L(\H)$. Then, we have $[e^{it\P},A] \in \L(\H)$ for all $t\geq 0$, together with
\begin{align}
\nor{[e^{it\P},A]}{\L(\H)} = \nor{e^{it\P}A  - Ae^{it\P}}{\L(\H)} \leq K_0^2 t e^{\gamma t} \nor{[\P,A]}{\L(\H)}, \quad \text{for all } t\geq0 . 
\end{align}
\end{lemma}
\bnp
Let $u \in D(\P)$ and denote $w(t) =A e^{it\P} u$. We have  
\begin{align*}
\d_t w(t) &= A i\P e^{it\P} u  =  i\P  A e^{it\P} u + [A , i\P ] e^{it\P} u =  i\P  w(t)+ [A , i\P ] e^{it\P} u .
\end{align*}
Remarking  that $w(0) = A u$, the Duhamel formula yields
$$
w(t) = e^{it\P} A u  + \int_0^t e^{i(t-s)\P}  [A , i\P ] e^{is \P} u\ ds .
$$
This implies 
\begin{align*}
 \nor{e^{it\P}A  - Ae^{it\P}}{\L(\H)} & \leq \int_0^t \nor{e^{i(t-s)\P}}{\L(\H)}  \nor{[A , i\P ]}{\L(\H)} \nor{e^{is \P}}{\L(\H)} \nor{u}{\H} ds \\
 & \leq \int_0^t K_0 e^{(t-s)\gamma} \nor{[A , \P ]}{\L(\H)} K_0 e^{s\gamma}  ds  \nor{u}{\H}  = K_0^2 t e^{t\gamma} \nor{[A , \P ]}{\L(\H)} \nor{u}{\H} .
\end{align*}
This concludes the proof for $u \in D(\P)$ and by a density argument, that of the lemma.
\enp

\begin{lemma}[Pertubations of the damping]
\label{l:perturbation}
On a Hilbert space $\H$, let $\P_{Q_j} = P_h  + i Q_j$ with $P$ selfadjoint with dense domain, and $Q_j$ bounded with $\Re(Q_j) = \frac12 (Q_j+ Q_j^*) \geq 0$ (in the sense of operators) for $j=1,2$. Then, we have 
\begin{align}
\label{e:erreur-2}
\nor{ (e^{it\P_{Q_2}} - e^{it\P_{Q_1}} ) u }{\H} \leq t \nor{Q_2-Q_1}{\L(\H)} \nor{u}{\H} , \quad \text{for all}\quad t \geq 0 ,\quad  u \in \H . 
\end{align}
\end{lemma}
\bnp
We consider $r(t) = (e^{it\P_{Q_2}} - e^{it\P_{Q_1}} ) u$ and notice that we have $r(0) = 0$. Moreover, assuming first that $u \in D(P)= D(\P_{Q_j})$, we have
\begin{align*}
\d_t r(t) & =i \left(\P_{Q_2} e^{it\P_{Q_2}} u - \P_{Q_1}e^{it\P_{Q_1}} u \right) 
=i \P_{Q_2} \left( e^{it\P_{Q_2}} u - e^{it\P_{Q_1}} u \right) +i \left(\P_{Q_2} -\P_{Q_1}  \right)e^{it\P_{Q_1}} u \\
& =i \P_{Q_2}r(t) - \left(Q_2 - Q_1  \right)e^{it\P_{Q_1}} u  .
\end{align*}
The Duhamel Formula then yields $r(t) = \int_0^t e^{i(t-s)\P_{Q_2}} \left(Q_1 - Q_2 \right)e^{is\P_{Q_1}} u \, ds$.
Using that $\P_{Q_j}$ generate contraction semigroups, we obtain
\begin{align*}
\nor{r(t)}{\H} \leq \int_0^t \nor{e^{i(t-s)\P_{Q_2}}}{\L(\H)} \nor{Q_1 - Q_2 }{\L(\H)} \nor{e^{is\P_{Q_1}}}{\L(\H)}  \nor{u}{\H} ds 
\leq   \int_0^t \nor{Q_1 - Q_2 }{\L(\H)} \nor{u}{\H} ds ,
\end{align*}
which proves~\eqref{e:erreur-2}, and concludes the proof of the lemma.
\enp

\subsection{Time-averages and pointwise-in-time estimates}
The following elementary Gronwall-type lemma compares averaged quantities with pointwise quantities.
\begin{lemma}
\label{l:debile}
Let $\Psi, \psi \in L^2_{\loc}(\R_+;\R)$ and $E \in W^{1,1}_{\loc}(\R_+;\R)$ such that $E' \leq 0$ \ae on $\R_+$. Then we have 
\begin{align}
\label{e:debile-1}
\left(\int_0^\tau\Psi^2(t)dt\right)E(\tau) &\leq \int_0^\tau \Psi^2(t) E(t)dt , \quad \text{ for all } \tau\geq 0 ,\\
\label{e:debile-2}
\nor{\psi}{L^2(0,\theta)}^2 E(\theta T) & \leq \frac{1}{T}\int_0^{\theta T} \psi^2\left(\frac{t}{T}\right) E(t) dt ,  \quad \text{ for all } T ,\theta>0  .
\end{align}
\end{lemma}
\bnp
Writing $\varphi(t) := \int_0^t\Psi^2(s)ds \in W^{1,1}_{\loc}(\R_+;\R)$, we have $\big(\varphi E \big)' = \varphi' E + \varphi E' \leq \varphi' E =  \Psi^2 E$, since $E' \leq 0$ and $\varphi \geq 0$. Integrating this inequality on $(0,\tau)$ proves~\eqref{e:debile-1}. Finally,~\eqref{e:debile-2} follows from~\eqref{e:debile-1} applied to $\tau=\theta T$ and $\Psi (t) = \psi(t/T)$.
\enp

The following elementary lemma is then used to optimize the choice of the weight function $\psi$ in order to pass from time-averages to pointwise-in-time bounds in Theorem~\ref{t:resolvent-implies-long-time}.

  \begin{lemma}
  \label{l:optimizing-psi}
We have
  \begin{align*}
\mathcal{I} & : = \inf \left\{\nor{\psi'}{L^2(\R)}^2 , \psi \in H^1_{\comp}(\R) , \nor{\mathds{1}_{[0,1]}\psi}{L^2(\R)} = 1 , \supp \psi \subset [0,+\infty) \right\} \\
& = \inf \left\{\frac{\nor{\psi'}{L^2(\R)}^2}{\nor{\mathds{1}_{[0,1]}\psi}{L^2(\R)}^2} , \psi \in C^\infty_c(\R) , \supp \psi \subset [0,+\infty) \right\} \\
& = 2 .
\end{align*}
Moreover, the infimum is not reached but, for any $L>1$, there is $\psi_{\min,L} \in H^1_{\comp}(\R)$ with $\supp \psi_{\min,L}  =[0,L]$ and such that
\begin{align*}
\nor{\psi_{\min,L}}{L^2(0,1)} = 1 , \quad 
 \nor{\psi_{\min,L}'}{L^2(\R)}^2 = 2 \left( 1+ \frac{L}{(L-1)^2}\right), \\
\nor{\psi_{\min,L}}{L^1(\R)} = \frac{L}{\sqrt{2}} , \quad \nor{\psi_{\min,L}'}{L^1(\R)} = \sqrt{2} \left(1+ \frac{L}{L-1}\right).
\end{align*}
  \end{lemma}
  \bnp
The first equality comes from homogeneity and density of $C^\infty_c(\R)$ in $H^1_{\comp}(\R)$. Therefore, we only need to prove that $\mathcal{I}=2$.
We first remark that if $\psi \in H^1_{\comp}(\R)$ and $\supp \psi \subset [0,+\infty)$, then $\psi$ is continuous and $\psi (0)=0$. As a consequence, the Cauchy-Schwarz inequality yields
\begin{align*}
\int_0^1|\psi(x)|^2 dx& = \int_0^1\left| \int_0^x\psi'(t)dt \right|^2 dx \leq  \int_0^1x \left(\int_0^x\left| \psi'(t) \right|^2dt\right) dx \\
&\leq \left( \int_0^1x  dx\right) \left(\int_0^1\left| \psi'(t) \right|^2dt\right) = \frac12 \int_0^1\left| \psi'(t) \right|^2dt .
\end{align*}
This implies $\mathcal{I}\geq 2$.
Second, to prove that $2$ is actually the infimum, we rewrite 
\begin{align*}
\mathcal{I} = \lim_{L \to +\infty} \inf \left\{\nor{\psi'}{L^2(\R)}^2 , \psi \in \mathcal{J}_L\right\} 
\quad \mathcal{J}_L :=\left\{ \psi \in H^1(\R) , \nor{\mathds{1}_{[0,1]}\psi}{L^2(\R)} = 1 , \supp \psi \subset [0,L] \right\} .
\end{align*}
Given $L>1$ fixed, the first order optimality condition for this minimization problem writes, 
$$
-\psi'' = \lambda \mathds{1}_{[0,1]}\psi , \quad \text{ on }[0,L] , \quad \psi(0)= 0 , \quad \psi(L)=0.
$$
for some $\lambda \in \R$. In $\mathcal{J}_L$, the minimization problem thus has $\psi_{\min,L}(t) := \frac{\psi_L}{\nor{\psi_L}{L^2(0,1)}} = \sqrt{2} \psi_L$, with $\psi_L(t) = \mathds{1}_{[0,1]}(t) t + \mathds{1}_{(1,L]}(t)\frac{L-t}{L-1}$ as a minimizer (unique modulo sign, and obtained for $\lambda =0$), with minimal energy $\nor{\psi_{\min,L}'}{L^2(\R)}^2 = 2 \left( 1+ \frac{L}{(L-1)^2}\right)$.
Letting $L$ go to $+\infty$, we conclude that the infimum is $\mathcal{I} = 2$ and is not reached in $H^1_{\comp}(\R)$. 

Note that $\nor{\psi_L}{L^1(\R)} = \frac{L}{2}$ and $\psi_L'(t)=\mathds{1}_{[0,1]}(t) -\frac{1}{L-1} \mathds{1}_{(1,L]}(t)$ so that $\nor{\psi_L'}{L^1(\R)} = 1+ \frac{L}{L-1}$. This conclude the proof recalling that 
 $\psi_{\min,L} = \sqrt{2} \psi_L$.
\enp

\small
\bibliographystyle{alpha}
\bibliography{bibli}
\end{document}